\documentclass{amsart}
\usepackage[dvipdfmx]{graphicx}
\usepackage{tikz}

\newtheorem{theorem}{Theorem}[section]
\newtheorem{lemma}[theorem]{Lemma}
\newtheorem{proposition}[theorem]{Proposition}
\newtheorem{corollary}[theorem]{Corollary}

\theoremstyle{definition}
\newtheorem{definition}[theorem]{Definition}
\newtheorem{example}[theorem]{Example}

\theoremstyle{remark}
\newtheorem{remark}[theorem]{Remark}

\numberwithin{equation}{section}

\begin{document}

\title[Solutions with time-dependent singular sets]
{Solutions with time-dependent singular sets\\
for the heat equation with absorption}

\author[J. Takahashi]{Jin Takahashi}
\address{Mathematical Institute, Tohoku University, 
6-3 Aramaki Aza, Aoba-ku, Sendai-shi, Miyagi 980-8578, Japan}
\email{jin.takahashi.b5@tohoku.ac.jp}
\thanks{The first author was supported by JSPS KAKENHI Grant Number 17J00693.}

\author[H. Yamamoto]{Hikaru Yamamoto}
\address{Department of Mathematics, Faculty of Science, 
Tokyo University of Science, 1-3 Kagurazaka, Shinjuku-ku, Tokyo 162-8601, Japan}
\email{hyamamoto@rs.tus.ac.jp}
\thanks{The second author was supported by JSPS KAKENHI Grant Number 16H07229.}



\begin{abstract}
We consider the heat equation with a superlinear absorption term 
$\partial_t u-\Delta u= -u^p$ in $\mathbb{R}^n$ and 
study the existence and nonexistence of nonnegative solutions with 
an $m$-dimensional time-dependent singular set, where $n-m\geq3$. 
First, we prove that if $p\geq (n-m)/(n-m-2)$, 
then there is no singular solution. 
We next prove that, if $1<p<(n-m)/(n-m-2)$, then there are two types of singular solution.  
Moreover, we show the uniqueness of the solutions 
and specify the exact behavior of the solutions near the singular set. 
\end{abstract}

\maketitle

\tableofcontents

\section{Introduction}
Let $\Omega$ be a domain in $\mathbb{R}^n$ and $I$ be an open interval. 
We consider nonnegative solutions of the heat equation with an absorption term 
\begin{equation}\label{eq:parabg2}
	\partial_t u-\Delta u= -|u|^{p-1}u \quad\mbox{ in }
	(\Omega\times I) \setminus\tilde{M} 
\end{equation}
in the case $\tilde{M}=\cup_{t\in I}(M_{t}\times\{t\})$. 
Here $p>1$ and $\{M_{t}\}_{t\in I}$ is a one parameter family of compact submanifolds 
in $\Omega$ with dimension $m\geq 1$ and codimension $n-m\geq 3$. 
Our purpose of this paper is to study the existence and nonexistence of 
singular solutions of \eqref{eq:parabg2}, that is, solutions 
which tend to $+\infty$ along $M_t$ 
for each $t\in I$. 

In the case where $\tilde M$ is a relatively closed subset of $\Omega\times I$, 
there are some results on the nonexistence of singular solutions of \eqref{eq:parabg2}. 
For instance, Baras and Pierre \cite{BP84m} gave a necessary and sufficient 
condition on the removability of singularities 
for all solutions of \eqref{eq:parabg2} 
by using a parabolic capacity. 
Recently, Hirata \cite{Hi16} employed 
the Minkowski content of $\tilde M$ with respect to the parabolic distance, 
and gave removability results for each of the nonnegative solutions of \eqref{eq:parabg2}. 
On the existence of singular solutions, more recently, 
the first author and Yanagida \cite{TY16} studied 
singular solutions in the specific case where each $M_{t}$ is a time-dependent point. 
The novelty of their results is 
the existence of singular solutions for 
\[
	p<p_{sg}:=\frac{n}{n-2}.
\]
For $p\geq p_{sg}$, as a consequence of \cite[Th\'eor\`em 4.1]{BP84m} or 
\cite[Corollary 2.8]{Hi16}, 
we can see that there is no nonnegative singular solution.

Before stating our main results, 
we summarize the results for the elliptic problem corresponding to \eqref{eq:parabg2}. 
For several decades, 
singular solutions of the equation 
\begin{equation}\label{eq:ellip2}
	-\Delta u= -|u|^{p-1}u \quad\mbox{ in }\tilde\Omega \setminus M  
\end{equation}
have been studied in many papers, see 
V\'eron \cite{Vbook}, Marcus and V\'eron \cite{MVbook} 
and references therein. 
Here $n\geq3$, 
$\tilde{\Omega}\subset\mathbb{R}^{n}$ 
is a domain and 
$M\subset\tilde{\Omega}$ is a prescribed singular set. 
First, we consider the case $M=\{0\}$, the origin. 
Br\'ezis and V\'eron \cite{BV80} showed that if $p\geq p_{sg}$, 
then the singularity at $x=0$ is removable. 
For $p<p_{sg}$, V\'eron \cite{Ve81s} gave the complete classification 
of the singularities as follows. 
Let $u\in C^2(\tilde\Omega\setminus\{0\})$ be a nonnegative solution of \eqref{eq:ellip2}. 
Then, one of the following holds. 
\begin{itemize}
\item[(a)]
$\displaystyle \lim_{x\to0}|x|^{n-2}u(x)=c$. 
\item[(b)]
$\displaystyle \lim_{x\to0}|x|^\frac{2}{p-1}u(x)=L_0$. 
\item[(c)]
The singularity at $x=0$ is removable. 
\end{itemize}
Here $c$ can take any positive constant and 
\[
	L_0=L_0(n,p):=\left(
	\frac{2}{p-1} \left( \frac{2}{p-1} -(n-2) \right)
	\right)^\frac{1}{p-1}. 
\]
The singularities of types (a) and (b) are called \emph{weak singularities} 
and \emph{strong singularities}, respectively. 
He also proved that each types of singularity truly exists. 
For more general elliptic equations, 
see Aviles \cite{Av82} and Vazquez and V\'eron \cite{VV85}. 
We refer Chen, Matano and V\'eron \cite{CMV86,CMV89} and Matano \cite{Ma90} for 
sign-changing solutions. 

In view of integrability, 
solutions with a weak singularity belong to $L^p_\mathrm{loc}(\tilde\Omega)$
and solutions with a strong singularity do not so. 
Based on this observation, 
Br\'ezis and Oswald \cite{BO87} showed that 
each of the nonnegative solutions with a weak singularity can be extended as a solution of 
$-\Delta u=-u^p+c\delta_0$ in $\mathcal{D}'(\tilde \Omega)$, 
where $\delta_0$ is the Dirac distribution concentrated at $x=0$. 
In contrast, solutions with a strong singularity do not define a distribution. 
They also showed that solutions with a weak singularity 
converge to a solution with a strong singularity as $c\uparrow+\infty$. 

Next, we consider the case where the singular set 
$M$ of \eqref{eq:ellip2} is a compact $m$-dimensional submanifold in $\tilde{\Omega}$. 
V\'eron \cite{Ve81e} showed that if $n-m\geq3$ and 
\begin{equation}\label{eq:pstar}
	p\geq p_*:=\frac{n-m}{n-m-2}, 
\end{equation}
then $M$ is removable. 
The optimality of this result was also shown. 
Actually, for $p<p_*$, 
he pointed out that 
adding the extra variable $x''\in\mathbb{R}^m$ to a singular solution of 
$-\Delta v=-v^p$ in $x'\in \mathbb{R}^{n-m}\setminus\{0\}$ 
provides a solution $u(x',x''):=v(x')$, which is singular on $\{0\}\times \mathbb{R}^m$. 
Baras and Pierre \cite{BP84e} characterized 
the condition on removability by employing 
an elliptic capacity. 

In \cite{Gr96,Gr97}, Grillot showed 
the removability for more general equations including \eqref{eq:ellip2} with $p\geq p_*$, 
and showed the existence and uniqueness 
of singular solutions of \eqref{eq:ellip2} with $p<p_*$. 
On existence, 
she proved that 
there exist a solution with a strong singularity and 
a family of solutions with a weak singularity 
parametrized by $c>0$. 
In this context, weak and strong singularities are the singularities 
such that the conditions 
\begin{itemize}
\item[(A)]
$\displaystyle \lim_{x\to M} d(x,M)^{n-m-2}u(x)=  c$, and 
\item[(B)]
$\displaystyle \lim_{x\to M} d(x,M)^{\frac{2}{p-1}}u(x)= L$ 
\end{itemize}
hold, respectively. 
Here $d(x,M):=\inf_{y\in M}|x-y|$ and $L$ is given by 
\begin{equation}\label{eq:Ldefi}
	L=L(n,m,p):=\left(
	\frac{2}{p-1} \left( \frac{2}{p-1} -(n-m-2) \right)
	\right)^\frac{1}{p-1}. 
\end{equation} 
She also proved that the solutions with a weak singularity 
converges to the solution with a strong singularity as $c\uparrow+\infty$. 
We note that, in \cite{Gr97}, the above results were stated 
in the framework that $\tilde \Omega$ is a complete connected Riemannian manifold, and 
the case where $n-m=2$ and $p>1$ was also studied. 
As far as the authors know, the classification of singularities 
of \eqref{eq:ellip2} is still incomplete. 
However, any singular solution must lie between weak singularities and strong singularities. 
Explicitly, there is no singular solution satisfying 
\begin{itemize}
\item[(C)]
$\displaystyle \lim_{x\to M} d(x,M)^{n-m-2}u(x)=0$, or
\item[(D)]
$\displaystyle \limsup_{x\to M} d(x,M)^{\frac{2}{p-1}}u(x)= +\infty$. 
\end{itemize}
The nonexistence for (C) is due to Hirata and Ono \cite{HO14}, 
and the one for 
(D) is due to Grillot \cite[Lemma 5]{Gr97} (see also V\'eron \cite[LEMME 1]{Ve81e}).

Let us return to the parabolic problem \eqref{eq:parabg2}, which we treat in this paper. 
As parabolic analogs of the elliptic results introduced above, 
the first author and Yanagida \cite{TY16} considered 
the case $m=0$, that is, 
the case where the singular set is a time-dependent point $M_{t}=\{\xi(t)\}$. 
For $p<p_{sg}$, their results imply that 
there are singular solutions satisfying (A') and (B'), and that 
there is no singular solution satisfying (C') or (D'), 
where these conditions are parabolic analogs of (A), (B), (C) and (D) introduced later.
They also gave a classification result under some assumption on the leading term 
of the asymptotic behavior of solutions near the singular point. 
We point out that the existence of singular solutions of \eqref{eq:parabg2} 
is still an open problem 
if $(\Omega\times I)\setminus \tilde{M}$ 
is a non-cylindrical domain in $\mathbb{R}^n\times\mathbb{R}$ and 
the time slice of $\tilde M$ at $t$ is not a point. 
Therefore, it is natural to study 
the existence of singular solutions of \eqref{eq:parabg2} 
with $\tilde{M}=\cup_{t\in I}(M_{t}\times\{t\})$. 
That is the aim of this paper. 

We shall give an existence result 
and a nonexistence result 
for \eqref{eq:parabg2} with $\tilde{M}=\cup_{t\in I}(M_{t}\times\{t\})$
under appropriate assumptions on $M_t$. 
Actually, we assume $n-m\geq3$ and 
show that there is no nonnegative singular solution 
of \eqref{eq:parabg2} with $p\geq p_*$ (see Theorem \ref{th:nex}). 
For $p<p_*$, we also show that 
there are singular 
solutions satisfying 
\begin{itemize}
\item[(A')]
$\displaystyle \lim_{x\to M_t} d(x,M_t)^{n-m-2}u(x,t)=c$ uniformly for $t\in I$, and
\item[(B')]
$\displaystyle \lim_{x\to M_t} d(x,M_t)^{\frac{2}{p-1}}u(x,t)= L$ uniformly for $t\in I$, 
\end{itemize}
respectively, where $c$ can take any positive constant. 
Moreover, we show the uniqueness of the solutions in some class. 
Furthermore, we see that 
any singular solution must lie between weak singularities and strong singularities. 
This means that there is no singular solution satisfying 
\begin{itemize}
\item[(C')]
$\displaystyle \lim_{x\to M_t} d(x,M_t)^{n-m-2}u(x,t)=  0$ 
locally uniformly for $t\in I$, or
\item[(D')]
$\displaystyle \liminf_{x\to M_t} d(x,M_t)^{\frac{2}{p-1}}u(x,t)= +\infty$ 
locally uniformly for $t\in I$.  
\end{itemize}
The nonexistence for (C') is due to Proposition \ref{pro:nexsub} below, and the one for 
(D') follows from a universal estimate on a general domain due to 
the first author and Yanagida \cite[Proposition 2.1 (iii)]{TY16}. 
See Remark \ref{rem:codim2} for the case $n-m=2$, 
and Remark \ref{rem:integ} for the parabolic analog of 
the result of Br\'ezis and Oswald \cite{BO87} mentioned above.

Our proofs of main results heavily rely on the properties and estimates of 
the singular solution $U$ of the linear heat equation,
which is defined by 
\begin{align}
	&U(x,t)=U(x,t;\underline{T})
	:=c_{n-m}^{-1}\int_{\underline{T}}^t \int_{M_s}
	G(x-\xi,t-s) d\mathcal{H}^m(\xi) ds,  \label{eq:Udefi} \\
	&c_{n-m}:=\frac{1}{4\pi^{(n-m)/2}}
	\int_0^\infty \tau^{\frac{n-m-2}{2}-1}
	e^{-\tau} d\tau,  \label{eq:cnm}
\end{align}
for $\underline{T}\in \mathbb{R}$. Here 
$G(x,t):=(4\pi t)^{-n/2}e^{-|x|^2/(4t)}$ is the heat kernel on $\mathbb{R}^n$, 
$\mathcal{H}^m$ is the $m$-dimensional Hausdorff measure on $\mathbb{R}^n$ 
and the normalized constant 
$c_{n-m}$ is the coefficient of the fundamental solution of 
Laplace's equation on $\mathbb{R}^{n-m}$. 
Almost all auxiliary functions in this paper 
are constructed by the nontrivial modifications of $U$. 
We note that $U$ is a natural parabolic analog of 
\[
	U_M(x):= \int_M \Gamma_{\mathbb{R}^n}(x-\xi) d\mathcal{H}^m(\xi), 
\]
where $M\subset\mathbb{R}^n$ is a compact $m$-dimensional submanifold and 
$\Gamma_{\mathbb{R}^n}$ is the fundamental solution of Laplace's equation on $\mathbb{R}^n$. 
Functions of type $U_M$ were employed 
by Delano\"e \cite{De92} and Kato and Nayatani \cite{KN93} 
for studying the singular Yamabe problem. 

Finally, we outline the rest of this paper and the proofs of main results. 
In Section \ref{sec:results}, we formulate our problem, 
state assumptions on $M_t$ and 
give the statement of Theorems \ref{th:nex} and \ref{th:ex}. 
The remaining sections can be divided into two parts. 
The former part, Sections \ref{sec:nonex} and \ref{sec:ess}, 
is devoted to the proofs of main results. 
The latter part, Sections \ref{sec:geom}, \ref{sec:U} and \ref{sec:supersub}, 
is devoted to the proofs of the facts employed in the former part. 

In Section \ref{sec:nonex}, 
the nonexistence theorem for $p\geq p_*$ is proved 
by checking the removability condition of Baras and Pierre \cite{BP84m}. 
Since the condition is stated by using the notion of a parabolic capacity, 
we first recall the definition of the capacity 
and their result. 
Then, we compute the capacity of the singular set. 
In Section \ref{sec:ess}, 
the existence theorem for $p<p_*$ is shown 
by the method of super- and sub-solutions. 
The proof of uniqueness is also given. 
To compute a parabolic capacity and to construct comparison functions，
the behavior of $U$ near $M_t$ plays a crucial role, 
where $U$ is given by \eqref{eq:Udefi}.  

In Section \ref{sec:geom}, 
we prepare a nice chart of $M_t$, the time-dependent Langer chart, 
and examine the properties of this chart. 
We also give 
estimates concerning integrals of functions over a tubular neighborhood 
of $M_{t}$. In Section \ref{sec:U}, we give fine estimates of $U$ near $M_t$ 
with the aid of the time-dependent Langer chart. 
We also prove uniform estimates of $U$. 
In Section \ref{sec:supersub}, we develop super- and sub-solution methods 
for non-cylindrical domains in $\mathbb{R}^n\times\mathbb{R}$.

\section{Main results}\label{sec:results}
In this section, we state assumptions and main results exactly. 
Let $M_{0} \subset \mathbb{R}^{n}$ be a connected compact 
$m$-dimensional smooth embedded submanifold without boundary and 
let $I$ be an open interval with $0\in I$. 
Let $F \in C^{\infty}(\mathbb{R}^{n}\times I;\mathbb{R}^{n})$ 
be a one parameter family of diffeomorphisms of $\mathbb{R}^n$ and let $F_{t}(x):= F(x,t)$ satisfy $F_{0} = \mathrm{id}_{\mathbb{R}^n}$. 
Define $M_{t}:=F_{t}(M_0)$, a time-dependent connected compact $m$-dimensional smooth embedded submanifold without boundary. 
Throughout this paper, we assume that $F_{t}$ can be extended for $t\in \overline{I}$. 
More precisely, there exist an open interval $J$ satisfying  $\overline{I}\subset J$ 
and 
a one parameter family of diffeomorphisms $\tilde{F} \in C^{\infty}(\mathbb{R}^{n}\times J;\mathbb{R}^{n})$ such that 
$F=\tilde{F}$ on $\mathbb{R}^{n}\times I$. 
This assumption ensures that 
$M_{t}$ is also a connected compact $m$-dimensional smooth embedded submanifold 
without boundary even for $t\in \overline{I}\setminus I$ 
by $M_{t}:=\tilde{F}_{t}(M_{0})$ for $t\in J\setminus I$. 
Additionally, we always assume that $F$ satisfies 
\begin{equation}\label{eq:bF}
	\left\{ 
	\begin{aligned}
	&\begin{aligned}
	&\sup_{s\in I}\sup_{p\in M_{0}}|\partial_t F(p,s)|
	+\sum_{i=1}^n \|\partial_{x_i} F\|_{L^\infty(\mathbb{R}^{n}\times I)}\\
	&\quad +\sum_{i=1}^n \|\partial_{t}\partial_{x_i} F\|_{L^\infty(\mathbb{R}^{n}\times I)}
	+\sum_{i,j=1}^n \|\partial_{x_i}\partial_{x_j} F\|_{L^\infty(\mathbb{R}^{n}\times I)}
	\leq B,
	\end{aligned}
	\\
	&\inf\left\{\frac{|F_{t}(p)-F_{t}(q)|}{|p-q|}; 
	p,q\in M_{0} \mbox{ with }p\neq q \mbox{ and } t\in I \right\}\geq B^{-1}
	\end{aligned}
	\right.
\end{equation}
for some constant $B=B(I)>0$.

\begin{example}\label{exofFt}
We clarify the meaning of the assumption \eqref{eq:bF} 
for $F$ and give some examples in two cases. 
\begin{description}
\item[Case 1] Assume that $I$ is bounded. For instance, let $I=(-1,1)$ and $J=(-2,2)$. 
Take a ball $B_{R}$ such that $M_{0}\subset B_{R}$. 
Assume that $F_{t}$ is a diffeomorphism of $\mathbb{R}^n$ 
for each $t\in J$ with $F_{0}=\mathrm{id}_{\mathbb{R}^n}$. 
If $F_{t}=\mathrm{id}_{\mathbb{R}^n}$ on $\mathbb{R}^{n}\setminus B_{R}$ for each $t\in J$, 
then $F$ restricted to $\mathbb{R}^{n}\times I$ always satisfies \eqref{eq:bF}. 
Furthermore, if $F$ satisfies $F_{t}=F_{t+2}$ for any $t\in J$ so that $t+2\in J$, 
then we have a periodic map $\overline{F}: \mathbb{R}^{n}\times (-1,\infty)\to \mathbb{R}^{n}$ by 
$\overline{F}_{s}:=F_{t}$ for $s=t+2k$ with $t\in (-1,1]$ and $k\in\mathbb{N}$. 
Then, $\overline{F}$ also satisfies \eqref{eq:bF}. 

\item[Case 2] 
Assume that $I$ is unbounded. 
For instance, let $I=(-1,\infty)$ and $J=(-2,\infty)$. 
Take $A\in C^{\infty}(J;\mathrm{GL}(n,\mathbb{R}))$ 
and $b\in C^{\infty}(J;\mathbb{R}^{n})$ with $A(0)=I_{n}$ and $b(0)=0$. 
Then, a one parameter family of diffeomorphisms 
$F_{t}(x):=A(t)x+b(t)$ satisfies \eqref{eq:bF} if 
$A$ and $b$ satisfy 
\begin{equation}\label{eq:bFeq2}
\sup_{t\in I}\left(|A(t)|+|A'(t)|+|b'(t)|\right)<\infty,\quad 
\inf_{t\in I}\inf_{v\neq 0}\frac{|A(t)v|}{|v|}>0. 
\end{equation}
For instance, if $A\in C^{\infty}(J; \mathrm{O}(n,\mathbb{R}))$ satisfies $\sup_{t\in I}|A'(t)|<\infty$ 
and $b$ satisfies $\sup_{t\in I}|b'(t)|<\infty$, 
then these satisfy \eqref{eq:bFeq2}, and so $F$ satisfies \eqref{eq:bF}. 
As a more concrete example, $F(t,x):=a(t)x+b(t)$ with $a\in C^{\infty}(J)$ satisfies \eqref{eq:bF} 
if $a(0)=1$, $b(0)=0$, $\inf_{t\in I}a(t)>0$ and $\sup_{t\in I}(|a(t)|+|a'(t)|+|b'(t)|)<\infty$. 
\end{description}
\end{example}

Let $\Omega\subset\mathbb{R}^n$ be a domain  
such that $M_t\subset \Omega$ for any $t\in \overline{I}$. 
Remark that $\Omega$ is not necessarily bounded. 
For an interval $I'\subset I$, we write 
\[
	Q_{\Omega,I'}:=\Omega\times I', \quad 
	M_{I'}:=\{(x,t)\in \mathbb{R}^n\times I'; x\in M_t \}.
\] 
We denote by $C^{2,1}$ the space of functions 
which are twice continuously differentiable in the $x$-variable  
and once in the $t$-variable.

First, we consider 
\begin{equation}\label{eq:main}
	\partial_t u-\Delta u= -u^p \quad
	\mbox{ in }Q_{\Omega,I}\setminus M_I
\end{equation}
and show a nonexistence result. 
The following result says that 
$M_I$ is removable if $p\geq p_*$, 
where $p_*$ is given by \eqref{eq:pstar}. 

\begin{theorem}\label{th:nex}
Let $m\geq 1$, $n\geq m+3$ and $p\geq p_*$. 
Suppose that $F$ satisfies \eqref{eq:bF} and that a nonnegative function 
$u\in C^{2,1}(Q_{\Omega,I}\setminus M_I)$ satisfies 
the equation \eqref{eq:main}. 
Then, $u\in C^{2,1}(Q_{\Omega,I})$ and 
$u$ is a classical solution of 
the equation \eqref{eq:main} in $Q_{\Omega,I}$. 
\end{theorem}

\begin{remark}\label{rem:multi}
The singular set $M_{I}$ in Theorem \ref{th:nex} can be replaced by 
a finite union of sets,  $M^{1}_{I}\cup\dots \cup M^{k}_{I}$, 
where each $M^{i}_{I}$ 
is a time track of $M_{t}^{i}:=F_{t}^{i}(M_{0}^{i})$ 
with $M_{0}^{i}$ and $F^{i}$ satisfying the assumptions in Theorem \ref{th:nex}. 
For the proof of this remark, see the end of Subsection \ref{subsec:pc}. 
\end{remark}

Next, for $p<p_*$ and $I=(\underline{T},\infty)$ with $\underline{T}\in(-\infty,0)$, 
we construct two types of nonnegative singular solution 
under the assumption that 
$F$ satisfies \eqref{eq:bF} and 
\begin{equation}\label{eq:HF}
	\sup\left\{  \frac{|F_t(p)-F_s(p)|}{|t-s|^{1/2}}; 
	p\in M_0 \mbox{ and } t,s\in I \mbox{ with }  t\neq s \right\}\leq B 
\end{equation}
with a constant $B=B(I)>0$.

\begin{example}
We explain the meaning of the assumption \eqref{eq:HF} by using 
$F_{t}(x):=A(t)x+b(t)$ for $t\in I:=(-1,\infty)$, 
where $J:=(-2,\infty)$, $A\in C^{\infty}(J;\mathrm{GL}(n,\mathbb{R}))$ 
and $b\in C^{\infty}(J;\mathbb{R}^{n})$ satisfy $A(0)=I_{n}$, $b(0)=0$ and \eqref{eq:bFeq2}. 
Note that 
\[\sup_{t\neq s}\sup_{p\in M_{0}}\frac{|F_{t}(p)-F_{s}(p)|}{|t-s|^{1/2}}\leq C\sup_{t\neq s}\frac{|A(t)-A(s)|}{|t-s|^{1/2}}
+\sup_{t\neq s}\frac{|b(t)-b(s)|}{|t-s|^{1/2}}=:C\mathcal{A}+\mathcal{B},\]
where $C:=\sup_{p\in M_{0}}|p|<\infty$. 
By \eqref{eq:bFeq2}, $\mathcal{A}<\infty$. 
However, to say that $\mathcal{B}<\infty$, the condition \eqref{eq:bFeq2} is not sufficient, 
and we have to control the asymptotic behavior of $b(t)$ when $t\to \infty$. 
Actually, under the assumption $\sup_{t\in I}|b'(t)|<\infty$, 
one can easily check that $\mathcal{B}<\infty$ if and only if 
$b(t)=O(\sqrt{t})$ as $t\to \infty$. 
Thus, by Example \ref{exofFt}, $F$ satisfies \eqref{eq:bF} and \eqref{eq:HF} if $A$ and $b$ satisfy \eqref{eq:bFeq2} 
and additionally $b(t)=O(\sqrt{t})$ as $t\to \infty$. 
For instance, $F(t,x):=x+(\sqrt{t+4}-2)b$ satisfies \eqref{eq:bF} and \eqref{eq:HF} 
for any constant vector $b\in\mathbb{R}^{n}$. 
\end{example}

The method to construct singular solutions 
is based on taking $\Omega=\mathbb{R}^n$ and 
solving an initial value problem below. 
Note that the case $\Omega\neq \mathbb{R}^n$ can be handled 
in the same way, see Remark \ref{rem:bdd}. 
We consider 
\begin{equation}\label{eq:mainex}
\left\{
\begin{aligned}
	&\partial_t u-\Delta u= -u^p &&\mbox{ in }
	Q_{(0,\infty)} \setminus M_{(0,\infty)},  \\
	&u(\cdot,0)= u_0 &&\mbox{ on }\mathbb{R}^n\setminus M_0. \\
\end{aligned}
\right.
\end{equation}
Here $Q_{I'}:=Q_{\mathbb{R}^n,I'}$  for  an interval $I'\subset I$ 
and $u_0$ is a nonnegative function 
which belongs to $X_{c,A}$ or $Y_A$ defined by 
\begin{align}
	&X_{c,A}:= \{
	v\in C(\mathbb{R}^n\setminus M_0); |v-cU(\cdot,0)|\leq A \mbox{ on }\mathbb{R}^n\setminus M_0 
	\}, 
	\label{eq:Xac}  \\
	&Y_A:= \left\{
	v\in C(\mathbb{R}^n\setminus M_0); \left| v-LU(\cdot,0)^\frac{2}{(n-m-2)(p-1)} \right|
	\leq A \mbox{ on }\mathbb{R}^n\setminus M_0 
	\right\}
	\label{eq:YL}
\end{align}
for constants $c>0$ and $A\geq 0$. 
Recall that 
$U=U(x,t;\underline{T})$ is defined by \eqref{eq:Udefi}. 

\begin{remark}
By Propositions \ref{pro:U} and \ref{pro:decay} below, $U$ behaves like 
$d(x,M_t)^{-(n-m-2)}$ near $M_t$ uniformly for $t\in[0,\infty)$ and 
$U$ is bounded by $C d(x,M_t)^{-(n-2)}$ 
for any $(x,t) \in Q_{[0,\infty)}\setminus M_{[0,\infty)}$. 
\end{remark}

We are now in a position to state a result of existence. 
This result particularly shows the optimality of Theorem \ref{th:nex} 
in view of the nonexistence of nonnegative singular solutions. 

\begin{theorem}\label{th:ex}
Let $m\geq 1$, $n\geq m+3$, $1<p<p_*$ and $A\geq0$. 
Suppose that $F$ satisfies \eqref{eq:bF} and \eqref{eq:HF}. 
Then the following {\rm(i)} and {\rm(ii)} hold. 
\begin{itemize}
\item[(i)]
For any constant $c>0$, the problem  \eqref{eq:mainex} with 
a nonnegative initial data $u_0\in X_{c,A}$ 
admits a unique nonnegative solution 
$u_{c,A}\in C^{2,1}(Q_{(0,\infty)}\setminus M_{(0,\infty)})\cap 
C(Q_{[0,\infty)}\setminus M_{[0,\infty)})$ satisfying 
\begin{equation}\label{eq:asymnm2uni}
	\lim_{x\to M_t} d(x,M_t)^{n-m-2}u_{c,A}(x,t)=  c \quad
	\mbox{ uniformly for }t\in[0,\infty). 
\end{equation}
\item[(ii)]
The problem  \eqref{eq:mainex} with 
a nonnegative initial data $u_0\in Y_A$ 
admits a unique nonnegative solution 
$u_A\in C^{2,1}(Q_{(0,\infty)}\setminus M_{(0,\infty)})\cap 
C(Q_{[0,\infty)}\setminus M_{[0,\infty)})$ satisfying 
\begin{equation}\label{eq:sasymnm2uni}
	\lim_{x\to M_t} d(x,M_t)^{\frac{2}{p-1}}u_A(x,t)=  L \quad
	\mbox{ uniformly for }t\in[0,\infty), 
\end{equation}
where the positive constant $L$ is given by \eqref{eq:Ldefi}. 
\end{itemize}
\end{theorem}

\begin{remark}\label{rem:unilim}
In this paper, the meaning of \eqref{eq:asymnm2uni} and \eqref{eq:sasymnm2uni} 
is as follows. 
For a subset $M\subset \mathbb{R}^n$ and $\delta>0$, 
we write 
\[
	M^\delta:= \{x\in \mathbb{R}^n; d(x,M)<\delta \}. 
\]
Let $f$ be a function defined on a neighborhood of $M$. 
We say that $\lim_{x\to M}f(x)=a$ 
for some $a\in\mathbb{R}$ (resp. $\lim_{x\to M}f(x)=+\infty$) 
if, for any $\varepsilon>0$, there exists $\delta>0$ such that 
$|f(x)-a|\leq \varepsilon$ (resp. $f(x)\geq 1/\varepsilon$) for any $x\in M^\delta\setminus M$. 
Let $I'\subset I$ be an interval and 
let $g$ be a function defined on a neighborhood of $M_{I'}$. 
We say that 
\[
	\lim_{x\to M_t}g(x,t)=a \quad \mbox{ uniformly for } t\in I', 
\]
if, for any $\varepsilon>0$, there exists $\delta>0$ independent of $t$ such that 
$|g(x,t)-a|\leq \varepsilon$ for any $x\in M_t^\delta\setminus M_t$ and $t\in I'$. 
\end{remark}

\begin{remark}\label{rem:bdd}
In the case $\Omega\neq\mathbb{R}^n$, 
an analog of Theorem \ref{th:ex} can be proved 
by a simple modification of the proof of Theorem \ref{th:ex}. 
In this case, we assume that $\Omega$ is a smooth domain and solve 
\[
\left\{
\begin{aligned}
	&\partial_t u-\Delta u= -u^p &&\mbox{ in }
	Q_{\Omega, (0,\infty)} \setminus M_{(0,\infty)},  \\
	&u= 0 &&\mbox{ on }\partial \Omega \times (0,\infty), \\
	&u(\cdot,0)= u_0 &&\mbox{ on }\Omega\setminus M_0, \\
\end{aligned}
\right.
\]
where $u_0\geq0$ belongs to $\tilde X_{c,A}$ or $\tilde Y_A$ defined by 
\[
\begin{aligned}
	&\tilde X_{c,A}:= \left\{
	v\in C(\Omega\setminus M_0); |v-cU(\cdot,0)|\leq A \mbox{ on }\mathbb{R}^n\setminus M_0 
	\mbox{ and }v=0 \mbox{ on }\partial \Omega
	\right\}, 
	\\
	&\tilde Y_A:= \left\{
	v\in C(\Omega\setminus M_0); 
	\begin{aligned}
	&\left| v-LU(\cdot,0)^\frac{2}{(n-m-2)(p-1)} \right|
	\leq A \mbox{ on }\mathbb{R}^n\setminus M_0 \\
	&\mbox{ and }v=0 \mbox{ on }\partial \Omega
	\end{aligned}
	\right\}. 
\end{aligned}
\]
Then we obtain singular solutions in 
$C^{2,1}(Q_{\Omega,(0,\infty)}\setminus M_{(0,\infty)})\cap 
C(  (Q_{\Omega,[0,\infty)}\setminus M_{[0,\infty)})
\cup(\partial \Omega \times (0,\infty))  )$. 
\end{remark}

\begin{remark}\label{rem:codim2}
In the case where $n-m=2$ and $p>1$, 
Theorem \ref{th:ex} seems to be true 
if $d(x,M_t)^{n-m-2}$ is replaced by $(-\log d(x,M_t))^{-1}$. 
However, the authors do not have a proof, since 
the estimates as in Proposition \ref{pro:U} may not hold in this case. 
\end{remark}

\begin{remark}\label{rem:integ}
In Theorem \ref{th:ex}, the strength of the singularity of $u_{c,A}$ 
is weaker than that of $u_A$. 
Moreover, $u_{c,A}\in L^p_\mathrm{loc}(Q_{(0,\infty)})$ and 
$u_A\not \in L^p_\mathrm{loc}(Q_{(0,\infty)})$ 
follows from Proposition \ref{tdeptnbint}. 
Then, 
as an analog of the elliptic results \cite{BO87} and \cite{Gr97}, 
it seems to the authors that $u_{c,A}$ is driven by 
some measure with support $M_{[0,\infty)}$ and $u_A$ is not so. 
Furthermore, 
it is expected that $u_{c,A}\to u_A$ in some sense as $c\uparrow+\infty$. 
However, these topics exceed the aim of this paper, and so 
we leave these for future work. 
\end{remark}

Theorems \ref{th:nex} and \ref{th:ex} are proved in 
Sections \ref{sec:nonex} and \ref{sec:ess}, respectively. 
Ingredients of the proofs are given by 
Sections \ref{sec:geom}, \ref{sec:U} and \ref{sec:supersub}.

\section{Nonexistence of singular solutions}\label{sec:nonex}
Let $I$ be an open interval with $0\in I$. 
In this section, we assume \eqref{eq:bF} for $F$ and do not assume \eqref{eq:HF}. 
In Subsection \ref{subsec:pc}, 
we recall the definition of a parabolic capacity, 
and then we prove Theorem \ref{th:nex} 
by some properties of the capacity. 
The properties are shown in Subsection \ref{subsec:tns}.

\subsection{A parabolic capacity}\label{subsec:pc}
We refer to Baras and Pierre~\cite[Paragraphe 2]{BP84m} and 
recall the definition of the parabolic capacity $c^{2,1}_q$ 
associated with the Sobolev space $W^{2,1}_q$. 
For $q\geq1$, we write 
\[
\begin{aligned}
	&W^{2,1}_q:=
	\left\{
	u\in L^q; 
	\partial_t u, \partial_{x_i} u, \partial_{x_i x_j} u\in L^q, 
	i,j=1,\ldots,n
	\right\}, \\
	&\|u\|_{W^{2,1}_q}
	:=
	\left( \|u\|_{L^q}^q +
	\|\partial_t u\|_{L^q}^q +
	\sum_{i=1}^n \|\partial_{x_i} u\|_{L^q}^q +
	\sum_{i,j=1}^n \|\partial_{x_i}\partial_{x_j} u\|_{L^q}^q \right)^{1/q},
\end{aligned}
\]
where $L^q=L^q(\mathbb{R}^{n+1})$. 

\begin{definition}
We define the $c^{2,1}_q$-capacity of a compact subset $\mathcal{C}\subset \mathbb{R}^{n+1}$ by 
\[
	c^{2,1}_q(\mathcal{C}):=
	\inf\left\{
	\|f\|_{W^{2,1}_q}^q; 
	f\in C^{\infty}_0(\mathbb{R}^{n+1}), 
	f\geq 1 \mbox{ on a neighborhood of }\mathcal{C}
	\right\}. 
\]
For an open set $\mathcal{O}$ and a subset $\mathcal{E}$, the $c^{2,1}_q$-capacity is also defined  by
\[
\begin{aligned}
	c^{2,1}_q(\mathcal{O}) & :=
	\sup\left\{
	c^{2,1}_q(\mathcal{C}); 
	\mathcal{C}\subset \mathcal{O} \mbox{ for a compact subset $\mathcal{C}$ in $\mathbb{R}^{n+1}$}
	\right\},  \\
	c^{2,1}_q(\mathcal{E}) & :=
	\inf\left\{
	c^{2,1}_q(\mathcal{O}); 
	\mathcal{O}\supset \mathcal{E} \mbox{ for an open subset $\mathcal{O}$ in $\mathbb{R}^{n+1}$}
	\right\}. 
\end{aligned}
\]
\end{definition}

By using the notion of $c^{2,1}_q$-capacity, 
Baras and Pierre \cite{BP84m} characterized 
the condition on 
the removability of a singular set for the solutions of 
\begin{equation}\label{eq:absdsing}
	\partial_t u-\Delta u= -|u|^{p-1}u \mbox{ in }\mathcal{D}'(Q_{\Omega,I}\setminus \tilde M), 
\end{equation}
where $\tilde M$ is a relatively closed subset of $Q_{\Omega,I}$. 

\begin{theorem}
[\mbox{Baras and Pierre~\cite[Th\'eor\`em 4.1]{BP84m}}]\label{th:BP}
Let $p,q>1$ such that $1/p+1/q=1$. 
Then $\tilde M$ is removable 
if and only if the $c^{2,1}_{q}$-capacity of $\tilde M$ is null. 
More precisely, any solution 
$u\in L^p_\mathrm{loc}(Q_{\Omega,I}\setminus \tilde M)$ 
of \eqref{eq:absdsing} belongs to $L^p_\mathrm{loc}(Q_{\Omega,I})$ and satisfies
\begin{equation}\label{eq:absd}
	\partial_t u-\Delta u= -|u|^{p-1}u \mbox{ in }\mathcal{D}'(Q_{\Omega,I}) 
\end{equation}
if and only if $c^{2,1}_{q}(\tilde M)=0$. 
\end{theorem}
 
By~\cite[Propositions~2.3 and 2.4]{BP84m}, 
a typical example of null sets is $M\times [T_1,T_2]$. 
Here $M$ is an $m$-dimensional 
compact smooth submanifold and $T_1<T_2$. 
Under this condition, 
$c^{2,1}_{q}(M\times [T_1,T_2])=0$ if $n-m\geq3$, $p\geq p_*$ and $1/p+1/q=1$. 
For a time-dependent set, the same property can be proved. 

\begin{proposition}\label{pro:null}
Let $p\geq p_*$ and $q>1$ satisfy $1/p+1/q=1$. 
Then $c^{2,1}_q(M_{[t_1,t_2]})=0$ for any $t_1, t_2 \in I$ with $t_1<t_2$. 
In particular, $c^{2,1}_q(M_{(t_1,t_2)})=0$. 
\end{proposition}

Before proving this proposition, we show Theorem~\ref{th:nex}. 

\begin{proof}[Proof of Theorem~\ref{th:nex}]
Let $u\in C^{2,1}(Q_{\Omega,I}\setminus M_I)$ 
be a nonnegative solution of \eqref{eq:main}. 
By Proposition~\ref{pro:null} and Theorem~\ref{th:BP}, 
$u$ satisfies \eqref{eq:main} in 
$\mathcal{D}'(Q_{\Omega,(t_1,t_2)})$ for any $t_1,t_2\in I$ with $t_1<t_2$. 
Hence $u$ satisfies \eqref{eq:absd}. 
In particular, $u\geq 0$ satisfies 
$\partial_t u-\Delta u\leq 0$ in $\mathcal{D}'(Q_{\Omega,I})$. 
Therefore $u\in L^\infty_\mathrm{loc}(Q_{\Omega,I})$ 
(see for instance \cite[p. 76]{BF83}). 
The standard regularity theory for parabolic equations yields 
$u\in C^{2,1}(Q_{\Omega,I})$. We conclude that 
$u$ is a classical solution of \eqref{eq:main} in $Q_{\Omega,I}$. 
\end{proof}

Remark \ref{rem:multi} follows from 
Proposition \ref{pro:null} and 
\[
c^{2,1}_q(\mathcal{C}_{1}\cup \mathcal{C}_{2})
\leq 2^{q-1}\left(c^{2,1}_q(\mathcal{C}_{1})+c^{2,1}_q(\mathcal{C}_{2})\right)
\]
for any compact subsets $\mathcal{C}_{1}$ and $\mathcal{C}_{2}$ in $\mathbb{R}^{n+1}$.

\subsection{Time-dependent null sets}\label{subsec:tns}
In order to show Proposition \ref{pro:null}, 
we define auxiliary functions and prepare a lemma. 
Fix $\varepsilon>0$ and 
$t_1,t_2,t_3,t_4,t_5,t_6,t_7,t_8\in I$ such that 
$t_7<t_5<t_3<t_1<t_2<t_4<t_6<t_8$. 
Let 
$\eta\in C^\infty(\mathbb{R})$ and $h \in C^\infty(\mathbb{R})$ 
satisfy 
\[
\left\{
\begin{aligned}
&0< \eta(t) < 1 &&\mbox{ for }t\in (t_5,t_3)\cup(t_4,t_6),  \\
&\eta(t)=0 &&\mbox{ for } t\in(-\infty,t_5]\cup[t_6,\infty),  \\
&\eta(t)=1 &&\mbox{ for }  t\in[t_3,t_4], 
\end{aligned}
\right. \quad
\left\{ 
\begin{aligned}
&0< h(\tau)< 1 &&\mbox{ for }\tau\in (1,2), \\
&h(\tau)=0 &&\mbox{ for } \tau\in(-\infty,1],   \\
&h(\tau)=1 &&\mbox{ for } \tau\in[2,\infty).  
\end{aligned}
\right.
\]
See Figure \ref{fig:1}. 
\begin{figure}[tb]
\includegraphics[bb=37 373 557 471, clip, scale=0.69]{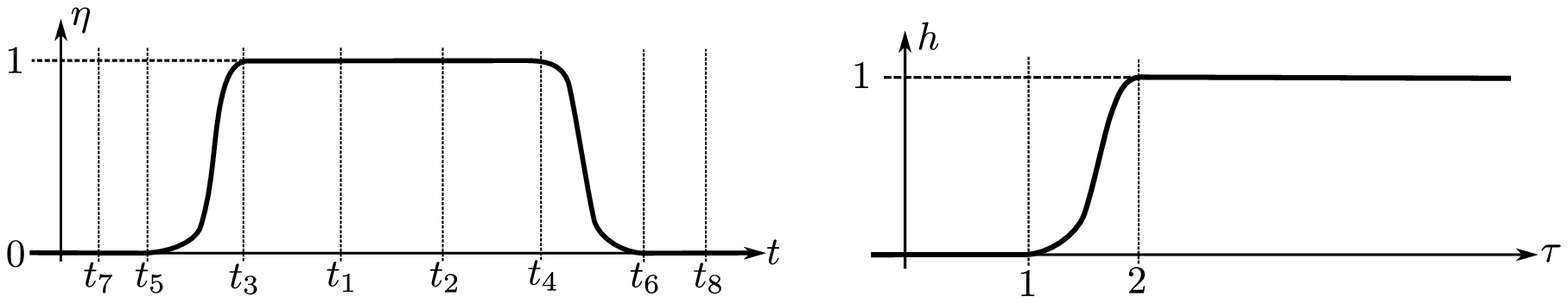}
\caption{$\eta$ and $h$}
\label{fig:1}
\end{figure}
Note that there exists a constant $C=C(t_3,t_4,t_5,t_6)>0$ such that 
\begin{equation}\label{eq:etahb}
	\|\eta'\|_{L^\infty(\mathbb{R})} 
	+\|h'\|_{L^\infty(\mathbb{R})}
	+\|h''\|_{L^\infty(\mathbb{R})} \leq C. 
\end{equation}
In this subsection, we set $\underline{T}=t_7$ in \eqref{eq:Udefi}. Namely, 
\[
	U=U(x,t;t_7)=c_{n-m}^{-1}
	\int_{t_7}^t\int_{M_s} G(x-\xi,t-s) d\mathcal{H}^m(\xi) ds. 
\] 
We will prove $c^{2,1}_q(M_{[t_1,t_2]})=0$ 
by using the following functions. For $\varepsilon>0$, define 
\[
	f_\varepsilon(x,t):=\left\{
	\begin{aligned}
	&\eta(t)h ( V_\varepsilon (x,t))
	&&\quad\mbox{if }d(x,M_t)>0,
	\\&\eta(t)
	&&\quad\mbox{if }d(x,M_t)=0,  
	\end{aligned}\right.
	\quad 
	V_\varepsilon(x,t):= \frac{\log U(x,t)}{(n-m-2)\log(1/\varepsilon)}. 
\]

In what follows, we write $d=d(x,M_t)$ when no confusion can arise. 
By Proposition \ref{pro:EUnsg} and \ref{pro:decay}, 
there exist constants $\delta>0$ and $C_0=C_0(\delta)>0$ such that 
\begin{align}
	&(C_0 d)^{-(n-m-2) }\leq U(x,t) \leq (C_0^{-1} d)^{-(n-m-2)}, \label{eq:Uulc0} 
	\\
	&|\nabla U(x,t)| \leq (C_0^{-1} d)^{-(n-m-1)}, \label{eq:nUuc0}
	\\
	&|\partial_t U(x,t)| + \sum_{i,j=1}^n |\partial_{x_i}\partial_{x_j} U(x,t)| 
	\leq (C_0^{-1} d)^{-(n-m)}  \label{eq:ppUuc0}
\end{align}
for $x\in M_t^\delta \setminus M_t$ and $t\in [t_5,\infty)$, and 
\begin{equation}\label{eq:Uunic0}
	U(x,t) \leq (C_0^{-1} d)^{-(n-2)}, \quad 
	x\in \mathbb{R}^n \setminus M_t, t\in [t_5,\infty). 
\end{equation}
Then, $f_\varepsilon\in C^\infty_0(\mathbb{R}^{n+1})$ and 
$f_\varepsilon=1$ on a neighborhood of $M_{[t_1,t_2]}$. 
Let 
\begin{equation}\label{eq:epsdef}
	0<\varepsilon< \sup \left\{ \tilde \varepsilon\in(0,e^{-1}); C_0^{-1}\tilde \varepsilon^2 < C_0\tilde \varepsilon 
	\leq C_0\tilde \varepsilon^\frac{n-m-2}{n-m} <\delta
	\right\}. 
\end{equation}

\begin{lemma}
The following properties hold for $t\in [t_5,t_6]$. 
\begin{align}
	\label{eq:h1}
	&h(V_\varepsilon(x,t))=1 
	\quad \mbox{ on }D^{\text{in}}_\varepsilon, \quad 
	D^{\text{in}}_\varepsilon:=
	\{x\in \mathbb{R}^n; 0<d(x,M_t)\leq C_0^{-1}\varepsilon^2\}, \\
	\label{eq:h0}
	&h(V_\varepsilon(x,t))=0 
	\quad \mbox{ on }D^{\text{out}}_\varepsilon, \quad 
	D^{\text{out}}_\varepsilon:=
	\{x\in \mathbb{R}^n; d(x,M_t)\geq C_0\varepsilon\}. 
\end{align}
\end{lemma}

\begin{proof}
Fix $t\in [t_5,t_6]$. 
By the definition of $h$ and 
the first inequality of \eqref{eq:Uulc0}, we have 
\[
\begin{aligned}
	\left\{
	x\in \mathbb{R}^n\setminus M_t; 
	h(V_\varepsilon(x,t)) =1 
	\right\}  
	&\supset \left\{
	x\in M_t^\delta \setminus M_t; 
	\frac{\log U(x,t)}{(n-m-2)\log(1/\varepsilon)} \geq 2 
	\right\}  \\
	&\supset \left\{
	x\in M_t^\delta \setminus M_t; 
	\frac{\log ((C_0 d)^{-(n-m-2)})}{\log(\varepsilon^{-(n-m-2)})}
	 \geq 2 
	\right\}  \\
	&= \left\{
	x\in M_t^\delta \setminus M_t; d \leq C_0^{-1}\varepsilon^2 
	\right\}. 
\end{aligned}
\]
This shows \eqref{eq:h1}. 
By the second inequality of \eqref{eq:Uulc0}, we also have 
\[
\begin{aligned}
	\left\{
	x\in \mathbb{R}^n\setminus M_t; 
	h(V_\varepsilon(x,t)) =0 
	\right\} 
	&\supset \left\{
	x\in M_t^\delta \setminus M_t; 
	\frac{\log U(x,t)}{(n-m-2)\log(1/\varepsilon)} \leq 1 
	\right\}  \\
	&\supset \left\{
	x\in M_t^\delta \setminus M_t; 
	\frac{\log ((C_0^{-1}d)^{-(n-m-2)})}{\log(\varepsilon^{-(n-m-2)})}
	 \leq 1 
	\right\}  \\
	&= \left\{
	x\in M_t^\delta \setminus M_t; d\geq C_0\varepsilon 
	\right\}. 
\end{aligned}
\]
Moreover, \eqref{eq:Uunic0} gives
\[
\begin{aligned}
	\left\{
	x\in \mathbb{R}^n\setminus M_t; 
	h(V_\varepsilon(x,t)) =0 
	\right\} 
	&\supset \left\{
	x\in \mathbb{R}^n\setminus M_t; 
	\frac{\log U(x,t)}{(n-m-2)\log(1/\varepsilon)} \leq 1 
	\right\}  \\
	&\supset \left\{
	x\in \mathbb{R}^n\setminus M_t; 
	\frac{\log ((C_0^{-1}d)^{-(n-2)})}{\log(\varepsilon^{-(n-m-2)})}
	 \leq 1 
	\right\}  \\
	&= \left\{
	x\in \mathbb{R}^n\setminus M_t; d\geq C_0\varepsilon^\frac{n-m-2}{n-2}
	\right\}. 
\end{aligned}
\]
Then by \eqref{eq:epsdef}, we obtain \eqref{eq:h0}. 
\end{proof}

\begin{proof}[Proof of Proposition~\ref{pro:null}]
First of all, we write $N:=n-m$ in this proof. Note that 
\begin{equation}\label{eq:qcondi}
	1<q\leq \frac{N}{2} \left(= \frac{n-m}{2}\right),
\end{equation}
since $p\geq p_*$ and $1/p+1/q=1$. 
To prove $c^{2,1}_q(M_{[t_1,t_2]})=0$, 
it suffices to show that 
\begin{equation}\label{eq:fepsgoal}
	\|f_\varepsilon\|^q_{L^q}
	+\sum_{i=1}^n\|\partial_{x_i} f_\varepsilon\|^q_{L^q}
	+\sum_{i,j=1}^n\|\partial_{x_i} \partial_{x_j}f_\varepsilon\|^q_{L^q}
	+\|\partial_t f_\varepsilon\|^q_{L^q} \to 0 
\end{equation}
as $\varepsilon\to0$. 
In what follows, we only give estimates of 
$\|\partial_{x_i}\partial_{x_j}f_\varepsilon\|^q_{L^q}$ and 
$\|\partial_t f_\varepsilon\|^q_{L^q}$, 
since the norms $\|f_\varepsilon\|^q_{L^q}$ and 
$\|\partial_{x_i}f_\varepsilon\|^q_{L^q}$ are handled similarly. 
Throughout the rest of this proof, we always write 
\[
	h= h(V_\varepsilon(x,t)), 
	\quad
	h'= h'(V_\varepsilon(x,t)), 
	\quad
	h''= h''(V_\varepsilon(x,t)). 
\]

By \eqref{eq:h1} and \eqref{eq:h0}, we have 
\[
\begin{aligned}
	\partial_{x_i} \partial_{x_j} f_\varepsilon&=
	\eta h_\varepsilon'' 
	\frac{U^{-2} (\partial_{x_i}U) ( \partial_{x_j}U) }{(N-2)^2 (\log(1/\varepsilon))^2}
	\chi_{\mathbb{R}^n\setminus(D^{\text{in}}_\varepsilon \cup D^{\text{out}}_\varepsilon)} \\
	&\quad +
	\eta h_\varepsilon' 
	\frac{U^{-1}(\partial_{x_i}\partial_{x_j}U )
	- U^{-2} (\partial_{x_i}U) (\partial_{x_j}U) }{(N-2) \log(1/\varepsilon)} 
	\chi_{\mathbb{R}^n\setminus(D^{\text{in}}_\varepsilon \cup D^{\text{out}}_\varepsilon)} 
\end{aligned}
\]
for $t\in[t_5,t_6]$ and $x\in\mathbb{R}^n\setminus M_t$. 
Since
$\mathbb{R}^n\setminus(D^{\text{in}}_\varepsilon \cup D^{\text{out}}_\varepsilon) \subset 
M_t^\delta$, 
the first inequality of \eqref{eq:Uulc0} and
the estimates \eqref{eq:etahb}, \eqref{eq:nUuc0} and \eqref{eq:ppUuc0} yield 
\begin{equation}\label{eq:ppfepsbo}
\begin{aligned}
	&|\partial_{x_i}\partial_{x_j}f_\varepsilon| \\
	&\leq 
	\eta 
	\left| h_\varepsilon'' \right|
	\frac{U^{-2} |(\partial_{x_i}U) (\partial_{x_j}U)|}{(N-2)^2 (\log(1/\varepsilon))^2}
	\chi_{\mathbb{R}^n\setminus(D^{\text{in}}_\varepsilon \cup D^{\text{out}}_\varepsilon)} \\
	&\quad +
	\eta 
	\left|h_\varepsilon' \right|
	\frac{U^{-1}|\partial_{x_i}\partial_{x_j}U| + U^{-2} 
	|(\partial_{x_i} U)(\partial_{x_j}U)|}{(N-2) \log(1/\varepsilon)} 
	\chi_{\mathbb{R}^n\setminus(D^{\text{in}}_\varepsilon \cup D^{\text{out}}_\varepsilon)} \\
	&\leq
	\frac{C}{\log(1/\varepsilon)} \eta d^{-2}
	\chi_{\mathbb{R}^n\setminus(D^{\text{in}}_\varepsilon \cup D^{\text{out}}_\varepsilon)}
\end{aligned}
\end{equation}
for $t\in[t_5,t_6]$ and $x\in\mathbb{R}^n\setminus M_t$ 
with some constant $C=C(n,m,C_0,t_3,t_4,t_5,t_6)>0$. 
Here, to obtain the second inequality, 
we have used $\varepsilon<e^{-1}$ in \eqref{eq:epsdef}. 
Then from \eqref{eq:etahb} and 
Proposition \ref{tdeptnbint} 
with $\beta=-2q$, $\delta=C_{0}\varepsilon$ and $\delta'=C_0^{-1}\varepsilon^2$,  
it follows that 
\[
\begin{aligned}
	\sum_{i,j=1}^n \|\partial_{x_i}\partial_{x_j}f_\varepsilon\|^q_{L^q} 
	&\leq 
	\frac{C}{(\log(1/\varepsilon))^q} \int_{t_5}^{t_6} 
	\int_{C_0^{-1} \varepsilon^2 <d(x,M_t)<C_0 \varepsilon }
	\eta(t)^q d(x,M_t)^{-2q} 
	dxdt \\
	&\leq 
	\frac{C}{(\log(1/\varepsilon))^q}
	\int_{t_5}^{t_6}
	\int_{C_0^{-1}\varepsilon^2<d(x,M_t)<C_0 \varepsilon }
	d(x,M_t)^{-2q}dxdt  \\
	&\leq 
	\left\{ 
	\begin{aligned}
	& \frac{C}{(\log(1/\varepsilon))^q} \times C\varepsilon^{n-m-2q} &&(n-m>2q) \\
	& \frac{C}{(\log(1/\varepsilon))^q} \times C \log\frac{1}{\varepsilon}&&(n-m=2q) 
	\end{aligned} 
	\right.  
\end{aligned}
\]
for some constant $C=C(n,m,q,C_0,t_3,t_4,t_5,t_6)>0$. 
Hence using \eqref{eq:qcondi} gives 
\begin{equation}\label{eq:ppfto0}
	\sum_{i,j=1}^n \|\partial_{x_i}\partial_{x_j}f_\varepsilon\|^q_{L^q} 
	\to 0 \quad \mbox{ as }\varepsilon\to0. 
\end{equation}

By \eqref{eq:h1} and \eqref{eq:h0}, we see that 
\[
	\partial_t f_\varepsilon =
	\eta' h \chi_{\mathbb{R}^n\setminus D^{\text{out}}_\varepsilon} 
	+ \eta h' \frac{U^{-1} \partial_t U}{(N-2) \log(1/\varepsilon)} 
	\chi_{\mathbb{R}^n\setminus(D^{\text{in}}_\varepsilon \cup D^{\text{out}}_\varepsilon)}  
\]
for $t\in[t_5,t_6]$ and $x\in\mathbb{R}^n\setminus M_t$. 
Similarly to \eqref{eq:ppfepsbo}, we have 
\[
\begin{aligned}
	|\partial_t f_\varepsilon| &\leq 
	|\eta'| h \chi_{\mathbb{R}^n\setminus D^{\text{out}}_\varepsilon} 
	+ \eta |h'| \frac{U^{-1} |\partial_t U|}{(N-2) \log(1/\varepsilon)} 
	\chi_{\mathbb{R}^n\setminus(D^{\text{in}}_\varepsilon \cup D^{\text{out}}_\varepsilon)} \\
	&\leq 
	|\eta'| h \chi_{\mathbb{R}^n\setminus D^{\text{out}}_\varepsilon} 
	+ \frac{C}{\log(1/\varepsilon)} \eta d^{-2}
	\chi_{\mathbb{R}^n\setminus(D^{\text{in}}_\varepsilon \cup D^{\text{out}}_\varepsilon)}  
\end{aligned}
\]
for $t\in[t_5,t_6]$ and $x\in\mathbb{R}^n\setminus M_t$, where 
$C=C(n,m,C_0,t_3,t_4,t_5,t_6)>0$ is a constant. 
By \eqref{eq:etahb} and 
Proposition \ref{tdeptnbint} 
with $\beta=0$, $\delta=C_{0}\varepsilon$ and $\delta'\to 0$, we have
\[
	\| \eta' h \chi_{\mathbb{R}^n\setminus D^{\text{out}}_\varepsilon} \|^q_{L^q}
	\leq
	C \int_{t_5}^{t_6} \int_{0<d(x,M_t)<C_0\varepsilon} dxdt  
	\leq C \varepsilon^{n-m}
\]
for some constant $C=C(n,m,C_0,t_5,t_6)>0$. 
This together with the way to prove \eqref{eq:ppfto0} gives 
\begin{equation}\label{eq:ptfto0}
	\sum_{i,j=1}^n \|\partial_t f_\varepsilon\|^q_{L^q} 
	\to 0 \quad \mbox{ as }\varepsilon\to0.  
\end{equation}

Similarly, 
we can see that  
\begin{equation}\label{eq:fto00}
	\|f_\varepsilon\|^q_{L^q} + \sum_{i=1}^n \|\partial_{x_i}f_\varepsilon\|^q_{L^q}
	\leq
	C \left( \varepsilon^{n-m}+ \frac{\varepsilon^{n-m-q}}{(\log(1/\varepsilon))^q} \right) \to 0
\end{equation}
for some constant $C=C(n,m,q,C_0,t_3,t_4,t_5,t_6)>0$ as $\varepsilon\to0$. 
Combining \eqref{eq:ppfto0}, \eqref{eq:ptfto0} and 
\eqref{eq:fto00},  
we obtain \eqref{eq:fepsgoal}. 
Then the proof is complete. 
\end{proof}

We remark that the above argument does not valid 
for the case $p<p_*$, since 
$C(\log(1/\varepsilon))^{-q}\times C\varepsilon^{n-m-2q}\not \to0$ as $\varepsilon\to0$. 
In this case, as a corollary of Theorems \ref{th:BP} and \ref{th:ex}, 
the following holds. 

\begin{corollary}
Suppose $1<p<p_*$, $q>1$ and $1/p+1/q=1$. 
Then $c^{2,1}_q(M_{[t_1,t_2]})>0$ for any $t_1, t_2 \in I$ with $t_1<t_2$. 
\end{corollary}

Even for $p<p_*$, we can show a result of nonexistence 
under the assumption that, for any bounded interval $I'$ satisfying $\overline{I'}\subset I$, 
\begin{equation}\label{eq:apriori}
	\lim_{x\to M_t} d(x,M_t)^{n-m-2} u(x,t) = 0
	\quad \mbox{ uniformly for }t\in I'. 
\end{equation}

\begin{proposition}\label{pro:nexsub}
Let $m\geq1$, $n\geq m+3$ and $1<p<p_*$. 
Suppose that $F$ satisfies \eqref{eq:bF} and that 
a nonnegative function $u\in C^{2,1}(Q_{\Omega,I}\setminus M_I)$ 
satisfies \eqref{eq:main} and 
\eqref{eq:apriori} for any bounded interval $I'$ with $\overline{I'}\subset I$. 
Then, $u\in C^{2,1}(Q_{\Omega,I})$ and $u$ 
is a solution of \eqref{eq:main} in $Q_{\Omega,I}$. 
\end{proposition}

\begin{proof}
We only give an outline. 
Define $f_\varepsilon$ as in Subsection \ref{subsec:tns} and 
write $g_\varepsilon:=1-f_\varepsilon$. 
Note that $g_\varepsilon=0$ on a neighborhood of $M_I$. 
Let $\varphi\in C^\infty_0(Q_{\Omega,I})$. 
Since $u\geq0$ satisfies \eqref{eq:absdsing}, we have 
\[
	\iint_{Q_{\Omega,I}} 
	\left( 
	-\partial_t(\varphi g_\varepsilon)-\Delta(\varphi g_\varepsilon)
	\right) udxdt 
	= 
	-\iint_{Q_{\Omega,I}}  \varphi g_\varepsilon u^p dxdt. 
\]
Then by \eqref{eq:apriori} and similar computations to the proof of 
Proposition \ref{pro:null}, 
one can directly show that $u\geq0$ satisfies \eqref{eq:absd}. 
Thus, $u\in C^{2,1}(Q_{\Omega,I})$ and 
$u$ is a classical solution of \eqref{eq:main} in $Q_{\Omega,I}$. 
The details are left to the reader. 
\end{proof}

\section{Existence of singular solutions}\label{sec:ess}
Let $I=(\underline{T},\infty)$ with $\underline{T}<0$. 
Actually, we always take $\underline{T}=-2$  for simplicity. 
Throughout this section, we consider the case $1<p<p_*$ and 
assume that $F$ satisfies \eqref{eq:bF} and \eqref{eq:HF}. 
By the method of super- and sub-solutions (for details see Section \ref{sec:supersub}), 
we construct a nonnegative solution of 
\begin{equation}\label{eq:iniabs}
	\left\{ 
	\begin{aligned}
	&\partial_t u-\Delta u= -|u|^{p-1}u &&
	\mbox{ in }Q_{(0,\infty)} \setminus M_{(0,\infty)},  \\
	&u(\cdot,0) = u_0  && 
	\mbox{ on }\mathbb{R}^n\setminus M_0  
	\end{aligned}
	\right.
\end{equation}
under the condition that 
$u_0\geq 0$ belongs to $X_{c,A}$ or $Y_A$ for some $c>0$ and $A>0$, 
where $X_{c,A}$ and $Y_A$ are defined by \eqref{eq:Xac} and \eqref{eq:YL}. 
Our comparison functions are constructed from 
the singular solution $U=U(x,t;-2)$ of the linear heat equation, where 
\[
	U=U(x,t;-2)=
	c_{n-m}^{-1}\int_{-2}^t\int_{M_s} G(x-\xi,t-s) d\mathcal{H}^m(\xi) ds. 
\] 

First, we prove the nonnegativity and the uniqueness of solutions. 
Next, in Subsections \ref{subsect:weaks} and \ref{subsect:strongs}, 
we construct two types of singular solution. 

\begin{lemma}\label{lem:posiw}
Let $u_0$ satisfy $u_0 \in C(\mathbb{R}^n\setminus M_0)$ and 
$u_0 \geq 0$ on $\mathbb{R}^n\setminus M_0$. 
Suppose that $u\in C^{2,1}(Q_{(0,\infty)} \setminus M_{(0,\infty)}) 
\cap C(Q_{[0,\infty)} \setminus M_{[0,\infty)})$ 
is a solution of \eqref{eq:iniabs} satisfying
$\lim_{x\to M_t} u(x,t)=+\infty$ uniformly for $t\in[0,\infty)$. 
Then $u\geq 0$ on $Q_{[0,\infty)} \setminus M_{[0,\infty)}$. 
\end{lemma}

\begin{proof}
Fix $(x_0,t_0)\in Q_{(0,\infty)} \setminus M_{(0,\infty)}$. 
It suffices to show $u(x_0,t_0)\geq 0$. 
For $\delta>0$, set 
\begin{equation}\label{eq:QBSdel}
\left\{
\begin{aligned}
	&Q_\delta :=\{ (x,t)\in Q_{(0,\infty)}; d(x,M_t) > \delta\}, \\
	&B_{\tau,\delta} := \{ (x,t)\in Q_{(0,\infty)} ; d(x,M_t) > \delta, 
	t=\tau \}\quad (\tau \geq 0), \\
	&S_\delta := \{ (x,t)\in Q_{(0,\infty)} ; d(x,M_t)=\delta \}. 
\end{aligned}
\right.
\end{equation}
Since $\lim_{x\to M_t} u(x,t)=+\infty$ uniformly for $t\in[0,\infty)$, 
there exists $\delta>0$ such that 
$(x_0,t_0)\in Q_{\delta}$ and $u\geq0$ on $S_{\delta}$.  
Then we have 
\[
\left\{
\begin{aligned}
	&\partial_t u -\Delta u + |u|^{p-1}u = 0 
	&& \mbox{ in }Q_{\delta},   \\
	&u\geq  0 
	&& \mbox{ on }S_{\delta},   \\
	&u\geq 0
	&&\mbox{ on }B_{0,\delta}.   
\end{aligned}
\right. 
\]
Since $|u|^{p-1}\geq0$, 
the comparison principle 
(see for instance \cite[Corollary 2.5]{Libook}) yields 
$u\geq 0$ in $\overline{Q}_{\delta}$. 
In particular, we obtain $u(x_0,t_0)\geq0$. 
\end{proof}

\begin{lemma}\label{lem:uniquew}
Let $u_1, u_2\in C^{2,1}(Q_{(0,\infty)} \setminus M_{(0,\infty)}) \cap
C(Q_{[0,\infty)} \setminus M_{[0,\infty)})$ 
be solutions of \eqref{eq:iniabs} with the same initial data $u_0$. 
Suppose either 
\begin{equation}\label{eq:u1u2limw}
	\lim_{x\to M_t} d(x,M_t)^{n-m-2}u_i(x,t)=  c 
	\mbox{ uniformly for }t\in[0,\infty), \quad i=1,2, 
\end{equation}
or 
\[
	\lim_{x\to M_t} d(x,M_t)^\frac{2}{p-1}u_i(x,t)=  L 
	\mbox{ uniformly for }t\in[0,\infty), \quad i=1,2.  
\]
Then $u_1\equiv u_2$ on $Q_{[0,\infty)} \setminus M_{[0,\infty)}$. 
\end{lemma}

\begin{proof}
We only consider the case where $u_1$ and $u_2$ satisfy \eqref{eq:u1u2limw}. 
The other case can be handled similarly. 
We write $Q_\delta$, $B_{\tau,\delta}$ and $S_\delta$ 
as in \eqref{eq:QBSdel}. 
Let $(x_0,t_0)\in Q_{(0,\infty)} \setminus M_{(0,\infty)}$. 
It suffices to show $u_1(x_0,t_0)=u_2(x_0,t_0)$. 
Let $\varepsilon>0$. 
Then by \eqref{eq:u1u2limw}, 
there exists $\delta>0$ such that 
$(x_0,t_0)\in Q_{2\delta}$ and 
\[
	|d(x,M_t)^{n-m-2}u_i(x,t)-c|\leq \varepsilon \quad\mbox{ in }
	(Q_{(0,\infty)} \setminus M_{(0,\infty)})\setminus Q_{2\delta}
\]
for $i=1,2$, and so
\begin{equation}\label{eq:bndrw}
	u_1(x,t)-u_2(x,t)\leq 2\varepsilon d(x,M_t)^{-(n-m-2)} 
	\quad\mbox{ on }S_\delta. 
\end{equation}
Since $u_1$ and $u_2$ are solutions of \eqref{eq:iniabs}, 
we have 
\[
	\partial_t (u_1-u_2) -\Delta (u_1-u_2) + h(x,t) (u_1-u_2) = 0 
	\quad  \mbox{ in }Q_{\delta},  
\]
where $h$ is given by 
\[
	h(x,t) := p\int_0^1 |\theta u_1(x,t) + (1-\theta) u_2(x,t) |^{p-1} d\theta. 
\]
Therefore, by $h\geq0$, $u_1-u_2=0$ on $B_{0,\delta}$, 
the estimate \eqref{eq:bndrw} 
and the maximum principle (see for instance 
\cite[(a), SECTION 3, CHAPTER 2]{Frbook}), we obtain 
$u_1-u_2\leq 2\varepsilon d(x,M_t)^{n-m-2}$ in $\overline{Q}_{\delta}$. 
Since $(x_0,t_0)\in Q_{2\delta}\subset Q_\delta$, we obtain 
\[
	u_1(x_0,t_0)-u_2(x_0,t_0)\leq 2\varepsilon d(x_0,M_{t_0})^{-(n-m-2)}. 
\] 
Hence letting $\varepsilon\to0$ 
gives $u_1(x_0,t_0)\leq u_2(x_0,t_0)$. 
Replacing $u_1-u_2$ by $u_2-u_1$, we also obtain $u_2(x_0,t_0)\leq u_1(x_0,t_0)$. 
Thus $u_1(x_0,t_0)=u_2(x_0,t_0)$, and the proof is complete. 
\end{proof}

\subsection{Weak singularities}\label{subsect:weaks}
Fix $c>0$ and $A\geq 0$. 
In this subsection, we always assume that 
\[
	u_0 \in C(\mathbb{R}^n\setminus M_0), \quad 
	u_0 \geq 0 \mbox{ on }\mathbb{R}^n\setminus M_0, \quad 
	u_0 \in X_{c,A}. 
\]
Let $\alpha$ be a constant such that 
\begin{equation}\label{eq:alcon}
	\max\left\{ 0,p-\frac{2}{n-m-2} \right\} <\alpha<1.
\end{equation}
For a constant $\tilde A>0$, 
define $\overline{u}, 
\underline{u} \in C^{2,1}(Q_{(0,\infty)} \setminus M_{(0,\infty)}) 
\cap C(Q_{[0,\infty)} \setminus M_{[0,\infty)})$ 
by 
\[
	\overline{u}:=  cU +\tilde A,\quad 
	\underline{u}:= cU -U^\alpha - \tilde A. 
\]

\begin{lemma}\label{lem:superw}
For any $\tilde A\geq A$, 
$\overline{u}$ is a super-solution of \eqref{eq:iniabs}. 
\end{lemma}

\begin{proof}
By \eqref{eq:heat} in Proposition \ref{pro:heat} and $U>0$, we have 
$\partial_t \overline{u} -\Delta \overline{u}
	 + |\overline{u}|^{p-1} \overline{u} 
	= (cU + \tilde A)^p > 0$
in $Q_{(0,\infty)} \setminus M_{(0,\infty)}$. 
Since $u_0\in X_{c,A}$, we have 
$u_0\leq cU(\cdot,0)+A\leq \overline{u}(\cdot,0)$ on $\mathbb{R}^n\setminus M_0$. 
The lemma follows. 
\end{proof}

\begin{lemma}\label{lem:subw}
There exists a constant $\tilde A>A$ such that 
$\underline{u}$ is a sub-solution of \eqref{eq:iniabs}. 
\end{lemma}

\begin{proof}
In this proof, we write $d=d(x,M_t)$ and $N=n-m\geq3$. 
Moreover, we use the Landau symbol $O$ in the following way. 
For an interval $I'\subset I$ and 
functions $g$ and $\tilde g$ on $Q_{I'}\setminus M_{I'}$,  we say that 
\begin{center}
$g(x,t)=O(\tilde g(x,t))$ uniformly for $t\in I'$ as $d(x,M_t)\to0$ 
\end{center}
if $\limsup_{x\to M_t}|g(x,t)/\tilde g(x,t)|<+\infty$ uniformly for $t\in I'$, 
where the upper limit is defined in a similar way 
to Remark \ref{rem:unilim}. 
From Proposition \ref{pro:U} with $\underline{t}=-1$, it follows that 
\begin{align}
	&U(x,t)= d^{-(N-2)} + O(d^{-(N-3)}\log(1/d)), \label{eq:Uulas}\\
	&|\nabla U(x,t)|= (N-2)d^{-(N-1)} + O(d^{-(N-2)}) \label{eq:nUulas}
\end{align}
uniformly for $t\in [-1,\infty)$ as $d\to0$.  
In addition, \eqref{eq:Uuni} in Proposition \ref{pro:decay} gives 
\begin{equation}\label{eq:Uto0}
	U(x,t)\to 0 \mbox{ uniformly for }t\in [-1,\infty)
	\mbox{ as } d\to\infty. 
\end{equation}

By \eqref{eq:heat} in Proposition \ref{pro:heat}, we compute that 
\begin{equation}\label{eq:subcal}
\begin{aligned}
	&\partial_t \underline{u} -\Delta \underline{u} 
	+|\underline{u}|^{p-1} \underline{u}  \\
	& =  \partial_t (cU-U^\alpha-\tilde A) -\Delta (cU-U^\alpha-\tilde A) \\
	&\quad +|cU-U^\alpha-\tilde A|^{p-1} (cU-U^\alpha-\tilde A)  \\
	& =
	|cU-U^\alpha-\tilde A|^{p-1} (cU-U^\alpha-\tilde A)
	-\alpha(1-\alpha) U^{\alpha-2}|\nabla U |^2  . 
\end{aligned}
\end{equation}
Let us first consider the case where $x$ is close to $M_t$. 
Since the function $s\mapsto |s|^{p-1}s$ is monotone increasing, 
we have
\[
\begin{aligned}
	\partial_t \underline{u} -\Delta \underline{u} 
	+|\underline{u}|^{p-1} \underline{u}  
	&\leq  c^p U^p - \alpha(1-\alpha) U^{\alpha-2}|\nabla U|^2  \\
	&=  U^p \left( 
	c^p - \alpha(1-\alpha) U^{-(p-\alpha)-2}|\nabla U|^2
	\right). 
\end{aligned}
\]
By \eqref{eq:Uulas} and \eqref{eq:nUulas}, we have 
\[
\begin{aligned}
	&U^{-(p-\alpha)-2}|\nabla U|^2  \\
	&=
	\left( d^{-(N-2)}+ O(d^{-(N-3)}\log(1/d)) \right)^{-(p-\alpha)-2}
	\left( (N-2)d^{-(N-1)}+ O(d^{-(N-2)}) \right)^2  \\
	&=
	(N-2)^2 d^{(N-2)(p-\alpha+2)-2(N-1)}+ O(d^{(N-2)(p-\alpha+2)-2(N-1)+1}\log(1/d))
\end{aligned}
\]
uniformly for $t\in [-1,\infty)$ as $d\to0$. 
Note that \eqref{eq:alcon} gives $(N-2)(p-\alpha+2)-2(N-1)<0$. 
This together with $N=n-m$ shows that 
there exists a constant $r_0>0$ independent of $\tilde A$ such that 
\begin{equation}\label{eq:subne}
\begin{aligned}
	\partial_t \underline{u} -\Delta \underline{u} 
	+|\underline{u}|^{p-1} \underline{u}  \leq 0
\end{aligned}
\end{equation}
for $x\in M_t^{r_0} \setminus M_t$ and $t\in[-1,\infty)$.

We next consider the case where $x$ is far away from $M_t$. 
The computation \eqref{eq:subcal} together with 
the monotonicity of $s\mapsto |s|^{p-1}s$ and $0<\alpha<1$ implies that 
\[
\begin{aligned}
	\partial_t \underline{u} -\Delta \underline{u} 
	+|\underline{u}|^{p-1} \underline{u}  
	&\leq
	U  |cU-U^\alpha|^{p-1} (c-U^{-(1-\alpha)}). 
\end{aligned}
\]
Then \eqref{eq:Uto0} shows that there exists 
a constant $R_0>r_0$ independent of $\tilde A$ such that 
\eqref{eq:subne}
holds for $x\in \mathbb{R}^n \setminus M_t^{R_0}$ and $t\in[-1,\infty)$.

Finally, we consider the intermediate region. From \eqref{eq:subcal} 
and Proposition \ref{pro:interm} with 
$c_0=r_0$, $C_0=R_0$ and $\underline{t}=-1$, it follows that  
\[
\begin{aligned}
	\partial_t \underline{u} -\Delta \underline{u} 
	+|\underline{u}|^{p-1} \underline{u}  
	\leq 
	|cU-\tilde A|^{p-1} (cU-\tilde A) 
	\leq 
	|cU-\tilde A|^{p-1} (cC-\tilde A) 
\end{aligned}
\]
for $x\in M_t^{R_0}\setminus M_t^{r_0}$ and $t\in[-1,\infty)$ 
with some $C>0$ independent of $\tilde A$. 
Hence there exists $\tilde A\geq A$ such that 
\eqref{eq:subne} holds for $x\in M_t^{R_0}\setminus M_t^{r_0}$ and $t\in[-1,\infty)$. 
Thus, 
\[
	\partial_t \underline{u} -\Delta \underline{u} 
	+|\underline{u}|^{p-1} \underline{u} \leq 0
	\quad \mbox{ in }Q_{(0,\infty)} \setminus M_{(0,\infty)}. 
\]
Then, since $u_0\in X_{c,A}$ and $\tilde A\geq A$, we can estimate that 
\[
	u_0\geq cU(\cdot,0)-A
	\geq cU(\cdot,0) - U(\cdot,0)^\alpha - \tilde A =\underline{u}(\cdot,0)
\]
on $\mathbb{R}^n\setminus M_0$. 
The proof is complete. 
\end{proof}

We are now in a position to prove Theorem \ref{th:ex} (i). 

\begin{proof}[Proof of Theorem \ref{th:ex} (i)] 
By Lemmas \ref{lem:superw} and \ref{lem:subw}, 
we have a super-solution $\overline{u}$ and a sub-solution $\underline{u}$ 
provided that $\tilde A>0$ is sufficiently large. 
Moreover, 
\[
	\underline{u} \leq cU-\tilde A \leq cU+\tilde A =\overline{u} 
	\quad \mbox{ on }Q_{[0,\infty)} \setminus M_{[0,\infty)}. 
\]
Hence from Theorem \ref{th:unbddss}, it follows that 
\eqref{eq:iniabs} admits a solution 
$u_{c,A}\in C^{2,1}(Q_{(0,\infty)} \setminus M_{(0,\infty)}) 
\cap C(Q_{[0,\infty)} \setminus M_{[0,\infty)})$ satisfying 
\[
	\underline{u}\leq u_{c,A} \leq \overline{u}
	\quad \mbox{ on }Q_{[0,\infty)} \setminus M_{[0,\infty)}.
\]
Then \eqref{eq:Uul} in Proposition \ref{pro:U} and $\alpha<1$ give 
\[
\begin{aligned}
	&\liminf_{x\to M_t} d(x,M_t)^{n-m-2}u_{c,A}(x,t) \geq 
	\liminf_{x\to M_t} d(x,M_t)^{n-m-2}(cU(x,t) -U(x,t)^\alpha) 
	=c, \\
	&\limsup_{x\to M_t} d(x,M_t)^{n-m-2}u_{c,A}(x,t) 
	\leq  \limsup_{x\to M_t} c d(x,M_t)^{n-m-2}U(x,t) =c
\end{aligned}
\]
uniformly for $t\in[0,\infty)$. 
Thus, $u_{c,A}$ satisfies \eqref{eq:asymnm2uni}. 
This together with $u_0\geq0$ and Lemma \ref{lem:posiw} yields 
$u_{c,A}\geq0$ on $Q_{[0,\infty)} \setminus M_{[0,\infty)}$. 
Finally, Lemma \ref{lem:uniquew} shows uniqueness, 
and the proof is complete. 
\end{proof}

\subsection{Strong singularities}\label{subsect:strongs}
Fix $A\geq 0$. 
In this subsection, we always assume that 
\[
	u_0 \in C(\mathbb{R}^n\setminus M_0), \quad 
	u_0 \geq 0 \mbox{ on }\mathbb{R}^n\setminus M_0, \quad 
	u_0 \in Y_A. 
\]
Define 
\begin{equation}\label{eq:betadef}
	\beta:=\frac{2}{(n-m-2)(p-1)} \quad (>1).
\end{equation}
Let $\beta'$ be a constant such that 
\begin{equation}\label{eq:betapr}
	\beta-\frac{1}{n-m-2} <\beta' <\beta. 
\end{equation}
Note that $\beta'>0$. 
Let $\gamma$ be a constant such that 
\[
	0 <\gamma <\min\left\{1,\beta', \frac{(n-2)\beta p -2}{n-2}\right\} \quad (<\beta).  
\]
For a constant $\tilde A>0$, define $\overline{u}, 
\underline{u} \in C^{2,1}(Q_{(0,\infty)}\setminus M_{(0,\infty)}) 
\cap C(Q_{[0,\infty)}\setminus M_{[0,\infty)})$ 
by 
\[
	\overline{u}:=  LU^\beta + U^{\beta'} + U^\gamma + \tilde A,\quad 
	\underline{u}:= LU^\beta - U^{\beta'} - U^\gamma - \tilde A. 
\]

\begin{lemma}\label{lem:supers}
There exists a constant $\tilde A\geq A$ such that 
$\overline{u}$ is a super-solution of \eqref{eq:iniabs}. 
\end{lemma}

\begin{proof}
In this proof, we write $d=d(x,M_t)$ and $N=n-m\geq3$. 
The property \eqref{eq:heat} in Proposition \ref{pro:heat} yields 
\begin{equation}\label{eq:ovuNfi}
\begin{aligned}
	\partial_t \overline{u} -\Delta \overline{u} + \overline{u}^p   
	&=
	- L \beta(\beta-1) U^{\beta-2}|\nabla U|^2
	-\beta'(\beta'-1) U^{\beta'-2}|\nabla U|^2  \\
	&\quad +\gamma(1-\gamma) U^{\gamma-2}|\nabla U|^2
	+ ( L U^\beta+ U^{\beta'} +U^\gamma+ \tilde A )^p. 
\end{aligned}
\end{equation}

We first consider the case where $x$ is close to $M_t$. 
Since $0<\gamma<1$ and the function $s\mapsto |s|^{p-1}s$ is monotone increasing, 
we have 
\begin{equation}\label{eq:ovuN}
\begin{aligned}
	\partial_t \overline{u} -\Delta \overline{u}
	 + \overline{u}^p 
	&\geq 
	- L \beta(\beta-1) U^{\beta-2}|\nabla U|^2 -\beta'(\beta'-1) U^{\beta'-2}|\nabla U|^2 \\
	&\quad +L^p U^{\beta p} + pL^{p-1} U^{\beta(p-1) +\beta'}. 
\end{aligned}
\end{equation}
In what follows, we use the Landau symbol $O$ 
as in the proof of Lemma \ref{lem:subw}. 
By \eqref{eq:Uulas}, \eqref{eq:nUulas}, \eqref{eq:betadef} and 
$-2/(p-1)-2=-2p/(p-1)$, we have 
\begin{equation}\label{eq:betaUne}
\begin{aligned}
	U^{\beta-2}|\nabla U|^2 
	&= \left( d^{-(N-2)} + O( d^{-(N-3)} \log(1/d) ) \right)^{\beta -2}  \\
	&\quad \times \left( (N-2) d^{-(N-1)} + O( d^{-(N-2)} ) \right)^2  \\
	&= \left( d^{-(N-2)(\beta-2)} + 
	O( d^{-(N-2)(\beta-3)-(N-3)} \log(1/d) ) \right) \\
	&\quad \times \left( (N-2)^2 d^{-2(N-1)} + O( d^{-(N-1)-(N-2)} ) \right)  \\ 
	&= (N-2)^2 d^{-\frac{2p}{p-1}} 
	+ O( d^{-\frac{2p}{p-1}+1} \log(1/d) )
\end{aligned}
\end{equation}
uniformly for $t\in [-1,\infty)$ as $d\to0$. 
Since $-(N-2)\beta'-1>-2p/(p-1)+1$ by \eqref{eq:betapr}, we also have 
\begin{equation}\label{eq:betaprUne}
\begin{aligned}
	U^{\beta'-2}|\nabla U|^2 
	&= \left( d^{-(N-2)(\beta'-2)} + 
	O( d^{-(N-2)(\beta'-3)-(N-3)} \log(1/d) ) \right) \\
	&\quad \times \left( (N-2)^2 d^{-2(N-1)} + O( d^{-(N-1)-(N-2)} ) \right)  \\ 
	&= (N-2)^2 d^{-(N-2)\beta'-2} 
	+ O( d^{-(N-2)\beta'-1} \log(1/d) )  \\
	&= (N-2)^2 d^{-(N-2)\beta'-2} 
	+ O( d^{-\frac{2p}{p-1}+1} \log(1/d) )
\end{aligned}
\end{equation}
uniformly for $t\in [-1,\infty)$ as $d\to0$. 
By \eqref{eq:Uulas} and \eqref{eq:betadef}, we see that 
\begin{equation}\label{eq:Ubetap}
\begin{aligned}
	U^{\beta p} &= \left( d^{-(N-2)} + O(d^{-(N-3)} \log(1/d) )\right)^{\beta p}  \\
	& = d^{-(N-2)\beta p} + O(d^{-(N-2)(\beta p -1 ) -(N-3)} \log(1/d)  )  \\
	& = d^{-\frac{2p}{p-1}} + O\left(d^{-\frac{2p}{p-1}+1} \log(1/d) \right) 
\end{aligned}
\end{equation}
and 
\begin{equation}\label{eq:Ubetapr}
\begin{aligned}
	U^{\beta(p-1)+\beta'} 
	& = d^{-(N-2)\beta'-2} + O\left(d^{-(N-2)\beta'-1} \log(1/d) \right) \\
	& = d^{-(N-2)\beta'-2} + O\left(d^{-\frac{2p}{p-1}+1} \log(1/d) \right) 
\end{aligned}
\end{equation}
uniformly for $t\in [-1,\infty)$ as $d\to0$. 
By \eqref{eq:Ldefi}, we have $L^{p-1}=\beta(\beta-1)(N-2)^2$. 
Hence plugging \eqref{eq:betaUne}, \eqref{eq:betaprUne}, 
\eqref{eq:Ubetap} and \eqref{eq:Ubetapr} into \eqref{eq:ovuN} gives 
\[
\begin{aligned}
	\partial_t \overline{u} -\Delta \overline{u} + \overline{u}^p
	& \geq 
	\left( L^p-L\beta(\beta-1) (N-2)^2 \right)d^{-\frac{2p}{p-1}} 
	+O(d^{-\frac{2p}{p-1}+1} \log(1/d)) \\
	&\quad + \left( pL^{p-1} -\beta'(\beta'-1) (N-2)^2  \right) d^{-(N-2)\beta'-2}  \\
	&=
	\left( p\beta(\beta-1) -\beta'(\beta'-1) \right) (N-2)^2 d^{-(N-2)\beta'-2} \\
	&\quad +O(d^{-\frac{2p}{p-1}+1} \log(1/d))
\end{aligned}
\]
uniformly for $t\in [-1,\infty)$ as $d\to0$. 
By \eqref{eq:betapr}, we have 
$p\beta(\beta-1)>\beta'(\beta'-1)$ and $-(N-2)\beta'-2<-2p/(p-1)+1$. 
Therefore, there exists $r_1>0$ independent of $\tilde A$ such that  
\begin{equation}\label{eq:supcal1}
	\partial_t \overline{u} -\Delta \overline{u} + \overline{u}^p \geq 0 
\end{equation}
for $x\in M_t^{r_1} \setminus M_t$ and $t\in[-1,\infty)$. 

We next consider the case where $x$ is far away from $M_t$.
By \eqref{eq:ovuNfi}, we have 
\[
\begin{aligned}
	\partial_t \overline{u} -\Delta \overline{u} + \overline{u}^p 
	&\geq 
	U^{\gamma-2}|\nabla U|^2
	\left( 
	\gamma(1-\gamma)
	- L \beta(\beta-1) U^{\beta-\gamma}
	- \beta'|\beta'-1| U^{\beta'-\gamma} 
	\right). 
\end{aligned}
\]
Then by  $\beta>\beta'>\gamma$ and \eqref{eq:Uto0}, 
there exists $R_1>r_1$ independent of $\tilde A$ such that 
\eqref{eq:supcal1}
holds for $x\in \mathbb{R}^n\setminus M_t^{R_1}$ and $t\in[-1,\infty)$. 

Finally, we consider the intermediate region. 
By \eqref{eq:ovuNfi}, $0<\gamma<1$, 
\eqref{eq:nUuni} in Proposition \ref{pro:decay} and Proposition \ref{pro:interm} with 
$c_0=r_1$, $C_0=R_1$ and $\underline{t}=-1$, 
we obtain 
\[
\begin{aligned}
	\partial_t \overline{u} -\Delta \overline{u} + \overline{u}^p   
	&\geq 
	\tilde A^p - L \beta(\beta-1) U^{\beta-2}|\nabla U|^2
	-\beta'(\beta'-1) U^{\beta'-2}|\nabla U|^2  \\
	&\geq 
	\tilde A^p - C
\end{aligned}
\]
for $x\in M_t^{R_1}\setminus M_t^{r_1}$ and $t\in[-1,\infty)$
with some constant $C>0$ depending on $r_1$ and $R_1$ but not on $\tilde A$. 
Therefore there exists $\tilde A\geq A$ such that 
\eqref{eq:supcal1} holds for $x\in M_t^{R_1}\setminus M_t^{r_1}$ and $t\in[-1,\infty)$. 
Thus, 
\[
	\partial_t \overline{u} -\Delta \overline{u} + \overline{u}^p \geq 0 
	\quad \mbox{ in }Q_{(0,\infty)} \setminus M_{(0,\infty)}. 
\]
Since $u_0\in Y_{A}$, 
$u_0\leq LU(\cdot,0)^\beta +A\leq \overline{u}(\cdot,0)$ on $\mathbb{R}^n\setminus M_0$. 
Then the lemma follows. 
\end{proof}

\begin{lemma}\label{lem:subs}
There exists a constant $\tilde A>A$ such that 
$\underline{u}$ is a sub-solution of \eqref{eq:iniabs}. 
\end{lemma}

\begin{proof}
In this proof, we write $d=d(x,M_t)$ and $N=n-m$. 
By \eqref{eq:heat}, we have 
\begin{equation}\label{eq:supovuNfi}
\begin{aligned}
	&\partial_t \underline{u} -\Delta \underline{u} + |\underline{u}|^{p-1}\underline{u}   \\
	&=
	- L \beta(\beta-1) U^{\beta-2}|\nabla U|^2
	+\beta'(\beta'-1) U^{\beta'-2}|\nabla U|^2 
	-\gamma(1-\gamma) U^{\gamma-2}|\nabla U|^2 \\
	&\quad + | L U^\beta- U^{\beta'} -U^\gamma- \tilde A |^{p-1}
	( L U^\beta- U^{\beta'} -U^\gamma- \tilde A  ). 
\end{aligned}
\end{equation}

We first consider the case 
where $x$ is close to $M_t$ or $x$ is far away from $M_t$. 
We estimate that
\begin{equation}\label{eq:supovuN}
\begin{aligned}
	\partial_t \underline{u} -\Delta \underline{u}
	 + |\underline{u}|^{p-1}\underline{u} 
	&\leq 
	- L \beta(\beta-1) U^{\beta-2}|\nabla U|^2 +\beta'(\beta'-1) U^{\beta'-2}|\nabla U|^2 \\
	&\quad + | L U^\beta- U^{\beta'} |^{p-1}
	( L U^\beta- U^{\beta'} ). 
\end{aligned}
\end{equation}
By \eqref{eq:Uulas} and \eqref{eq:nUulas}, we see that 
\begin{equation}\label{eq:LUbetap}
\begin{aligned}
	&| L U^\beta- U^{\beta'} |^{p-1}
	( L U^\beta- U^{\beta'} )  \\
	&= \left( Ld^{-\frac{2}{p-1}} - d^{-(N-2)\beta'} 
	+ O(d^{-\frac{2}{p-1}+1} \log(1/d) )\right)^p  \\
	& =  L^p d^{-\frac{2p}{p-1}} 
	- pL^{p-1} d^{-(N-2)\beta'-2} 
	+ O(d^{-\frac{2p}{p-1}+1} \log(1/d) )
\end{aligned}
\end{equation}
uniformly for $t\in [-1,\infty)$ as $d\to0$. 
Plugging \eqref{eq:betaUne}, \eqref{eq:betaprUne} and \eqref{eq:LUbetap} 
into \eqref{eq:supovuN} gives 
\[
\begin{aligned}
	\partial_t \underline{u} -\Delta \underline{u} + |\underline{u}|^{p-1}\underline{u}
	& \leq 
	\left( L^p-L\beta(\beta-1) (N-2)^2 \right)d^{-\frac{2p}{p-1}} 
	+O(d^{-\frac{2p}{p-1}+1} \log(1/d)) \\
	&\quad - \left( pL^{p-1} -\beta'(\beta'-1) (N-2)^2  \right) d^{-(N-2)\beta'-2}  \\
	&=
	- \left( p\beta(\beta-1) -\beta'(\beta'-1) \right) (N-2)^2 d^{-(N-2)\beta'-2} \\
	&\quad +O(d^{-\frac{2p}{p-1}+1} \log(1/d))
\end{aligned}
\]
uniformly for $t\in [-1,\infty)$ as $d\to0$. 
Since $p\beta(\beta-1)>\beta'(\beta'-1)$ and $-(N-2)\beta'-2<-2p/(p-1)+1$, 
there exists $r_2>0$ independent of $\tilde A$ such that
\eqref{eq:subne} holds 
for $x\in M_t^{r_2} \setminus M_t$ and $t\in[-1,\infty)$. 
On the other hand, by \eqref{eq:supovuNfi}, we have 
\[
\begin{aligned}
	\partial_t \underline{u} -\Delta \underline{u} + |\underline{u}|^{p-1}\underline{u}   
	&\leq 
	-U^{\gamma-2}|\nabla U|^2
	\left( \gamma(1-\gamma)  -  \beta'(\beta'-1) U^{\beta'-\gamma}  \right) \\
	&\quad - U^\gamma| L U^\beta- U^\gamma |^{p-1} ( 1- L U^{\beta-\gamma} ). 
\end{aligned}
\]
This together with $\beta>\beta'>\gamma$ and \eqref{eq:Uto0} proves that 
there exists $R_2>r_2$ independent of $\tilde A$ such that 
\eqref{eq:subne} holds 
for $x\in \mathbb{R}^n\setminus M_t^{R_2}$ and $t\in[-1,\infty)$. 

Next, let us consider the intermediate region. 
By \eqref{eq:supovuNfi}, \eqref{eq:nUuni} in Proposition \ref{pro:decay} 
and Proposition \ref{pro:interm} with 
$c_0=r_2$, $C_0=R_2$ and $\underline{t}=-1$, 
we have 
\[
\begin{aligned}
	\partial_t \underline{u} -\Delta \underline{u} + |\underline{u}|^{p-1}\underline{u} 
	\leq 
	C  + | C- \tilde A |^{p-1}( C - \tilde A  ) 
\end{aligned}
\]
for some constant $C>0$ depending on $r_2$ and $R_2$ but not on $\tilde A$. 
Then there exists $\tilde A\geq A$ such that 
\eqref{eq:subne} holds 
for $x\in M_t^{R_2} \setminus M_t^{r_2}$ and $t\in[-1,\infty)$. 
Hence, 
\[
	\partial_t \underline{u} -\Delta \underline{u} 
	+|\underline{u}|^{p-1} \underline{u} \leq 0
	\quad \mbox{ in }Q_{(0,\infty)} \setminus M_{(0,\infty)}. 
\]
Since $u_0\in Y_{A}$, 
$u_0\geq \underline{u}(\cdot,0)$ on $\mathbb{R}^n\setminus M_0$. 
This completes the proof. 
\end{proof}

We are now in a position to prove Theorem \ref{th:ex} (ii). 

\begin{proof}[Proof of Theorem \ref{th:ex} (ii)] 
This theorem can be proved by the same argument as in 
the proof of Theorem \ref{th:ex} (i). 
Thus we only give an outline. 
By Lemmas \ref{lem:supers} and \ref{lem:subs}, 
we have a super-solution $\overline{u}$ and a sub-solution $\underline{u}$ 
such that $\underline{u}\leq \overline{u}$
on $Q_{[0,\infty)} \setminus M_{[0,\infty)}$. 
Using Theorem \ref{th:unbddss}, we obtain a solution 
$u_A\in C^{2,1}(Q_{(0,\infty)} \setminus M_{(0,\infty)}) 
\cap C(Q_{[0,\infty)} \setminus M_{[0,\infty)})$ 
of \eqref{eq:iniabs} such that
$\underline{u}\leq u_A\leq \overline{u}$
on $Q_{[0,\infty)} \setminus M_{[0,\infty)}$. 
By \eqref{eq:Uul} 
and the definitions of $\overline{u}$ and $\underline{u}$, 
we can see that $u_A$ satisfies \eqref{eq:sasymnm2uni}. 
This together with $u_0\geq0$ and Lemma \ref{lem:posiw} gives 
$u_A\geq0$ on $Q_{[0,\infty)} \setminus M_{[0,\infty)}$. 
Lemma \ref{lem:uniquew} shows uniqueness. 
The proof is complete. 
\end{proof}

\section{Geometric facts}\label{sec:geom}
In this section, we prepare geometric facts needed in this paper. 
From Subsection \ref{LangChar} to \ref{IoTNB}, we consider a fixed submanifold, that is, the time-independent case. 
In Subsection \ref{tdepLC}, we generalize results obtained in those subsections to a one parameter family of submanifolds, that is, 
the time-dependent case. 

In Subsection \ref{LangChar}, we introduce Langer charts and the embedding constant, and prove some properties. 
In Subsection \ref{ATNB}, we consider a tubular neighborhood and a normal bundle of a submanifold, 
and prove that these two are diffeomorphic by the normal exponential map when 
the width of each tube is sufficiently small, see Proposition \ref{normalanddelta}. 
In Subsection \ref{IoTNB}, we give an estimate 
for the integral of a function 
over a tubular neighborhood, see Proposition \ref{tbnint}. 
In Subsection \ref{tdepLC}, we state that the above results also hold 
for time-dependent submanifolds. 

Before going further, we introduce notation. 
We write $B_{R}$ as the ball centered at the origin of radius $R$. 
Let $U\subset \mathbb{R}^{m}$ be an open set and let 
$F:U\to \mathbb{R}^{n}$ be a sufficiently smooth map. 
We denote its $k$-th derivative at $x\in U$ by 
$D^{k}F(x):\mathbb{R}^{m}\times \dots \times \mathbb{R}^{m}\to \mathbb{R}^{n}$. 
For partial derivatives, we sometimes denote $D^{k}F(x)(e_{i_{1}},\dots,e_{i_{k}})$ by 
$\partial_{x_{i_{1}}}\dots\partial_{x_{i_{k}}}F(x)$, 
where $e_{1},\dots,e_{m}$ is the standard basis of $\mathbb{R}^{m}$. 

\subsection{Langer charts}\label{LangChar}
Let $M \subset \mathbb{R}^{n}$ be a connected compact 
$m$-dimensional smooth embedded submanifold without boundary. 
We denote the tangent space of $M$ at $x\in M$ by $T_{x}M\subset \mathbb{R}^{n}$. 
First, we introduce a nice parametrization of $M$, 
the so-called Langer chart, following \cite[Section 2]{La85}. 
Fix $x \in M$ and $E_{x} \in \mathrm{GL}(n,\mathbb{R})$ such that 
\[T_{x} M=\{E_{x}\bar{y}+x ; \bar{y}=(y_{1},\dots,y_{m},0,\dots,0)\in \mathbb{R}^{m}\times \mathbb{R}^{n-m}\}. \]
Then, a liner isomorphism $\mathbb{R}^{m}\cong T_{x} M$ is given by $y\mapsto E_{x}(y,0)+x$. 
Put 
\[\overline{\lambda}(E_{x}):=\sup_{v\neq 0}\frac{|E_{x}v|}{|v|}, \quad
\underline{\lambda}(E_{x}):=\inf_{v\neq 0}\frac{|E_{x}v|}{|v|}. \]
It is clear that $\overline{\lambda}(E_{x})\geq \underline{\lambda}(E_{x})=(\overline{\lambda}(E_{x}^{-1}))^{-1}>0$. 
Define $G:M\to \mathbb{R}^{m}$ by 
\[G(x'):=\mathrm{proj}_{\mathbb{R}^{m}}(E_{x}^{-1}(x'-x)), \]
where $\mathrm{proj}_{\mathbb{R}^{m}}:\mathbb{R}^{m}\times \mathbb{R}^{n-m}\to \mathbb{R}^{m}$ is the projection to the first factor. 
It is clear that $G(x)=0$ and $DG(x)$ is the identity map by the isomorphism $\mathbb{R}^{m}\cong T_{x} M$. 
Hence, by the inverse function theorem, there exists an embedding $\psi:B_{R}\to M\subset \mathbb{R}^{n}$ 
such that $G\circ \psi(y)=y$ for all $y\in B_{R}$. 
We call $\psi:B_{R}\to M\subset \mathbb{R}^{n}$ a \emph{Langer chart} around $x$ and call $R$ the radius of the Langer chart, see Figure \ref{fig:2}. 
A Langer chart depends on the choice of $E_{x}\in \mathrm{GL}(n,\mathbb{R})$. 
We remark that if $\psi:B_{R}\to M\subset \mathbb{R}^{n}$ is a Langer chart around $x$ then its restriction to $B_{R'}$ with $R'<R$ is also a Langer chart around $x$. 
\begin{figure}[tb]
\includegraphics[bb=30 321 564 523, clip, scale=0.67]{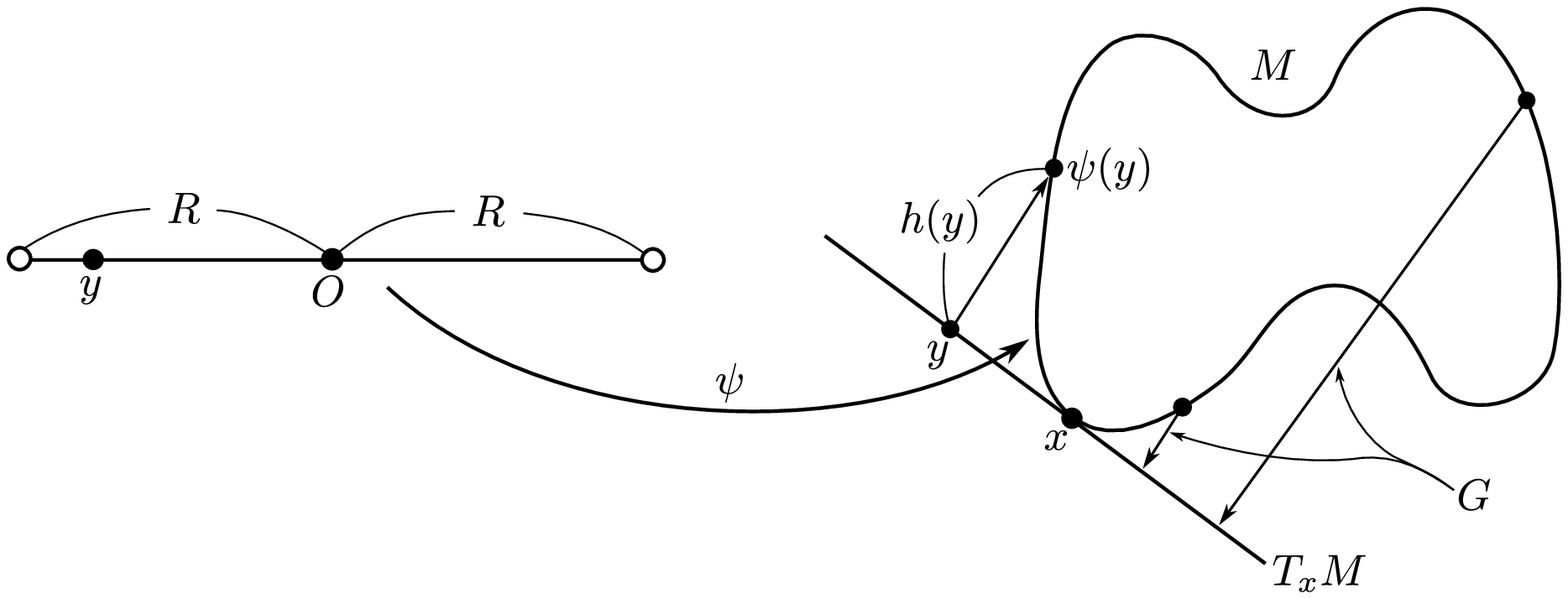}
\caption{A Langer Chart}
\label{fig:2}
\end{figure}
We define the \emph{height function} $h:B_{R}\to \mathbb{R}^{n-m}$ associated with $\psi$ by 
\[h(y):=\mathrm{proj}_{\mathbb{R}^{n-m}}(E_{x}^{-1}(\psi(y)-x)), \]
where $\mathrm{proj}_{\mathbb{R}^{n-m}}:\mathbb{R}^{m}\times \mathbb{R}^{n-m}\to \mathbb{R}^{m}$ is the projection to the second factor. 
Since 
$\mathrm{proj}_{\mathbb{R}^{m}}(E_{x}^{-1}(\psi(y)-x))=G\circ \psi(y)=y$, 
we have 
\[\psi(y)=E_{x}(y,h(y))+x\]
for all $y\in B_{R}$. Clearly, $h(0)=0$ and $D h(0)=0$. 
Moreover, the following holds. 

\begin{lemma}\label{CooperLem2.1}
The height function $h:B_{R}\to \mathbb{R}^{n-m}$ satisfies 
\[|D^{2}h|^{2}\leq \left(\overline{\lambda}(E_{x})^2(1+|D h|^2) \right)^{3}|\mathrm{II}|^2\]
on $B_{R}$, where $\mathrm{II}$ is the second fundamental form of $M$. 
\end{lemma}

Lemma \ref{CooperLem2.1} was essentially proved in \cite[Lemma 2.1]{Co11} 
under the assumption $E_{x}\in \mathrm{O}(n,\mathbb{R})$, 
especially $\overline{\lambda}(E_{x})=1$. 
For general $E_{x}\in \mathrm{GL}(n,\mathbb{R})$, the proof is similar.

\begin{definition}
Fix $R>0$ and $\alpha>0$. 
Then, $M$ is called an $(R,\alpha)$-embedding if for each point $x\in M$, 
there exists a Langer chart $\psi:B_{R}\to M\subset \mathbb{R}^{n}$ around $x$ of radius $R$ 
such that the height function $h:B_{R}\to \mathbb{R}^{n-m}$ associated with $\psi$ satisfies 
$|D h(y)|\leq \alpha$ for any $y\in B_{R}$. 
\end{definition}

Then, by Lemma \ref{CooperLem2.1} and the same argument as in the proof of \cite[Lemma 2.3]{Co11}, the following holds.

\begin{lemma}\label{CooperLem2.3}
Assume that $\max_{x\in M}|\mathrm{II}_{x}|\leq K$ and $\max_{x\in M}\overline{\lambda}(E_{x})\leq\overline{\lambda}$ for some constants $K>0$ and $\overline{\lambda}>0$. 
If $R>0$ and $\alpha>0$ satisfy
\begin{equation}\label{CooperLem2.3c}
R\leq \alpha \left({\overline{\lambda}}^2(1+\alpha^2)\right)^{-3/2}K^{-1},
\end{equation}
then $M$ is an $(R,\alpha)$-embedding. 
\end{lemma}

Assume that 
\begin{align}\label{Klamlam}
\max_{x\in M}|\mathrm{II}_{x}|\leq K, \quad \max_{x\in M}\overline{\lambda}(E_{x})\leq\overline{\lambda}, \quad 
\min_{x\in M}\overline{\lambda}(E_{x})\geq\underline{\lambda}
\end{align}
for some constants $K>0$, $\overline{\lambda}>0$ and $\underline{\lambda}>0$. 
Let $R>0$ and $\alpha>0$ satisfy \eqref{CooperLem2.3c}. 
Then, by Lemma \ref{CooperLem2.3}, $M$ is an $(R,\alpha)$-embedding. 
Hence, for each point $x\in M$, 
there exists a Langer chart $\psi:B_{R}\to M\subset \mathbb{R}^{n}$ around $x$ 
such that the height function $h:B_{R}\to \mathbb{R}^{n-m}$ associated with $\psi$ satisfies 
\begin{align}\label{dh}
|D h(y)|\leq \alpha
\end{align}
for any $y\in B_{R}$. 
Moreover, by Lemma \ref{CooperLem2.1}, $h$ also satisfies 
\begin{align}\label{ddh}
|D^2 h|\leq \overline{\lambda}^3(1+\alpha^2)^{3/2}K.
\end{align}
In this setting, the Langer chart satisfies the following properties.

\begin{lemma}\label{LanPro}
For any $y\in B_{R}$ and $\eta\in\mathbb{R}^{m}$,
\[
\left\{
\begin{aligned}
&\underline{\lambda}|\eta|\leq |D\psi(y)\eta| \leq\overline{\lambda}\sqrt{1+\alpha^2}|\eta|,\\
&\underline{\lambda}^{m}\leq J\psi(y) \leq \overline{\lambda}^{m}(1+\alpha^{2})^{m/2},\\
&|D^2\psi(y)|\leq \overline{\lambda}^4 (1+\alpha^2)^{3/2}K,\\
&|D J\psi(y)|\leq C_{1}(m) \underline{\lambda}^{-m}\overline{\lambda}^{7}(1+\alpha^2)^{3}K, 
\end{aligned}
\right.
\]
where $J\psi:=\sqrt{\det({}^t(D \psi)(D \psi))}$ 
and $C_{1}(m)>0$ is a constant. 
\end{lemma}

\begin{proof}
Since $\psi(y)=E_{x}(y,h(y))+x$, 
we have
\begin{align}\label{dpsi}
\partial_{y_{j}}\psi(y)=E_{x}(\bar{e}_{j},\partial_{y_{j}}h(y)), 
\end{align}
where $\{\bar{e}_{1},\dots,\bar{e}_{m}\}$ is the standard basis of $\mathbb{R}^{m}$. 
Hence, the $(k,j)$-component of the matrix ${}^t(D \psi(y))(D \psi(y))$ is given by 
${}^{t}(\bar{e}_{k},\partial_{y_{k}}h(y)){}^{t}E_{x}E_{x}(\bar{e}_{j},\partial_{y_{j}}h(y))$. 
Let $\mu$ be an eigenvalue of ${}^t(D \psi(y))(D \psi(y))$ 
and let $\eta$ be its eigenvector. 
Then, these satisfy
\[\sum_{j=1}^{m}\left({}^{t}(\bar{e}_{k},\partial_{y_{k}}h(y)){}^{t}E_{x}E_{x}(\bar{e}_{j},\partial_{y_{j}}h(y))\right)\eta_{j}=\mu\eta_{k}.\]
Multiplying both sides by $\eta_{k}$ and summing over $k$, we have 
$|E_{x}(\eta,D h(y)\eta)|^2=\mu|\eta|^2$. 
By \eqref{dh}, we have
$\underline{\lambda}|\eta|\leq |E_{x}(\eta,D h(y)\eta)|\leq \overline{\lambda}|\eta|\sqrt{1+\alpha^{2}}$. 
Thus, we obtain 
$\underline{\lambda}^{2}\leq \mu \leq \overline{\lambda}^{2}(1+\alpha^{2})$. 
This shows the first inequality. 
Since $\det({}^t(D \psi(y))(D \psi(y)))$ is the product of all eigenvalues, 
we obtain the second inequality. 
Form (\ref{ddh}) and (\ref{dpsi}), it follows that 
$|D^2\psi(y)|\leq \overline{\lambda}^4 (1+\alpha^2)^{3/2}K$. 
For the last inequality, note that 
\[\partial_{y_{i}}J\psi(y)=
\frac{1}{2}\left(\det({}^t(D \psi(y))(D \psi(y)))\right)^{-1/2}
\partial_{y_{j}}\left(\det({}^t(D \psi(y))(D \psi(y)))\right).\]
By (\ref{dh}), (\ref{ddh}) and (\ref{dpsi}), one can see that there exists 
a constant $C_{1}(m)>0$ such that 
$|D J\psi(y)|\leq C_{1}(m) \underline{\lambda}^{-m}\overline{\lambda}^{7}(1+\alpha^2)^{3}K$. 
Then the proof is complete. 
\end{proof}

Next, we introduce the embedding constant of $M$. 
Let $\iota:M\to \mathbb{R}^{n}$ be the inclusion map. 
We denote by $dx^{2}$ the standard Riemannian metric on $\mathbb{R}^{n}$. 
Then, a Riemannian metric $g:=\iota^{*}(dx^2)$ on $M$ is called the induced metric. 
We define the induced distance $d_{g}$ on $(M,g)$ in the standard manner. 
Then, following \cite[Definition 4.2]{Co11}, the embedding constant is defined as follows. 
\begin{definition}\label{embconst}
We define the embedding constant $\kappa(M)$ of $M$ by
\[\kappa(M):=\sup\left\{\frac{d_{g}(p,q)}{|p-q|};  p,q\in M \mbox{ with }p\neq q  \right\}. \]
\end{definition}

\begin{proposition}
$1\leq \kappa(M)<\infty$. 
\end{proposition}

\begin{proof}
The first inequality is trivial, since the induced distance $d_{g}$ always satisfies $|p-q|\leq d_{g}(p,q)$. 
To obtain a contradiction, assume $\kappa(M)=\infty$. 
Then, there exist sequences $p_{i},q_{i} \in M$ such that $p_{i}\neq q_{i}$ and 
\begin{align}\label{contradist}
\frac{d_{g}(p_{i},q_{i})}{|p_{i}-q_{i}|}\to \infty \quad \mbox{ as } i\to \infty. 
\end{align} 
Since $M$ is compact, $d_{g}$ is a bounded function. 
Thus, it must satisfy that $|p_{i}-q_{i}| \to 0$ as $i\to \infty$. 
By the compactness of $M$, there exist $\bar{p}, \bar{q} \in M$ and subsequences $p_{k}$, $q_{k}$ such that 
$p_{k}\to \bar{p}$ and $q_{k}\to \bar{q}$. 
Then, by the condition $|p_{i}-q_{i}| \to 0$ as $i\to \infty$, we have $\bar{p}=\bar{q}$. 

Since $M$ is compact, $M$ is an 
$(R,\alpha)$-embedding for some $R>0$ and $\alpha>0$ by Lemma \ref{CooperLem2.3}. 
Without loss of generality, we can assume $E_{x}\in \mathrm{O}(n,\mathbb{R})$ for any $x\in M$. 
Let $\psi:B_{R}\to M$ be a Langer chart around $\bar{p}$ $(=\bar{q})$. 
Take points $y_{1},y_{2} \in B_{R}$. 
Put $c(t):=(1-t)y_{1}+ty_{2} \in B_{R}$ 
and $\tilde{c}(t):=\psi(c(t))=E_{\bar{p}}(c(t),h(c(t)))+\bar{p}$  for $t\in[0,1]$. 
Then, $\tilde{c}(t)$ is a curve joining $\psi(y_{1})$ and $\psi(y_{2})$, 
and its length with respect to $g=\iota^{*}(dx^2)$ is given by
\begin{equation}\label{Lgc}
\begin{aligned}
L_{g}(\tilde{c})&=\int_{0}^{1}|D\psi(c(t))(y_{2}-y_{1})|dt  \\
&=\int_{0}^{1}\sqrt{|y_{2}-y_{2}|^2+|D h(c(t))(y_{2}-y_{2})|^2}dt.
\end{aligned}
\end{equation}
By \eqref{dh}, we have 
$L_{g}(\tilde{c}) \leq \sqrt{1+\alpha^2}|y_{2}-y_{1}|\leq \sqrt{1+\alpha^2}|\psi(y_{2})-\psi(y_{1})|$. 
Thus,  
$d_{g}(\psi(y_{2}),\psi(y_{1}))\leq \sqrt{1+\alpha^2}|\psi(y_{2})-\psi(y_{1})|$. 
This means that $d_{g}(p,q)\leq \sqrt{1+\alpha^2}|p-q|$ for any $p,q\in \psi(B_{R})$. 
Since $p_{k}$ and $q_{k}$ converge to $\bar{p}$, 
there exists $k_{0}$ such that $p_{k},q_{k} \in \psi(B_{R})$ for all $k\geq k_{0}$. 
Then, 
\[\frac{d_{g}(p_{k},q_{k})}{|p_{k}-q_{k}|}\leq \sqrt{1+\alpha^2}. \]
This contradicts \eqref{contradist}. 
\end{proof}

\begin{lemma}\label{distcomp}
Let $x\in M$ and let $\psi:B_{R}\to M\subset \mathbb{R}^{n}$ be a Langer chart around $x$. 
Then $|x'-x|\geq (\kappa(M))^{-1}R'$ for any $x'\in M\setminus \psi(B_{R'})$ and $0<R'\leq R$. 
\end{lemma}

\begin{proof}
By Definition \ref{embconst}, we have 
$|x'-x|\geq (\kappa(M))^{-1} d_{g}(x,x')$. 
Then, it suffices to prove $d_{g}(x,x')\geq R'$. 
Let $\tilde{c}:[0,T]\to M$ be a curve such that $\tilde{c}(0)=x$, $\tilde{c}(T)=x'$ 
and $d_{g}(x,x')=L_{g}(\tilde{c})$. 
Since $x'\notin \psi(B_{R'})$, there exists the first time $T'\in (0,T]$ so that 
$\tilde{c}(T')$ meets the boundary of $\psi(B_{R'})$. 
Then, it is clear that 
\[d_{g}(x,x')=\int_{0}^{T}|\tilde{c}'(t)|dt\geq \int_{0}^{T'}|\tilde{c}'(t)|dt.\]
By the definition of $T'$, we see that $\tilde{c}(t)\in \psi(B_{R'})$ for $t\in [0,T')$. 
Thus, we can define a curve $c$ in $B_{R'}$ 
by $c(t):=\psi^{-1}(\tilde{c}(t))$ for $t\in [0,T')$. 
Note that $c(0)=0$ and $c(T')\in \partial B_{R'}$. 
Then, by a similar computation to \eqref{Lgc}, we have
\[\int_{0}^{T'}|\tilde{c}'(t)|dt= \int_{0}^{T'}\sqrt{|c'(t)|^2+|D h (c(t))c'(t)|^2}dt\geq \int_{0}^{T'}|c'(t)|dt. \]
The right hand side is just the standard length of the curve $c$ in $B_{R'}$. 
Since $c(0)=0$ and $c(T')\in \partial B_{R'}$, the right hand side is bounded from below by $R'$. 
Then the proof is complete. 
\end{proof}

\subsection{A tubular neighborhood}\label{ATNB}
We denote the normal space of $M$ at $x\in M$ by 
\[T_{x}^{\bot}M:=\{y\in \mathbb{R}^{n} ; y\cdot v=0 \text{ for all }v\in T_{x}M \}\]
and denote the normal bundle of $M$, an $n$-dimensional manifold, by
\[T^{\bot}M:=\{ (x,y)\in M\times \mathbb{R}^{n}  ;   y\in T_{x}^{\bot}M \}\subset \mathbb{R}^{n}\times \mathbb{R}^{n}. \]
Define the  \emph{normal exponential map} $\mathrm{exp}^{\bot}:T^{\bot}M\to \mathbb{R}^{n}$ by $\exp^{\bot}(x,y):=x+y$. 
For $\delta>0$, we write the $\delta$-neighborhood of the normal bundle by 
\[\mathcal{N}_{\delta}(T^{\bot}M)
:=\{ (x,y)\in M \times \mathbb{R}^{n}  ;  y\in T_{x}^{\bot}M, |y|<\delta\}. \]
We continue to assume (\ref{Klamlam}).

\begin{proposition}\label{embexpnor}
For $\delta>0$ satisfying 
\begin{equation}\label{assumpdelta}
	\delta\leq \min\left\{\frac{1}{2\sqrt{2}}
	\frac{\underline{\lambda}^2}{\overline{\lambda}^{4}}K^{-1},  
	\frac{1}{25}\frac{\underline{\lambda}}{\overline{\lambda}^{4}}
	K^{-1}(\kappa(M))^{-1}\right\}, 
\end{equation}
the restriction of $\exp^{\bot}$ to $\mathcal{N}_{\delta}(T^{\bot}M)$ is 
an embedding into $\mathbb{R}^{n}$. 
\end{proposition}

\begin{proof}
Put 
\begin{equation}\label{RalamK}
\begin{aligned}
\alpha:=\frac{1}{6}\frac{\underline{\lambda}}{\overline{\lambda}},\quad
R:=\frac{1}{12}\frac{\underline{\lambda}}{\overline{\lambda}^{4}}K^{-1}.
\end{aligned}
\end{equation}
Then, we have $(1+\alpha^2)^{-3/2}\geq 1/2$, since $\alpha \leq 1/6$. 
From this, one can easily check that $R$ and $\alpha$ satisfy \eqref{CooperLem2.3c}. 
Then,  by Lemma \ref{CooperLem2.3}, $M$ is an $(R,\alpha)$-embedding. 
Hence, for each point $x\in M$, 
there exists a Langer chart $\psi:B_{R}\to M\subset \mathbb{R}^{n}$ around $x$ of radius $R$ 
such that the height function $h:B_{R}\to \mathbb{R}^{n-m}$ associated with $\psi$ satisfies 
\eqref{dh}. 
Define $h_{i}:B_{R}\to \mathbb{R}$ ($i=1,\dots,n-m$) by 
$h=(h_{1},\dots,h_{n-m})$. 
For $y\in B_{R}$, define 
\[\nu_{i}(y):=E_{x}(-D h_{i}(y),e_{i})\in T_{\psi(y)}^{\bot}M, \]
for $i=1,\dots,n-m$, where $\{e_{1},\dots,e_{n-m}\}$ is the standard basis of $\mathbb{R}^{n-m}$. 
Define $\tilde{\psi}:B_{R}\times\mathbb{R}^{n-m}\to T^{\bot}M$ by 
\begin{equation}\label{deflocaltriv}
\tilde{\psi}(y,\xi):=\left(\psi(y), \sum_{i=1}^{n-m}\xi_{i}\nu_{i}(y)\right). 
\end{equation}
Then, $\tilde{\psi}$ gives a chart of $T^{\bot}M$ 
and is called the \emph{local trivialization} of $T^{\bot}M$ associated with $\psi$. 
Put $\Theta:=\exp^{\bot}\circ \tilde{\psi}:B_{R}\times\mathbb{R}^{n-m}\to\mathbb{R}^{n}$. 
Then, we have
\begin{equation}\label{XYZ}
\Theta(y,\xi)=\psi(y)+\sum_{i=1}^{n-m}\xi_{i}\nu_{i}(y),\quad
D\Theta(y,\xi)(y',\xi')=X-Y-Z, 
\end{equation}
for any $(y',\xi')\in\mathbb{R}^{m}\times\mathbb{R}^{n-m}$, where we put
\[
\left\{
\begin{aligned}
X=X(y,\xi)(y',\xi'):=&E_{x}(y',\xi'),\\
Y=Y(y,\xi)(y',\xi'):=&E_{x}\left(\sum_{j=1}^{m}y'_{j}\left(\sum_{i=1}^{n-m}\xi_{i}\partial_{y_{j}}D h_{i}(y),-\partial_{y_{j}}h(y)\right)\right),\\
Z=Z(y,\xi)(y',\xi'):=&E_{x}\left(\sum_{i=1}^{n-m}\xi'_{i}(D h_{i}(y),0)\right), 
\end{aligned}
\right.
\]
for short. 
Thus, 
\begin{equation}\label{normXYZ}
\left\{
\begin{aligned}
&\underline{\lambda}^2|(y',\xi')|^2\leq |X|^2\leq \overline{\lambda}^2|(y',\xi')|^2,\\
&|Y|^2\leq \overline{\lambda}^2\left(|\xi|^2|D^2 h|^2+|D h|^2\right)|y'|^2\leq \overline{\lambda}^2\left(|\xi|^2|D^2 h|^2+\alpha^2\right)|(y',\xi')|^2,\\
&|Z|^2\leq \overline{\lambda}^2|D h|^2|\xi'|^2\leq \overline{\lambda}^2\alpha^2|(y',\xi')|^2. 
\end{aligned}
\right.
\end{equation}

Here, we claim that 
\begin{align}\label{injtheta}
|\Theta(y_{1},\xi_{1})-\Theta(y_{2},\xi_{2})|\geq\frac{\underline{\lambda}}{2\sqrt{2}}|(y_{1},\xi_{1})-(y_{2},\xi_{2})|
\end{align}
for any $(y_{1},\xi_{1}),(y_{2},\xi_{2})\in B_{R}\times \{ |\xi|<\delta/\underline{\lambda}\}$. 
Define a path $c$ in $B_{R}\times \{ |\xi|<\delta/\underline{\lambda}\}$ by 
$c(\theta):=(1-\theta)(y_{1},\xi_{1})+\theta(y_{2},\xi_{2})$ for $\theta\in[0,1]$, and put 
$(y',\xi'):=dc/d\theta=(y_{2},\xi_{2})-(y_{1},\xi_{1})$. 
Then, 
\begin{align}\label{fundtheta}
\begin{aligned}
&\Theta(c(1))-\Theta(c(0))\\
&=\int_{0}^{1}D\Theta(c(\theta))(y',\xi')d\theta\\
&=\int_{0}^{1}X(c(\theta))(y',\xi')d\theta-\int_{0}^{1}Y(c(\theta))(y',\xi')d\theta- \int_{0}^{1}Z(c(\theta))(y',\xi')d\theta\\
&=:\mathcal{X}-\mathcal{Y}-\mathcal{Z}, 
\end{aligned}
\end{align}
where we used (\ref{XYZ}). 
Then by (\ref{normXYZ}), we have
\begin{equation}\label{normsoftXYZ}
\left\{
\begin{aligned}
|\mathcal{X}|^2\geq &\underline{\lambda}^2|(y',\xi')|^2,\\
|\mathcal{Y}|^2\leq &\overline{\lambda}^2\left(\frac{\delta^2}{\underline{\lambda}^2}|D^2 h|^2+\alpha^2\right)|(y',\xi')|^2,\\
|\mathcal{Z}|^2\leq & \overline{\lambda}^2\alpha^2|(y',\xi')|^2. 
\end{aligned}
\right.
\end{equation}
Since $|\mathcal{X}-\mathcal{Y}-\mathcal{Z}|^2
\geq |\mathcal{X}|^2/4-|\mathcal{Y}|^2/2-|\mathcal{Z}|^2$, we have
\begin{align}\label{yam1}
|\Theta(c(1))-\Theta(c(0))|^2\geq
\left(\frac{1}{4}\underline{\lambda}^2-\frac{1}{2}\overline{\lambda}^2\left(\frac{\delta^2}{\underline{\lambda}^2}|D^2 h|^2+3\alpha^2\right) \right)|(y',\xi')|^2, 
\end{align}
where we used (\ref{fundtheta}) and (\ref{normsoftXYZ}). 
Substituting (\ref{ddh}) into (\ref{yam1}) gives
\[
|\Theta(c(1))-\Theta(c(0))|^2\geq
\left(\frac{1}{4}\underline{\lambda}^2-\frac{1}{2}\overline{\lambda}^2
\left(\frac{\delta^2}{\underline{\lambda}^2}\overline{\lambda}^6K^2(1+\alpha^2)^3+3\alpha^2\right) \right)|(y',\xi')|^2. 
\]
By the assumption for $\delta$, one can easily see that 
\[
\frac{\delta^2}{\underline{\lambda}^2}\overline{\lambda}^6K^2(1+\alpha^2)^3+3\alpha^2
\leq \frac{1}{8}\frac{\underline{\lambda}^2}{\overline{\lambda}^2}\left(1+\frac{1}{36}\right)^3+\frac{3}{36}\frac{\underline{\lambda}^2}{\overline{\lambda}^2}
\leq \frac{1}{4}\frac{\underline{\lambda}^2}{\overline{\lambda}^2}. 
\]
Thus, 
\[
	|\Theta(c(1))-\Theta(c(0))|^2\geq \frac{1}{8}\underline{\lambda}^2|(y',\xi')|^2. 
\]
Hence the claim \eqref{injtheta} follows. 

Then, by \eqref{injtheta}, it is clear that $\Theta$ is injective on $B_{R}\times \{ |\xi|<\delta/\underline{\lambda}\}$ and 
its derivative $D\Theta(y,\xi)$ at $(y,\xi) \in B_{R}\times \{ |\xi|<\delta/\underline{\lambda}\}$ is also injective. 
Thus, we proved that the restriction of $\exp^{\bot}$ to $\tilde{\psi}(B_{R}\times \{ |\xi|<\delta/\underline{\lambda}\})$ is an embedding. 
Moreover, by the definition of $\nu_{i}(y)$, we have
\begin{align}\label{yam4.5}
\underline{\lambda}|\xi|\leq \left|\sum_{i=1}^{n-m}\xi_{i}\nu_{i}(y)\right|\leq \overline{\lambda}\sqrt{1+\alpha^2}|\xi|. 
\end{align}
This shows that
\begin{equation}\label{nbdinclusion}
\left\{
\begin{aligned}
&\tilde{\psi}(B_{R}\times \mathbb{R}^{n-m})\cap \mathcal{N}_{\delta}(T^{\bot}M)\subset \tilde{\psi}(B_{R}\times \{ |\xi|<\delta/\underline{\lambda}\}), \\
&\tilde{\psi}(B_{R}\times \mathbb{R}^{n-m})\cap \mathcal{N}_{\delta}(T^{\bot}M)\supset \tilde{\psi}(B_{R}\times \{ |\xi|<\delta/(\overline{\lambda}\sqrt{1+\alpha^2})\}). 
\end{aligned}
\right.
\end{equation}
Hence, the restriction of $\exp^{\bot}$ to $\tilde{\psi}(B_{R}\times \mathbb{R}^{n-m})\cap \mathcal{N}_{\delta}(T^{\bot}M)$ 
is also an embedding. 
Especially, the restriction of $\exp^{\bot}$ to $\mathcal{N}_{\delta}(T^{\bot}M)$ is an immersion. 

In what follows, we prove 
that the restriction of $\exp^{\bot}$ to $\mathcal{N}_{\delta}(T^{\bot}M)$ is injective. 
To obtain a contradiction, assume that there exist $(x_{1},y_{1}), (x_{2},y_{2})\in \mathcal{N}_{\delta}(T^{\bot}M)$ such that $(x_{1},y_{1})\neq (x_{2},y_{2})$ and 
$\exp^{\bot}(x_{1},y_{1})=\exp^{\bot}(x_{2},y_{2})$, that is,  
$x_{1}+y_{1}=x_{2}+y_{2}$. 
Let $\psi_{x_{1}}:B_{R}\to M$ be a Langer chart around $x_{1}$. 
Without loss of generality, we can assume $x_{2}\notin \psi_{x_{1}}(B_{R})$, 
since the restriction of $\exp^{\bot}$ to 
$\tilde{\psi_{x_{1}}}(B_{R}\times \mathbb{R}^{n-m})\cap \mathcal{N}_{\delta}(T^{\bot}M)$ 
is injective. 
Then, by Lemma \ref{distcomp} and \eqref{RalamK}, we have 
\[|x_{1}-x_{2}|\geq (\kappa(M))^{-1}R=\frac{1}{12}\frac{\underline{\lambda}}{\overline{\lambda}^{4}}K^{-1}(\kappa(M))^{-1}. \]
On the other hand, since $x_{1}+y_{1}=x_{2}+y_{2}$, we have
$|x_{1}-x_{2}|=|y_{1}-y_{2}|\leq 2\delta$. 
Thus, 
\[
\delta
\geq\frac{1}{24}\frac{\underline{\lambda}}{\overline{\lambda}^{4}}K^{-1}(\kappa(M))^{-1}.
\]
This contradicts the assumption for $\delta$. 
Hence the restriction of $\exp^{\bot}$ to $\mathcal{N}_{\delta}(T^{\bot}M)$ is an injective immersion, that is, an embedding. 
\end{proof}

\begin{proposition}\label{normalanddelta}
$\exp^{\bot}(\mathcal{N}_{\delta}(T^{\bot}M))=M^{\delta}$ 
for any $\delta>0$ with \eqref{assumpdelta}. 
Moreover, for each $z\in M^{\delta}$, there exists 
a unique point $x\in M$ satisfying 
\[d(z,M)=|z-x|\quad\text{and}\quad z-x\in T^{\bot}_{x}M.\]
\end{proposition}

\begin{proof}
For any $\exp^{\bot}(x,y)\in \exp^{\bot}(\mathcal{N}_{\delta}(T^{\bot}M))$, 
we have $|\exp^{\bot}(x,y)-x|=|y|<\delta$. 
This implies $\exp^{\bot}(x,y)\in M^{\delta}$, since $x \in M$. 
Conversely, fix $z\in M^{\delta}$ arbitrary. 
Consider a smooth function $d_{z}^{2}:M\to \mathbb{R}$ defined by $d_{z}^{2}(x):=|z-x|^2$. 
Since $M$ is compact, there exists a point $x\in M$ such that 
\[|z-x|^2=\min_{M}d_{z}^2=d(z,M)^2<\delta^2. \]
Putting $\xi:=z-x$, we have $|\xi|<\delta$. 
Let $\psi:B_{R}\to M$ be a Langer chart around $x$. 
Then, the derivative of $d_{z}^2\circ \psi :B_{R}\to \mathbb{R}$ at $y\in B_{R}$ is given by 
$\partial_{y_{j}}(d_{z}^2\circ \psi)(y)=2(z-\psi(y))\cdot  \partial_{y_{j}} \psi(y)$. 
Since $d_{z}^{2}$ attains its minimum at $x=\psi(0)$, it must satisfy 
\[0=\partial_{y_{j}}(d_{z}^2\circ \psi)(0)=2(z-x)\cdot  \partial_{y_{j}} \psi(0). \]
Since $\partial_{y_{j}} \psi(0)$ ($=E_{x}\bar{e}_{j}$) are basis of $T_{x}M$, 
we have $\xi=z-x\in T^{\bot}_{x}M$. 
This means that $(x,\xi)\in \mathcal{N}_{\delta}(T^{\bot}M)$ and $z=x+\xi=\exp^{\bot}(x,\xi)\in \exp^{\bot}(\mathcal{N}_{\delta}(T^{\bot}M))$. 

Suppose that there 
exists $x'\in M$ satisfying $x'\neq x$ and 
\[|z-x'|^2=\min_{M}d_{z}^2=(d(z,M))^2<\delta^2. \]
Put $\xi':=z-x'\in\mathbb{R}^{n}$ and repeat the same argument as above. 
Then, $(x',\xi') \in\mathcal{N}_{\delta}(T^{\bot}M)$ and $z=\exp^{\bot}(x',\xi')$. 
This contradicts Proposition \ref{embexpnor}. 
\end{proof}

\subsection{Integral over the tubular neighborhood}\label{IoTNB}
Let $\tilde{\psi}:B_{R}\times\mathbb{R}^{n-m}\to T^{\bot}M$ be the local trivialization 
associated with a Langer chart $\psi$ around $x$ defined by \eqref{deflocaltriv}. 
Put $\Theta:=\exp^{\bot}\circ \tilde{\psi}$.

\begin{lemma}
On $\tilde{\psi}(B_{R}\times \mathbb{R}^{n-m})\cap \mathcal{N}_{\delta}(T^{\bot}M)$, 
\begin{align}\label{JJ}
\underline{C} J\psi\leq J\Theta\leq \overline{C} J\psi, 
\end{align}
where $\underline{C},\overline{C}>0$ are given by
\[
\left\{
\begin{aligned}
\underline{C}:=&\left(\frac{\underline{\lambda}}{2\sqrt{2}}\right)^{\frac{n}{2}}\left(\overline{\lambda}^{m}(1+\alpha^{2})^{m/2}\right)^{-1},\\
\overline{C}:=&\left(5\overline{\lambda}^8(1+\alpha^2)^{3}K^2\delta^2\underline{\lambda}^{-2}+5\overline{\lambda}^2\left(1+2\alpha^2\right)\right)^{n/2}\underline{\lambda}^{-m}. 
\end{aligned}
\right.
\]
\end{lemma}

\begin{proof}
Fix a point $\tilde{\psi}(x,\xi)\in \tilde{\psi}(B_{R}\times \mathbb{R}^{n-m})\cap \mathcal{N}_{\delta}(T^{\bot}M)$. 
By \eqref{injtheta} and \eqref{nbdinclusion}, 
we have $|D\Theta(x,\xi)|\geq \underline{\lambda}/(2\sqrt{2})$. 
Thus, $J\Theta\geq (\underline{\lambda}/(2\sqrt{2}))^{\frac{n}{2}}$. 
This together with Lemma \ref{LanPro} yields 
\[|D\Theta(x,\xi)|\geq \left(\frac{\underline{\lambda}}{2\sqrt{2}}\right)^{\frac{n}{2}}\left(\overline{\lambda}^{m}(1+\alpha^{2})^{m/2}\right)^{-1}J\psi(y).\]
This is the first inequality of \eqref{JJ}. 

To prove the second inequality, 
let $\mu$ be a maximum eigenvalue of $({}^{t}D \Theta)(D\Theta)$ at $(x,\xi)$ and let $\eta=(x',\xi')$ be its eigenvector. 
Then, $|D\Theta(x,\xi)\eta|^2=\mu|\eta|^2$, and so $\mu\leq |D\Theta(x,\xi)|^2$. 
By (\ref{XYZ}), (\ref{normXYZ}) and Lemma \ref{CooperLem2.1}, we have
\begin{align*}
|D\Theta(x,\xi)|^2&\leq  5\left(\overline{\lambda}^2+\overline{\lambda}^2\left(|\xi|^2|D^2 h|^2+\alpha^2\right)+\overline{\lambda}^2\alpha^2\right)\\
&\leq 5\overline{\lambda}^8(1+\alpha^2)^{3}K^2|\xi|^2+5\overline{\lambda}^2\left(1+2\alpha^2\right). 
\end{align*}
This together with $|\xi|<\delta/\underline{\lambda}$ shows
\[
\mu\leq |D\Theta(x,\xi)|^2\leq  5\overline{\lambda}^8(1+\alpha^2)^{3}K^2\delta^2\underline{\lambda}^{-2}+5\overline{\lambda}^2\left(1+2\alpha^2\right)
\]
on $\tilde{\psi}(B_{R}\times \mathbb{R}^{n-m})\cap \mathcal{N}_{\delta}(T^{\bot}M)$. 
From $J\Theta\leq \mu^{n/2}$ and Lemma \ref{LanPro}, 
the second inequality of (\ref{JJ}) immediately follows. 
\end{proof}

\begin{proposition}\label{tbnint}
Let $N:=n-m$ and $M^{\delta,\delta'}:=M^{\delta}\setminus M^{\delta'}$. 
For any $\delta>0$ with \eqref{assumpdelta}, 
$\delta'>0$ with $\delta'<\underline{\lambda}(\overline{\lambda}\sqrt{1+\alpha^2})^{-1}\delta$ and 
$\beta\in\mathbb{R}$, 
\[
\begin{aligned}
\underline{C}'\int_{\delta'/\underline{\lambda}}^{\delta/(\overline{\lambda}\sqrt{1+\alpha^2})}
r^{\beta+N-1}dr
\leq\int_{M^{\delta,\delta'}}d(x,M)^{\beta}dx
\leq& \overline{C}'\int_{\delta'/(\overline{\lambda}\sqrt{1+\alpha^2})}^{\delta/\underline{\lambda}}r^{\beta+N-1}dr,
\end{aligned}
\]
where $\underline{C}',\overline{C}'>0$ are given by 
\[
\left\{
\begin{aligned}
\underline{C}'&:=\mathcal{H}^{m}(M)
N\omega_{N}
\left(\frac{\underline{\lambda}}{2\sqrt{2}}\right)^{\frac{n}{2}}\left(\overline{\lambda}^{m}(1+\alpha^{2})^{m/2}\right)^{-1}\\
&\quad\times
\min\left\{(\overline{\lambda}\sqrt{1+\alpha^2})^{\beta}, \underline{\lambda}^{\beta} \right\}, \\
\overline{C}'&:=\mathcal{H}^{m}(M)
N\omega_{N}
\left(5\overline{\lambda}^8(1+\alpha^2)^{3}K^2\delta^2\underline{\lambda}^{-2}+5\overline{\lambda}^2\left(1+2\alpha^2\right)\right)^{n/2}\underline{\lambda}^{-m}\\
&\quad\times 
\max\left\{(\overline{\lambda}\sqrt{1+\alpha^2})^{\beta}, \underline{\lambda}^{\beta} \right\}. 
\end{aligned}
\right.
\]
In particular, $\overline{C}'<\infty$ when $\delta\to 0$. 
\end{proposition}

\begin{proof}
Let 
$\mathcal{A}:=\{(\psi_{i},\psi_{i}(B_{R}))\}_{i=1,\dots , N}$ 
be a finite collection of Langer charts which covers $M$. 
Then, $\mathcal{B}:=\{(\tilde{\psi_{i}},\tilde{\psi_{i}}(B_{R}\times\mathbb{R}^{n-m}))\}_{i=1,\dots , N}$ 
is a finite collection of local trivializations which covers $T^{\bot}M$.
Let $\{\rho_{i}\}_{i=1,\dots , N}$ be a partition of unity subordinate to $\mathcal{A}$. 
Define $\tilde{\rho}_{i}: \tilde{\psi}(B_{R}\times\mathbb{R}^{n-m})\to\mathbb{R}$ 
by $\tilde{\rho}_{i}(x,y):=\rho_{i}(x)$. 
Then, $\{\tilde{\rho_{i}}\}_{i=1,\dots , N}$ is a partition of unity subordinate to $\mathcal{B}$. 

By Proposition \ref{normalanddelta},  
$\exp^{\bot}:\mathcal{N}_{\delta,\delta'}(T^{\bot}M) \to M^{\delta,\delta'}$ is a diffeomorphism, 
where 
\[\mathcal{N}_{\delta,\delta'}(T^{\bot}M):=\{ (x,y)\in M \times \mathbb{R}^{n}  ;  y\in T_{x}^{\bot}M, \delta'\leq |y|<\delta\}.\]
Thus, 
\[\int_{M^{\delta,\delta'}}d(x,M)^{\beta}dx=\int_{\mathcal{N}_{\delta,\delta'}(T^{\bot}M)}d(\exp^{\bot}(\cdot),M)^{\beta}d\nu, \]
where $\nu$ is the induced measure on $\mathcal{N}_{\delta,\delta'}(T^{\bot}M)$ by 
$\exp^{\bot}:\mathcal{N}_{\delta,\delta'}(T^{\bot}M)\to \mathbb{R}^{n}$. 
By using $\mathcal{B}$ and $\{\tilde{\rho_{i}}\}$, 
we have
\[
\int_{\mathcal{N}_{\delta,\delta'}(T^{\bot}M)}d(\exp^{\bot}(\cdot ),M)^{\beta}d\nu
=\sum_{i=1}^{N}\int_{\tilde{\psi_{i}}(\mathcal{U}_{i})}\tilde{\rho_{i}}(\cdot)d(\exp^{\bot}(\cdot),M)^{\beta}d\nu,
\]
where we put $\mathcal{U}_{i}:=\tilde{\psi_{i}}^{-1}(\mathcal{N}_{\delta,\delta'}(T^{\bot}M))$ for short. 
By (\ref{nbdinclusion}), 
we have
\[B_{R}\times \left\{\frac{\delta'}{\underline{\lambda}}\leq|\xi|<\frac{\delta}{\overline{\lambda}\sqrt{1+\alpha^2}}\right\}
\subset \mathcal{U}_{i}
\subset B_{R}\times \left\{\frac{\delta'}{\overline{\lambda}\sqrt{1+\alpha^2}}\leq|\xi|<\frac{\delta}{\underline{\lambda}}\right\} \]
and the set on the left hand side is nonempty by the assumption for $\delta'$. 
For any $(y,\xi)\in B_{R}\times \{|\xi|<\delta/\underline{\lambda}\}$, 
we have $\tilde{\rho_{i}}(\tilde{\psi_{i}}(x,\xi))=\rho_{i}(\psi_{i}(x))$ 
by definitions of $\tilde{\psi_{i}}$ and $\tilde{\rho_{i}}$. 
Moreover, by Proposition \ref{normalanddelta}, we have
\[d(\exp^{\bot}(\tilde{\psi_{i}}(x,\xi)),M))=\left| \left(\psi(y)+\sum_{i=1}^{n-m}\xi_{i}\nu_{i}(y)\right)-\psi(y)\right|
=\left|\sum_{i=1}^{n-m}\xi_{i}\nu_{i}(y)\right|. \]
This together with \eqref{yam4.5} implies 
\[C''|\xi|^{\beta}\leq d(\exp^{\bot}(\tilde{\psi_{i}}(x,\xi)),M))^{\beta}\leq C'|\xi|^{\beta}, \]
where $C'=(\overline{\lambda}\sqrt{1+\alpha^2})^{\beta}$ and $C''=\underline{\lambda}^{\beta}$ when $\beta\geq 0$, and 
$C'=\underline{\lambda}^{\beta}$ and $C''=(\overline{\lambda}\sqrt{1+\alpha^2})^{\beta}$ when $\beta \leq 0$. 
Thus, 
\begin{align*}
&C''\int_{B_{R}\times \{\delta'/\underline{\lambda}\leq|\xi|<\delta/(\overline{\lambda}\sqrt{1+\alpha^2})\}}
\rho_{i}(\psi_{i}(x))|\xi|^{\beta}J\Theta_{i}(x,\xi)dxd\xi\\
&\leq \int_{\tilde{\psi_{i}}(\mathcal{U}_{i})}\tilde{\rho_{i}}(\cdot)d(\exp^{\bot}(\cdot),M)^{\beta}d\nu\\
&\leq C'\int_{B_{R}\times \{\delta'/(\overline{\lambda}\sqrt{1+\alpha^2})\leq|\xi|<\delta/\underline{\lambda}\}}
\rho_{i}(\psi_{i}(x))|\xi|^{\beta}J\Theta_{i}(x,\xi)dxd\xi. 
\end{align*}
Combining this and \eqref{JJ} yields 
\begin{align*}
&C''\underline{C}\int_{\{\delta'/\underline{\lambda}\leq|\xi|<\delta/(\overline{\lambda}\sqrt{1+\alpha^2})\}}|\xi|^{\beta}d\xi
\int_{B_{R}}\rho_{i}(\psi_{i}(x))J\psi_{i}(x)dx\\
&\leq \int_{\tilde{\psi_{i}}(\mathcal{U}_{i})}\tilde{\rho_{i}}(\cdot)d(\exp^{\bot}(\cdot),M)^{\beta}d\nu\\
&\leq C'\overline{C}\int_{\{\delta'/(\overline{\lambda}\sqrt{1+\alpha^2})\leq|\xi|<\delta/\underline{\lambda}\}}|\xi|^{\beta}d\xi
\int_{B_{R}}\rho_{i}(\psi_{i}(x))J\psi_{i}(x)dx. 
\end{align*}
Summing the above inequalities over $i=1,\dots,N$, and combining elementary facts 
\[
\begin{aligned}
&\sum_{i=1}^{N}\int_{B_{R}}\rho_{i}(\psi_{i}(x))J\psi_{i}(x)dx=\mathcal{H}^{m}(M),\\
&\int_{\{a\leq|\xi|<b\}}|\xi|^{\beta}d\xi=N\omega_{N}\int_{a}^{b}r^{\beta+N-1}dr, 
\end{aligned}
\]
we proved the desired estimate. 
\end{proof}

\subsection{Time-dependent Langer charts}\label{tdepLC}
Let $M_{0} \subset \mathbb{R}^{n}$ be a 
connected compact $m$-dimensional smooth submanifold without boundary 
and let $I$ be an open interval with $0\in I$. 
Let $F \in C^{\infty}(\mathbb{R}^{n}\times I;\mathbb{R}^{n})$ such that $F_{t}(x):= F(x,t)$ satisfy 
$F_{0} = \mathrm{id}_{\mathbb{R}^n}$ and \eqref{eq:bF} for some constant $B>0$. 
We remark that the assumption \eqref{eq:HF} is not necessary in this subsection. 

We write $M_{t}:=F_{t}(M_0)$. 
In this subsection, we introduce time-dependent Langer charts 
and show some properties. 
At $t=0$, take $E_{x}$, appeared at the beginning of Subsection \ref{LangChar}, from $\mathrm{O}(n,\mathbb{R})$ for each $x\in M_0$. 
At any $t\in I$, define $E_{F_{t}(x)}\in \mathrm{GL}(n,\mathbb{R})$ by 
\[E_{F_{t}(x)}:=DF_{t}(x)E_{x}\]
for $F_{t}(x)\in M_{t}$. 
Then, 
\[T_{F_{t}(x)}M_{t}=\{E_{F_{t}(x)}\bar{y}+F_{t}(x) ; \bar{y}=(y_{1},\dots,y_{m},0,\dots,0)\in \mathbb{R}^{m}\times \mathbb{R}^{n-m}\}. \]
By using $E_{F_{t}(x)}$, a Langer chart $\psi_{t}:B_{R'}\to M_{t}\subset \mathbb{R}^{n}$ 
centered at $F_{t}(x)$ of some radius $R'=R'(x,t)$ is defined as in Subsection \ref{LangChar}. 

We claim that there exist a sufficiently small radius $R$ which does not depend on $x$ and $t$. 
Since $E_{x}\in \mathrm{O}(n,\mathbb{R})$, we have 
$\overline{\lambda}(E_{F_{t}(x)})=\overline{\lambda}(DF_{t}(x))$ and 
$\underline{\lambda}(E_{F_{t}(x)})=\underline{\lambda}(DF_{t}(x))$. 
This together with \eqref{eq:bF} yields 
\[\overline{\lambda}(E_{F_{t}(x)})\leq B\quad\text{and}\quad \underline{\lambda}(E_{F_{t}(x)})\geq B^{-1}. \]
Thus, we can take $\underline{\lambda}=B^{-1}$ and $\overline{\lambda}=B$ in \eqref{Klamlam} uniformly for $t\in I$. 
Moreover, by the compactness of $M_{0}$ and \eqref{eq:bF}, 
there exists $K=K(m,B,M_{0})>0$ such that 
\[
\max_{F_{t}(x)\in M_{t}}|\mathrm{II}_{F_{t}(x)}|\leq K
\]
for any $t\in I$. 
Fix $R>0$ and $\alpha>0$ satisfying 
\begin{align}\label{RBK-1}
R\leq \alpha\left(B^{2}(1+\alpha^2)\right)^{-3/2}K^{-1}. 
\end{align}
Then, by Lemma \ref{CooperLem2.3}, $M_{t}$ is an $(R,\alpha)$-embedding, that is, 
each Langer chart $\psi_{t}$ is defined on $B_{R}$ and 
the height function $h_{t}:B_{R}\to\mathbb{R}^{n-m}$ associated with $\psi_{t}$ satisfies 
$|D h_{t}|\leq \alpha$ on $B_{R}$. 
Note that $R$ and $\alpha$ do not depend on $x$ and $t$. 
We call $\psi_{t}:B_{R}\to M_{t}$ the \emph{time-dependent Langer chart} centered at $F_{t}(x)$. 

\begin{lemma}\label{tdLanger}
$|\partial_{t}\psi_{t}|\leq C_{2}$ on $B_R$ for some $C_{2}=C_{2}(B,R,\alpha)>0$.
\end{lemma}
\begin{proof}
Fix $y \in B_{R}$ and $t'\in I$. 
Define a curve $\tau\mapsto y_{\tau}$ in $\mathbb{R}^{m}$ by 
\[
y_{\tau}:=\mathrm{proj}_{\mathbb{R}^{m}}(E_{F_{t'+\tau}(x)}^{-1}(F_{t'+\tau}\circ F_{t'}^{-1}\circ \psi_{t'}(y)-F_{t'+\tau}(x)))
\]
for any $\tau$ such that $\tau,\tau+t'\in I$. 
When $\tau=0$, it is clear that $y_{0}=y\in B_{R}$ by the definition of the Langer chart $\psi_{t'}$. 
Hence, by the continuity of $y_{\tau}$, there exists $\tau_{0}>0$ such that 
$y_{\tau}\in B_{R}$ for any $\tau$ with $|\tau|\leq \tau_{0}$. 
Note that $\psi_{t'+\tau}(y_{\tau})=F_{t'+\tau}\circ F_{t'}^{-1}\circ \psi_{t'}(y)$ by the definition of the Langer chart $\psi_{t'+\tau}$. 
Thus, 
\begin{align*}
\psi_{t'+\tau}(y)-\psi_{t'}(y)&=(\psi_{t'+\tau}(y_{\tau})-\psi_{t'}(y))+(\psi_{t'+\tau}(y)-\psi_{t'+\tau}(y_{\tau}))\\
&=(F_{t'+\tau}(\tilde{y})-F_{t'}(\tilde{y}))+(\psi_{t'+\tau}(y)-\psi_{t'+\tau}(y_{\tau})), 
\end{align*}
where we put $\tilde{y}:=F_{t'}^{-1}\circ \psi_{t'}(y)\in M_{0}$ for short, 
and so
\[
|\partial_{t}\psi_{t}|\leq B+ |D\psi_{t'}|\lim_{\tau\to 0}\left|\frac{y_{\tau}-y}{\tau}\right|
\leq  B+ \left(B\sqrt{1+\alpha^2}\right)\lim_{\tau\to 0}\left|\frac{y_{\tau}-y}{\tau}\right|, 
\]
where we used the assumption \eqref{eq:bF} and Lemma \ref{LanPro}. 
Thus, it suffices to give an upper bound of $\lim_{\tau\to 0}\left|(y_{\tau}-y)/\tau\right|$. 
By the definition of $y_{\tau}$ and a trivial identity
\[
y=\mathrm{proj}_{\mathbb{R}^{m}}\left(E_{F_{t'+\tau}(x)}^{-1}\left(DF_{t'+\tau}(x)E_{x}(y,h_{t'}(y))\right)\right), 
\]
we have 
\[
\begin{aligned}
|y_{\tau}-y|
&\leq  B|F_{t'+\tau}(\tilde{y})-F_{t'+\tau}(x)-DF_{t'+\tau}(x)E_{x}(y,h_{t'}(y))|\\
&\leq  B |F_{t'+\tau}(\tilde{y})-F_{t'}(\tilde{y})|+B |F_{t'+\tau}(x)-F_{t'}(x)|\\
&\quad +B |F_{t'}(x)-F_{t'}(\tilde{y})+DF_{t'+\tau}(x)E_{x}(y,h_{t'}(y))|\\
&=: B(X+Y+Z). 
\end{aligned}
\]
It is clear that $\lim_{\tau\to 0}|X/\tau|\leq B$ and $\lim_{\tau\to 0}|Y/\tau|\leq B$ by \eqref{eq:bF}. 
By the definition of the Langer chart $\psi_{t'}$ and $F_{t'}(\tilde{y})=\psi_{t'}(y)$, we have
\[F_{t'}(x)-F_{t'}(\tilde{y})=-DF_{t'}(x)E_{x}(y,h_{t'}(y)). \]
Then, from \eqref{eq:bF}, it follows that
\[
\lim_{\tau\to 0}\left|\frac{Z}{\tau}\right|\leq \lim_{\tau\to 0}\left|\frac{DF_{t'+\tau}(x)-DF_{t'}(x)}{\tau}\right||E_{x}(y,h_{t'}(y))|\leq B\sqrt{1+\alpha^2}R. 
\]
Then the proof is complete. 
\end{proof}

By Lemma \ref{LanPro} and \ref{tdLanger}, 
the following holds. 

\begin{proposition}\label{tdepLanPro}
For any $y\in B_{R}$, $\eta\in\mathbb{R}^{m}$ and $t\in I$, 
\[
\left\{
\begin{aligned}
&B^{-1}|\eta|\leq |D\psi_{t}(y)\eta| \leq B\sqrt{1+\alpha^2}|\eta|,\\
&B^{-m}\leq J\psi_{t}(y) \leq B^{m}(1+\alpha^{2})^{m/2},\\
&|D^2\psi_{t}(y)|\leq B^4 (1+\alpha^2)^{3/2}K,\\
&|D J\psi_{t}(y)|\leq C_{1}(m) B^{m+7}(1+\alpha^2)^{3}K, \\
&|\partial_{t}\psi_{t}|\leq C_{2}(B,R,\alpha). 
\end{aligned}
\right.
\]
\end{proposition}

Here, we estimate the embedding constant and the diameter of $M_{t}$. 

\begin{lemma}\label{kapkap}
$\kappa(M_{t})\leq B^{2}\kappa(M_{0})$ for any $t\in I$. 
\end{lemma}
\begin{proof}
By Definition \ref{embconst}, we have 
\[\kappa(M_{t}):=\sup\left\{\frac{d_{g_{t}}(F_{t}(p),F_{t}(q))}{|F_{t}(p)-F_{t}(q)|};  p,q\in M_{0}\mbox{ with } p\neq q  \right\},\]
where $g_{t}$ is the induced Riemannian metric on $M_{t}$. 
Let $c:[0,s_{0}]\to M_{0}$ be any 
curve with 
$c(0)=p$ and $c(s_{0})=q$. 
Then, we have a curve $F_{t}\circ c$ in $M_{t}$ with $(F_{t}\circ c)(0)=F_{t}(p)$ and $(F_{t}\circ c)(s_{0})=F_{t}(q)$. 
Then, by the assumption \eqref{eq:bF}, 
\[L_{g_{t}}(F_{t}\circ c)=\int_{0}^{s_{0}}|D F_{t}(c'(s))|ds\leq B \int_{0}^{s_{0}}|c'(s)|ds=BL_{g_{0}}(c). \]
Thus, $d_{g_{t}}(F_{t}(p),F_{t}(q))\leq Bd_{g_{0}}(p,q)$. 
On the other hand, also by the assumption \eqref{eq:bF}, we have
$|F_{t}(p)-F_{t}(q)|/|p-q|\geq B^{-1}$. 
Combining these inequalities yields 
\[\frac{d_{g_{t}}(F_{t}(p),F_{t}(q))}{|F_{t}(p)-F_{t}(q)|}\leq B^{2}\frac{d_{g_{0}}(p,q)}{|p-q|},\]
and the lemma is proved. 
\end{proof}

\begin{lemma}\label{diambound}
There exists a constant $D=D(B,M_{0})>0$ such that $\mathop{\mathrm{diam}}(M_{t})\leq D$ 
for any $t\in I$, where $\mathop{\mathrm{diam}}(M_{t}):=\sup\{ |x-y|;x,y\in M_{t}\}$. 
\end{lemma}
\begin{proof}
For any $F_{t}(p),F_{t}(q)\in M_{t}$, we proved $d_{g_{t}}(F_{t}(p),F_{t}(q))\leq Bd_{g_{0}}(p,q)$ in the proof of Lemma \ref{kapkap}. 
Since $M_{0}$ is compact, $\sup\{d_{g_{0}}(p,q);p,q\in M_{0}\}$ is bounded. 
Moreover, we always have $|F_{t}(p)-F_{t}(q)|\leq d_{g_{t}}(F_{t}(p),F_{t}(q))$. 
Then the proof is complete. 
\end{proof}

By Lemma \ref{distcomp} and \ref{kapkap}, the following holds. 
\begin{proposition}\label{tdistcomp}
Let $t\in I$ and $x\in M_{t}$ and let $\psi_{t}:B_{R}\to M_{t} \subset \mathbb{R}^{n}$ be a time-dependent Langer chart centered at $x$. 
Then $|x'-x|\geq B^{-2}(\kappa(M_{0}))^{-1}R'$ for any  $x'\in M_{t}\setminus \psi_{t}(B_{R'})$ and $0<R'\leq R$. 
\end{proposition}

By Proposition \ref{embexpnor}, \ref{normalanddelta} and Lemma \ref{kapkap}, the following holds. 

\begin{proposition}\label{tdepnormal}
For $\delta>0$ satisfying 
\begin{align}\label{nicedelta}
\delta\leq \min\left\{\frac{1}{2\sqrt{2}}B^{-6}K^{-1}  ,  \frac{1}{25}B^{-7}K^{-1}(\kappa(M_{0}))^{-1}\right\}, 
\end{align}
the restriction of $\exp^{\bot}$ to $\mathcal{N}_{\delta}(T^{\bot}M_{t})$ is 
an embedding and $\exp^{\bot}(\mathcal{N}_{\delta}(T^{\bot}M_{t}))=M_{t}^{\delta}$ for any $t\in I$. 
Moreover, for each $t\in I$ and $z\in M_{t}^{\delta}$, there exists a 
unique point $x\in M_{t}$ satisfying 
\begin{align}\label{nicedelta2}
d(z,M_{t})=|z-x|\quad\text{and}\quad z-x\in T^{\bot}_{x}M_{t}.
\end{align}
\end{proposition}

Here, we estimate the volume of $M_{t}$. 

\begin{proposition}\label{estVol}
$B^{-m}\mathcal{H}^{m}(M_{0})\leq \mathcal{H}^{m}(M_{t})\leq B^{m}\mathcal{H}^{m}(M_{0})$ for any $t\in I$. 
\end{proposition}
\begin{proof}
We write the restriction of $F_{t}:\mathbb{R}^{n}\to \mathbb{R}^{n}$ to $M_{0}$ by 
$\bar{F}_{t}:M_{0}\to \mathbb{R}^{n}$. 
By $M_{t}=\bar{F}_{t}(M_{0})$ and the change of variables, we have
\[\mathcal{H}^{m}(M_{t})=\int_{M_{0}}J\bar{F}_{t}d\mathcal{H}^{m}.\]
Then, by the assumption \eqref{eq:bF}, we see that $B^{-m}\leq J\bar{F}_{t}\leq B^{m}$, 
and the proof is complete. 
\end{proof}

By Proposition \ref{tbnint} and \ref{estVol}, the following holds. 
\begin{proposition}\label{tdeptnbint}
Let $N:=n-m$. For any $\delta>0$ with \eqref{nicedelta}, $\delta'>0$ with $\delta'<B^{-1}(B\sqrt{1+\alpha^2})^{-1}\delta$, 
$\beta\in\mathbb{R}$ and $t\in I$, 
\[
\underline{C}'\int_{\delta'B}^{\delta/(B\sqrt{1+\alpha^2})}r^{\beta+N-1}dr
\leq\int_{M_{t}^{\delta,\delta'}}d(x,M_{t})^{\beta}dx
\leq \overline{C}'\int_{\delta'/(B\sqrt{1+\alpha^2})}^{B\delta}r^{\beta+N-1}dr,
\]
where $\underline{C}',\overline{C}'>0$ are given by 
\[
\left\{
\begin{aligned}
\underline{C}'&:= B^{-m}\mathcal{H}^{m}(M_{0})
N\omega_{N}
\left(\frac{B^{-1}}{2\sqrt{2}}\right)^{\frac{n}{2}}\left(B^{m}(1+\alpha^{2})^{m/2}\right)^{-1}\\
&\quad\times
\min\left\{(B\sqrt{1+\alpha^2})^{\beta}, B^{-\beta} \right\}, \\
\overline{C}'& := B^{m}\mathcal{H}^{m}(M_{0})N\omega_{N}
\left(5B^{10}(1+\alpha^2)^{3}K^2\delta^2+5B^2\left(1+2\alpha^2\right)\right)^{n/2}\\
&\quad \times B^{m}
\max\left\{(B\sqrt{1+\alpha^2})^{\beta}, B^{-\beta} \right\}. 
\end{aligned}
\right.
\]
In particular, $\overline{C}'<\infty$ when $\delta\to 0$. 
\end{proposition}

\section{A singular solution of the linear heat equation}\label{sec:U}
Let $\underline{T}\in \mathbb{R}$. Throughout this section, we set a time interval 
$I=(\underline{T},\infty)$ so that $F_{t}$ is defied for $t\in I$. 
In Subsection \ref{subsect:proU}, 
we show that $U=U(x,t;\underline{T})$ defined by \eqref{eq:Udefi} is a singular solution of the linear heat equation. 
The behavior of $U$ near $M_t$ is shown in 
Subsections \ref{subsect:bUnsing}. 
Uniform estimates of $U$ are given in 
Subsection \ref{subsect:Uunif}. 
Remark that we continue to assume that $F$ satisfies \eqref{eq:bF}
in the all parts of this section, 
but we assume \eqref{eq:HF} only in Subsection \ref{subsect:Uunif}.

\subsection{Properties of the solution}\label{subsect:proU}
The function $U$ has the following properties. 

\begin{proposition}\label{pro:heat}
Suppose that $F$ satisfies \eqref{eq:bF}. 
Then the function $U$ belongs to 
$L_\mathrm{loc}^1(Q_{I})\cap C^\infty(Q_I \setminus M_I)$ 
and satisfies 
\begin{equation}\label{eq:heat}
	\partial_t U=\Delta U \quad \mbox{ in }
	Q_I \setminus M_I. 
\end{equation}
\end{proposition}

\begin{proof}
This lemma can be proved by the argument of 
\cite[Section~3]{KT16} and \cite[Proposition 2.1]{HTYpre}. 
Let $\underline{T}<\underline{t}<\overline{t}<+\infty$. 
Then, by the Fubini theorem and 
$\|G(\cdot,t)\|_{L^1(\mathbb{R}^n)}=1$ for $t>0$, 
we have 
\[
	\|U\|_{L^1(Q_{(\underline{t},\overline{t})})} =
	c_{n-m}^{-1}\int_{\underline{t}}^{\overline{t}} 
	\int_{\underline{T}}^t \int_{M_s} d\mathcal{H}^m(\xi) ds dt 
	\leq
	\frac{(\overline{t}-\underline{T})^2}{c_{n-m}} 
	\sup_{s\in (\underline{T},\overline{t})} \mathcal{H}^m (M_s). 
\]
This together with Proposition \ref{estVol} shows 
$U\in L_\mathrm{loc}^1(Q_I)$. 
Let 
$\varphi\in C^\infty_0(Q_I \setminus M_I)$. 
By the Fubini theorem, we obtain 
\[
\begin{aligned}
	&\iint_{Q_I} U(x,t) 
	(-\partial_t \varphi(x,t)-\Delta \varphi(x,t) ) dxdt   \\
	& =c_{n-m}^{-1}\iint_{M_{(\underline{T},\infty)}}  
	\iint_{Q_{(s,\infty)}} 
	G(x-\xi,t-s)  
	(-\partial_t \varphi(x,t)-\Delta \varphi(x,t) ) dxdt d\mathcal{H}^m(\xi) ds   \\
	& =c_{n-m}^{-1}\iint_{M_I} 
	\varphi(\xi,s)  d\mathcal{H}^m(\xi) ds \\
	& =0. 
\end{aligned}
\]
Hence $U$ satisfies \eqref{eq:heat} in 
$\mathcal{D}'(Q_I \setminus M_I)$. 
The regularity theory for parabolic equations shows that 
$U\in C^\infty(Q_I \setminus M_I)$ 
and that $U$ satisfies \eqref{eq:heat}. 
\end{proof}

\subsection{Behavior of the solution near the singular set}\label{subsect:bUnsing}
This subsection is devoted to proving estimates 
of $U$ near the singular set $M_t$. 

\begin{proposition}\label{pro:U}
Let $n,m\geq1$ satisfy $n-m\geq 3$ and let $\underline{t}>\underline{T}$. 
Suppose that $F$ satisfies \eqref{eq:bF}. 
Then there exists a constant $\delta'>0$ depending on 
$m$, $B$, $M_{0}$ and $\underline{t}-\underline{T}$ 
such that the following estimates hold. 
For any $\delta\in (0,\delta')$, there exists a constant $C>0$ depending on 
$n$, $m$, $B$, $M_{0}$ and $\delta$ 
such that 
\begin{align}
	& 
	\left| U(x,t)-d(x,M_t)^{-(n-m-2)} \right|\leq 
	\left\{
	\begin{aligned}
	&C\log(1/d(x,M_t))  && \mbox{ if }n-m=3, \\
	&Cd(x,M_t)^{-(n-m-3)}  && \mbox{ if }n-m\geq 4, 
	\label{eq:Uul}  
	\end{aligned}\right.
	\\
	& 
	\left| |\nabla U(x,t)| - (n-m-2) d(x,M_t)^{-(n-m-1)} \right|
	\leq C d(x,M_t)^{-(n-m-2)},
	\label{eq:Unul}  
\end{align}
for any $x\in M^\delta_t \setminus M_t$ and $t\in[\underline{t},\infty)$. 
\end{proposition}

\begin{proof}
We prepare some constants and estimates due to Section \ref{sec:geom}. 
Let $R\in(0,1)$ satisfy \eqref{RBK-1} with $\alpha=1$. 
Then, $R$ depends on $m$, $B$ and $M_{0}$, since $K$ in \eqref{RBK-1} depends on $m$, $B$ and $M_{0}$. 
Write the right hand side of \eqref{nicedelta} by 
\[
\delta_{0}=\delta_{0}(m,B,M_{0})>0. 
\]
By Proposition \ref{tdepLanPro}, there exists a constant $\mathcal{K}=\mathcal{K}(m,B,M_{0})>1$ satisfying
\begin{align}
	&\sup_{t\in I}\sup_{y\in B_{R}}
	\left(
	J\psi_t(y) + |D\psi_t(y)| 
	+ |DJ\psi_t(y)| 
	+ |D^2 \psi_{t}(y)|
	+ |\partial_t\psi_{t}(y)|
	\right)\leq \mathcal{K},       \label{eq:ubK2}\\
	&\inf_{t\in I}\inf_{y\in B_{R}} 
	\min\left\{ 
	\inf_{\eta\in \mathbb{R}^{m}\setminus\{0\}}\frac{|D \psi_t(y) \eta|}{|\eta|}, 
	J\psi_t(y)
	\right\}
	\geq \mathcal{K}^{-1},  \label{eq:lbK2}
\end{align}
where $\psi_{t}:B_{R}\to M_{t}$ is a time-dependent 
Langer chart centered at $F_{t}(x)$ 
with $x\in M_{0}$ as in Section \ref{tdepLC}. 
Put 
\[\delta'=\delta'(m,B,M_{0},\underline{t}-\underline{T})
:=\min\left\{R, \delta_0, \underline{t}-\underline{T}, 1/(8\mathcal{K}^4) \right\}>0.\]
Note that $\delta'<1$ since $R<1$. 
We check that $\delta'$ satisfies the desired property 
by using the computations and estimates 
\eqref{eq:tilUcal}, \eqref{eq:tilnUcal}, \eqref{eq:I1oI1}, \eqref{eq:J1tJ1}, 
\eqref{eq:I23oI23} and \eqref{eq:J23tJ23} below.  
These will be proved as 
Lemmas \ref{lem:oUtU}, \ref{lem:I1oI1}, \ref{lem:J1tJ1} and \ref{lem:IJ23oItJ23} 
following this proposition.

Fix $\delta \in (0,\delta')$, $t\in(\underline{t},\infty)$ and $x\in M_t^\delta\setminus M_t$. 
Let $\overline{x}\in M_t$ satisfy \eqref{nicedelta2}. 
Remark that $\overline{x}$ is uniquely determined by Proposition \ref{tdepnormal}. 
Then, there exists $x_{0}\in M_{0}$ such that $\overline{x}=F_{t}(x_{0})$. 
Let $\psi_{s}:B_{\delta}\to M_{s}$ be a time-dependent 
Langer chart centered at $F_{s}(x_{0})$ 
of radius $\delta$. 
We split each of $U$ and $\nabla U$ into three terms as 
\[
\begin{aligned}
	U(x,t)&=
	c_{n-m}^{-1}\int_{\underline{T}}^t \int_{M_s} 
	G(x-\xi,t-s) d\mathcal{H}^m(\xi) ds \\
	&=
	c_{n-m}^{-1}\left( 
	\int_{t-\delta}^t \int_{M_s\cap \psi_s(B_\delta)} 
	+\int_{t-\delta}^{t} \int_{M_s\setminus \psi_s(B_\delta)}
	+\int_{\underline{T}}^{t-\delta} \int_{M_s}   \right)\\
	&=:
	c_{n-m}^{-1}(I_1 +I_2+I_3), 
\end{aligned}
\]
and 
\[
\begin{aligned}
	\nabla U(x,t)
	&= -(2c_{n-m})^{-1}
	\int_{\underline{T}}^t \int_{M_s} 
	\frac{x-\xi}{t-s} G(x-\xi,t-s) d\mathcal{H}^m(\xi) ds  \\
	&= -(2c_{n-m})^{-1} \left( 
	\int_{t-\delta}^t \int_{M_s\cap \psi_s(B_\delta)}
	+\int_{t-\delta}^t \int_{M_s\setminus \psi_s(B_\delta)}
	+\int_{\underline{T}}^{t-\delta} \int_{M_s}
	\right) \\
	&=: -(2c_{n-m})^{-1} (J_1 + J_2 + J_3). 
\end{aligned}
\]
Define an embedding $\overline{\psi}\in C^\infty(\mathbb{R}^m; \mathbb{R}^n)$ by 
\begin{equation}\label{eq:opsi}
	\overline{\psi}: \mathbb{R}^m \ni \eta \longmapsto 
	\overline{x}+D \psi_t(0)\eta \in 
	T_{\overline{x}}M_t. 
\end{equation}
We define 
\[
\begin{aligned}
	&\overline U(x,t)
	:=c_{n-m}^{-1} \int_{-\infty}^t \int_{T_{\overline{x}}M_t}
	G(x-\xi,t-s) d\mathcal{H}^m(\xi) ds, \\
	&\tilde U(x,t)
	:=  -(2c_{n-m})^{-1}
	\int_{-\infty}^t \int_{T_{\overline{x}}M_t}
	\frac{x-\xi}{t-s} G(x-\xi,t-s) d\mathcal{H}^m(\xi) ds.   
\end{aligned}
\]
Then by Lemma \ref{lem:oUtU} below, we have 
\begin{align}
	\label{eq:tilUcal}
	&\overline U(x,t)
	=|x-\overline{x}|^{-(n-m-2)},   \\
	\label{eq:tilnUcal}
	&\tilde U(x,t)
	= -(n-m-2)|x-\overline{x}|^{-(n-m)}(x-\overline{x}).   
\end{align} 
We also split each of $\overline{U}$ and $\tilde U$ into three terms as 
\[
\begin{aligned}
	\overline U
	&=c_{n-m}^{-1}\left( 
	\int_{t-\delta}^t \int_{T_{\overline{x}}M_t\cap \overline{\psi}(B_\delta)} 
	+\int_{-\infty}^{t-\delta} \int_{T_{\overline{x}}M_t \cap \overline{\psi}(B_\delta)}
	+\int_{-\infty}^t \int_{T_{\overline{x}}M_t \setminus 
	\overline{\psi}(B_\delta)}   \right)\\
	&=:c_{n-m}^{-1}(\overline{I}_1+ \overline{I}_2 + \overline{I}_3),
\end{aligned}
\]
and
\[
\begin{aligned}
	\tilde U
	&=
	-(2c_{n-m})^{-1} \left( 
	\int_{t-\delta}^t \int_{T_{\overline{x}}M_t\cap \overline{\psi}(B_\delta)} 
	+\int_{-\infty}^{t-\delta} \int_{T_{\overline{x}}M_t \cap \overline{\psi}(B_\delta)}
	+\int_{-\infty}^t \int_{T_{\overline{x}}M_t \setminus 
	\overline{\psi}(B_\delta)}   \right)  \\
	& =: 
	-(2c_{n-m})^{-1} (\tilde J_1 + \tilde J_2 + \tilde J_3). 
\end{aligned}
\]
Since $t-\delta>\underline{T}$ for $t\in [\underline{t},\infty)$, 
each of $I_1$, $I_2$, $I_3$, $\overline{I}_1$, $\overline{I}_2$ 
and $\overline{I}_3$ is positive.  
Remark that each of $J_1$, $J_2$, $J_3$, $\tilde J_1$, $\tilde J_2$ and $\tilde J_3$ is 
a vector in $\mathbb{R}^n$. 
Then, 
\begin{align}
	&|U(x,t)-\overline{U}(x,t)| \leq 
	c_{n-m}^{-1}\left( |I_1-\overline{I}_1| 
	+I_2+\overline{I}_2 + I_3+ \overline{I}_3 \right),
	\label{eq:I123oI123}  \\
	&|\nabla U(x,t)- \tilde U(x,t)| \leq 
	( 2c_{n-m})^{-1} \left( |J_1-\tilde J_1| + |J_2| + |\tilde J_2| + 
	|J_3| + |\tilde J_3|  \right)
	\label{eq:J123tJ123}. 
\end{align}
By Lemmas \ref{lem:I1oI1} and \ref{lem:J1tJ1} below, 
we see that 
\begin{align}
	&|I_1-\overline{I}_1| \leq 
	\left\{
	\begin{aligned}
	&C(n,m,\mathcal{K},\delta) \log(1/|x-\overline{x}|)  && \mbox{ if }n-m=3,  \\
	&C(n,m,\mathcal{K}) |x-\overline{x}|^{-(n-m-3)}  && \mbox{ if }n-m\geq4, 
	\end{aligned}
	\right. \label{eq:I1oI1}  \\
	&|J_1-\tilde J_1| \leq 
	C(n,m,\mathcal{K}) |x-\overline{x}|^{-(n-m-2)}.  \label{eq:J1tJ1}
\end{align}
On the other hand, 
one can observe by Lemma \ref{lem:IJ23oItJ23} that 
\begin{align}
	&I_2+\overline{I}_2+I_3+\overline{I}_3\leq C(n,m,B,M_0,\mathcal{K},\delta),  \label{eq:I23oI23}\\
	&|J_2|+|\tilde J_2| + |J_3| + | \tilde J_3| \leq  C(n,m,B,M_0,\mathcal{K},\delta).  
	\label{eq:J23tJ23}  
\end{align}
Hence, plugging 
\eqref{eq:I1oI1}, \eqref{eq:J1tJ1}, 
\eqref{eq:I23oI23} and \eqref{eq:J23tJ23} 
into \eqref{eq:I123oI123} and \eqref{eq:J123tJ123} 
yields \eqref{eq:Uul} and \eqref{eq:Unul} for 
some constant $C=C(n,m,B,M_0,\mathcal{K},\delta)>0$. 
Since $\mathcal{K}$ depends on $m$, $B$ and $M_{0}$, the proposition follows. 
\end{proof}

What is left is to state and prove 
Lemmas \ref{lem:oUtU}, \ref{lem:I1oI1}, \ref{lem:J1tJ1} and \ref{lem:IJ23oItJ23}. 

\begin{lemma}\label{lem:oUtU}
\eqref{eq:tilUcal} and \eqref{eq:tilnUcal} hold. 
\end{lemma}

\begin{proof}
Although this lemma was already proved by \cite[Lemma 2.1]{HTYpre}, 
we repeat their computations for completeness. 
Let $\{b_1^{\overline{x}},\ldots,b_m^{\overline{x}}\}$ 
be an orthonormal basis of $T_{\overline{x}}M_t$. 
We define an embedding 
$\Psi\in C^\infty(\mathbb{R}^m;\mathbb{R}^n)$ by
$\Psi(\eta):= 
\overline{x}+ \eta_1 b_1^{\overline{x}} +\cdots+ \eta_m b_m^{\overline{x}}$ 
for $\eta=(\eta_1,\ldots,\eta_m)$. 
Note that $\Psi(\mathbb{R}^m)=T_{\overline{x}}M_t$ and 
$J\Psi(\eta)=1$. 
From the change of variables, it follows that 
\[
\begin{aligned}
	\overline{U}(x,t)&=
	c_{n-m}^{-1} \int_{-\infty}^t 
	\int_{{\Psi}^{-1}(T_{\overline{x}}M_t)}
	G(x-\Psi(\eta),t-s) 
	J\Psi(\eta) d\eta ds \\
	&=
	c_{n-m}^{-1} \int_{-\infty}^t 
	\int_{\mathbb{R}^m}
	G(x-\Psi(\eta),t-s) d\eta ds. 
\end{aligned}
\]
Since $\overline{x}$ satisfies \eqref{nicedelta2} and 
$\{b_1^{\overline{x}},\ldots,b_m^{\overline{x}}\}$ 
is orthonormal, we have 
\[
	|x-\Psi(\eta)|^2 
	=
	|x-\overline{x}|^2 + | \Psi(\eta)-\overline{x} |^2 
	=
	|x-\overline{x}|^2 + \eta_1^2+\cdots+ \eta_m^2.
\]
Thus, by \eqref{eq:cnm}, 
\[
\begin{aligned}
	\overline{U}(x,t)&=
	c_{n-m}^{-1} \int_{-\infty}^t 
	(4\pi(t-s))^{-\frac{n}{2}}
	e^{-\frac{|x-\overline{x}|^2}{4(t-s)}}
	\int_{\mathbb{R}^m}
	e^{-\frac{\eta_1^2+\cdots+\eta_m^2}{4(t-s)} } d\eta ds \\
	&=
	c_{n-m}^{-1} \int_{-\infty}^t 
	(4\pi(t-s))^{-\frac{n-m}{2}}
	e^{-\frac{|x-\overline{x}|^2}{4(t-s)}} ds \\
	&=
	\frac{|x-\overline{x}|^{-(n-m-2)}}{c_{n-m}(4\pi^{(n-m)/2}) }
	\int_0^\infty 
	\tau^{\frac{n-m-2}{2}-1}  e^{-\tau}  d\tau \\
	&=  
	|x-\overline{x}|^{-(n-m-2)}. 
\end{aligned}
\]
Then \eqref{eq:tilUcal} holds. Since 
$\int_\mathbb{R} \eta_i e^{-\eta_i^2/(4(t-s))} d\eta_i =0$, we have 
\[
	\int_{\mathbb{R}^m}
	\left( \eta_1 b_1^{\overline{x}} +\cdots+ \eta_m b_m^{\overline{x}}  \right)
	e^{-\frac{\eta_1^2 + \cdots +\eta_m^2}{4(t-s)}} d\eta 
	=0, 
\]
and so 
\[
\begin{aligned}
	\tilde U(x,t)&=
	-(2c_{n-m})^{-1}
	\int_{-\infty}^t \int_{\mathbb{R}^m}
	\frac{x-\Psi(\eta)}{t-s} G(x-\Psi(\eta),t-s) J\Psi(\eta) d\eta ds \\
	&=
	-\frac{x-\overline{x}}{2c_{n-m}(4\pi)^{n/2}}
	\int_{-\infty}^t (t-s)^{-\frac{n}{2}-1}
	e^{-\frac{|x-\overline{x}|^2}{4(t-s)}}
	\int_{\mathbb{R}^m} 
	e^{-\frac{\eta_1^2 + \cdots +\eta_m^2}{4(t-s)}}d\eta ds.  
\end{aligned}
\]
Thus,
\[
\begin{aligned}
	\tilde U(x,t)&=
	-\frac{x-\overline{x}}{2c_{n-m}(4\pi)^{(n-m)/2}}
	\int_{-\infty}^t (t-s)^{-\frac{n-m}{2}-1}
	e^{-\frac{|x-\overline{x}|^2}{4(t-s)}}  ds  \\
	&=
	-\frac{(x-\overline{x})|x-\overline{x}|^{-(n-m)}}{2c_{n-m}\pi^{(n-m)/2}}
	\int_0^\infty \tau^{\frac{n-m}{2}-1} e^{-\tau}  d\tau  \\
	&=
	-\frac{(n-m-2)(x-\overline{x})|x-\overline{x}|^{-(n-m)}}{4c_{n-m}\pi^{(n-m)/2}}
	\int_0^\infty \tau^{\frac{n-m-2}{2}-1} e^{-\tau}  d\tau  \\
	&=
	-(n-m-2)(x-\overline{x})|x-\overline{x}|^{-(n-m)}, 
\end{aligned}
\]
and the lemma follows. 
\end{proof}

To  prove Lemmas \ref{lem:I1oI1} and \ref{lem:J1tJ1} below, 
we have to estimate $e^{-|x-\psi_s(\eta)|^2/(4(t-s))}$. 
Note that 
\begin{equation}\label{eq:xpsidcal}
\begin{aligned}
	|x-\psi_s(\eta)|^2 &=
	|x-\psi_t(0)-(\psi_t(\eta)-\psi_t(0))
	+\psi_t(\eta)-\psi_s(\eta)|^2 \\
	&= |x-\psi_t(0)|^2 + |D \psi_t(0)\eta|^2 + f(\eta), 
\end{aligned}
\end{equation}
where $f$ is given by 
\begin{equation}\label{eq:fetadefi}
\begin{aligned}
	f(\eta)&:=
	|\psi_t(\eta)-\psi_t(0)|^2- |D \psi_t(0)\eta|^2  \\
	&\quad
	-2(x-\psi_t(0))\cdot (\psi_t(\eta)-\psi_t(0)) 
	+|\psi_t(\eta)-\psi_s(\eta)|^2 \\
	&\quad 
	+2(x-\psi_t(0))\cdot (\psi_t(\eta)-\psi_s(\eta))  \\
	&\quad 
	-2(\psi_t(\eta)-\psi_t(0))\cdot (\psi_t(\eta)-\psi_s(\eta)). 
\end{aligned}
\end{equation}
We prepare a technical lemma concerning $f(\eta)$. 
This lemma is only used for proving Lemma \ref{lem:I1oI1} below. 

\begin{lemma}
For $\eta\in B_\delta$ and $s\in(t-\delta,t)$, 
\begin{equation}\label{eq:absfeta}
\begin{aligned}
	|f(\eta)|& \leq 
	2\mathcal{K}^2 \left(  |\eta| + |x-\overline{x}| \right) |\eta|^2 
	+ 2\mathcal{K}^2 \left( (t-s) 
	+|x-\overline{x}| + |\eta| \right) (t-s) \\
	&\leq 
	4\mathcal{K}^2 \delta |\eta|^2 
	+ 8\mathcal{K}^2 (t-s). 
\end{aligned}
\end{equation}
\end{lemma}

\begin{proof}
The Taylor expansion yields 
\[
\begin{aligned}
	& \psi_t(\eta)-\psi_t(0) 
	- D \psi_t(0) \eta 
	= R(\eta), \quad 
	&R(\eta):= \int_0^1 (1-\theta) 
	D^2\psi_{t}(\theta \eta)(\eta,\eta)
	d\theta.  
\end{aligned}
\]
The estimate \eqref{eq:ubK2} gives 
\begin{equation}\label{eq:Reta}
	|R(\eta)| \leq \mathcal{K} |\eta|^2/2.
\end{equation}
Thus, 
\begin{equation}\label{eq:psit3}
\begin{aligned}
	\left| |\psi_t(\eta)-\psi_t(0)|^2- |D \psi_t(0)\eta|^2 \right|&=
	|R(\eta) \cdot ( R(\eta) + 2D\psi_t(0)\eta )|  \\
	&\leq
	\frac{\mathcal{K}}{2} |\eta|^2 
	\left( 
	\frac{\mathcal{K}}{2} |\eta|^2  +2\mathcal{K}|\eta|
	\right) \\
	&\leq
	2\mathcal{K}^2 |\eta|^3. 
\end{aligned}
\end{equation}
Since $\psi_t(0)=\overline{x}$ satisfies \eqref{nicedelta2}, 
we have 
\begin{equation}\label{eq:psitinn}
	|(x-\psi_t(0))\cdot (\psi_t(\eta)-\psi_t(0))| =
	|(x-\psi_t(0))\cdot R(\eta)| \leq
	\frac{\mathcal{K}}{2} |x-\overline{x}| |\eta|^2.
\end{equation}
From \eqref{eq:ubK2} and $\psi_t(0)=\overline{x}$, it follows that 
\begin{equation}\label{eq:xpsiu}
\begin{aligned}
	&|\psi_t(\eta)-\psi_s(\eta)|^2 \leq \mathcal{K}^2 (t-s)^2, \\
	&|(x-\psi_t(0))\cdot (\psi_t(\eta)-\psi_s(\eta))| \leq
	\mathcal{K}|x-\overline{x}|(t-s), \\
	&|(\psi_t(\eta)-\psi_t(0))\cdot (\psi_t(\eta)-\psi_s(\eta))| \leq
	\mathcal{K}^2 |\eta| (t-s). 
\end{aligned}
\end{equation}

Plugging \eqref{eq:psit3}, \eqref{eq:psitinn} and \eqref{eq:xpsiu} 
into \eqref{eq:fetadefi} 
yields
\begin{align}
	\label{eq:fuway}
	&
	\begin{aligned}
	f(\eta) &\leq
	2\mathcal{K}^2 |\eta|^3 +\frac{\mathcal{K}}{2} |x-\overline{x}| |\eta|^2 \\
	&\quad +\mathcal{K}^2 (t-s)^2 
	+2\mathcal{K}|x-\overline{x}|(t-s)
	+2\mathcal{K}^2 |\eta| (t-s),  
	\end{aligned}
	\\
	\label{eq:flway}
	&f(\eta) \geq
	-2\mathcal{K}^2 |\eta|^3 
	-\frac{\mathcal{K}}{2} |x-\overline{x}| |\eta|^2 
	-2\mathcal{K}|x-\overline{x}|(t-s)
	-2\mathcal{K}^2 |\eta| (t-s).
\end{align}
Hence \eqref{eq:fuway}, \eqref{eq:flway} and $\mathcal{K}>1$ show the lemma. 
\end{proof}

In what follows, 
we frequently use the following type of the change of variables such as 
\[
	\int_{M_s\cap \psi_s(B_\delta)}
	G(x-\xi,t-s) d\mathcal{H}^m(\xi) 
	=
	\int_{B_\delta}G(x-\psi_s(\eta),t-s) J\psi_s(\eta) d\eta. 
\]
When no confusion can arise, we write 
\begin{equation}\label{eq:GoGdef}
	G=G(x-\psi_s(\eta),t-s), \quad 
	\overline{G}=G(x-\overline{\psi}(\eta),t-s),
\end{equation}

\begin{lemma}\label{lem:I1oI1}
\eqref{eq:I1oI1} holds. 
\end{lemma}

\begin{proof}
The change of variables and 
$J\overline{\psi}(\eta)=J\psi_t(0)$ yield 
\begin{equation}\label{eq:Hdef}
	\overline{I}_1 - I_1= 
	\int_{t-\delta}^t \int_{B_\delta}
	H(\eta,s) d\eta ds, \quad 
	H(\eta,s):=\overline{G}J\psi_t(0)  - G J\psi_s(\eta).  
\end{equation}
We first claim that 
\begin{equation}\label{eq:Hbd}
\begin{aligned}
	|H(\eta,s)| & \leq 
	C(n,\mathcal{K})(t-s)^{-\frac{n}{2}}
	e^{-\frac{|x-\overline{x}|^2}{4(t-s)} } \\
	&\quad \times  
	\left( 
	\frac{\left(  |\eta| + |x-\overline{x}| \right) |\eta|^2 }{t-s}
	+  (t-s) +|x-\overline{x}| + |\eta| 
	\right)  
	e^{-\frac{|\eta|^2}{8\mathcal{K}^2(t-s)} }
\end{aligned}
\end{equation}
for $\eta\in B_\delta$ and $s\in(t-\delta,t)$. 
We write 
\[
	H =\left( \overline{G}- G \right) J\psi_s(\eta)
	+  \overline{G} ( J\psi_t(0) - J\psi_s(\eta) ) 
	=:  H_1 + H_2
\]
and estimate $|H_1|$ and $|H_2|$. 
By \eqref{eq:lbK2}, we have 
\begin{equation}\label{eq:xoxpsi}
	|x-\overline{\psi}(\eta)|^2 =|x-\overline{x}|^2 + |D \psi_t(0)\eta|^2 
	\geq |x-\overline{x}|^2 + \mathcal{K}^{-2}|\eta|^2.  
\end{equation}
This together with \eqref{eq:xpsidcal} shows that 
\[
\begin{aligned}
	\left| e^{ -\frac{|x-\overline{\psi}(\eta)|^2}{4(t-s)} }
	- e^{ -\frac{|x-\psi_s(\eta)|^2}{4(t-s)} }  \right|
	&=
	e^{-\frac{|x-\overline{x}|^2}{4(t-s)} }
	e^{-\frac{|D \psi_t(0)\eta|^2}{4(t-s)} } 
	\left| 1-e^{-\frac{f}{4(t-s)}} \right|  \\
	&=
	e^{-\frac{|x-\overline{x}|^2}{4(t-s)} }
	e^{-\frac{|D \psi_t(0)\eta|^2}{4(t-s)} } 
	\left| \frac{f}{4(t-s)}\int_0^1 e^{-\frac{f}{4(t-s)}\theta} d\theta \right| \\
	&\leq 
	e^{-\frac{|x-\overline{x}|^2}{4(t-s)} }
	e^{-\frac{|\eta|^2}{4\mathcal{K}^2 (t-s)} } 
	\frac{|f|}{4(t-s)}  e^{\frac{|f|}{4(t-s)}}, 
\end{aligned}
\]
where $f$ is given by \eqref{eq:fetadefi}. 
From \eqref{eq:lbK2} and \eqref{eq:absfeta}, 
it follows that 
\[
\begin{aligned}
	&\left| e^{ -\frac{|x-\overline{\psi}(\eta)|^2}{4(t-s)} }
	- e^{ -\frac{|x-\psi_s(\eta)|^2}{4(t-s)} }  \right|   \\
	&\leq 
	C(\mathcal{K})
	\left( 
	\frac{\left(  |\eta| + |x-\overline{x}| \right) |\eta|^2 }{t-s}
	+  (t-s) +|x-\overline{x}| + |\eta| 
	\right)  
	e^{-\frac{|x-\overline{x}|^2}{4(t-s)} }
	e^{-\frac{|\eta|^2}{4\mathcal{K}^2(t-s)} } 
	e^{\frac{\mathcal{K}^2 \delta |\eta|^2 }{t-s}}.   
\end{aligned}
\]
Then, by 
$\mathcal{K}^2 \delta \leq 1/(8\mathcal{K}^2)$ and 
\eqref{eq:ubK2}, we see that $|H_1|$ is bounded by the right hand side of 
\eqref{eq:Hbd}. 
Each of the estimates \eqref{eq:xoxpsi} and \eqref{eq:ubK2} gives
\begin{gather}
	\overline{G}\leq C(n)(t-s)^{-\frac{n}{2}}e^{-\frac{|x-\overline{x}|^2}{4(t-s)}}
	e^{-\frac{|\eta|^2}{4\mathcal{K}^2(t-s)}}, 
	\label{eq:oGes}  \\
	|J \psi_t(0)-J\psi_s(\eta)| \leq 
	|J \psi_t(0)-J\psi_t(\eta)| + |J \psi_t(\eta)-J\psi_s(\eta)|  
	\leq
	\mathcal{K} (|\eta| +(t-s) ).   
	\notag
\end{gather}
Then $|H_2|$ is also bounded by the right hand side of \eqref{eq:Hbd}, 
and the claim follows.

The claim and $t-s\leq (t-s)^{1/2}$ for $s\in(t-\delta,t)$ show that 
\begin{equation}\label{eq:sfcalc}
\begin{aligned}
	|\overline{I}_1 - I_1| 
	&\leq C(n,m,\mathcal{K}) \int_{t-\delta}^t (t-s)^{-\frac{n-m}{2}} 
	\left( (t-s)^\frac{1}{2} +|x-\overline{x}| \right)
	e^{-\frac{|x-\overline{x}|^2}{4(t-s)} }  ds \\
	&\leq
	C(n,m,\mathcal{K})\int_{t-\delta}^t (t-s)^{-\frac{n-m-3}{2}-1} 
	e^{-\frac{|x-\overline{x}|^2}{4(t-s)} }  ds  \\
	&\quad 
	+C(n,m,\mathcal{K})|x-\overline{x}| \int_{-\infty}^t (t-s)^{-\frac{n-m-2}{2}-1} 
	e^{-\frac{|x-\overline{x}|^2}{4(t-s)} }  ds \\
	&\leq
	C(n,m,\mathcal{K})|x-\overline{x}|^{-(n-m-3)} \left( 
	\int_{|x-\overline{x}|^2/\delta}^\infty \tau^{\frac{n-m-3}{2}-1} 
	e^{-\frac{1}{4}\tau }  d\tau  + 1 
	\right). 
\end{aligned}
\end{equation}
From this and the following computations
\[
\begin{aligned}
	\int_{|x-\overline{x}|^2/\delta}^\infty \tau^{\frac{n-m-3}{2}-1} 
	e^{-\frac{1}{4}\tau }  d\tau    
	&\leq 
	\left\{
	\begin{aligned}
	&\int_{|x-\overline{x}|^2/\delta}^{1/\delta} \tau^{-1}  d\tau
	+ \int_{1/\delta}^\infty \tau^{-1} 
	e^{-\frac{1}{4}\tau }  d\tau  && (n-m=3)  \\
	& 
	\int_0^\infty \tau^{\frac{n-m-3}{2}-1} 
	e^{-\frac{1}{4}\tau }  d\tau && (n-m\geq4) 
	\end{aligned}
	\right.  \\
	&\leq 
	\left\{
	\begin{aligned}
	&C(\delta)\log(1/|x-\overline{x}|)  && (n-m=3)  \\
	& C(n,m)  && (n-m\geq4), 
	\end{aligned}
	\right.
\end{aligned}
\]
it follows that \eqref{eq:I1oI1} holds. 
Then the lemma follows.  
\end{proof}

\begin{lemma}\label{lem:J1tJ1}
\eqref{eq:J1tJ1} holds. 
\end{lemma}

\begin{proof}
Since this lemma can be proved by similar computations to 
the proof of Lemma \ref{lem:I1oI1}, 
we only give an outline of the proof. 
We rewrite $\tilde J_1-J_1$ as  
\[
\begin{aligned}
	\tilde J_1- J_1 
	&= 
	\int_{t-\delta}^t \int_{B_\delta}
	\left( 
	\frac{x-\overline{\psi}(\eta)}{t-s}
	\overline {G} J\psi_t(0) 
	- \frac{x-\psi_s(\eta)}{t-s} G J\psi_s(\eta) 
	\right) d\eta ds \\
	&=
	\int_{t-\delta}^t \int_{B_\delta}
	\left( 
	-\frac{\psi_t(\eta)-\psi_t(0)}{t-s}
	+\frac{x-\overline{x}}{t-s}
	+\frac{\psi_t(\eta)-\psi_s(\eta)}{t-s}
	\right)
	H d\eta ds \\
	&\quad +
	\int_{t-\delta}^t \int_{B_\delta}
	\left( 
	\frac{\psi_t(\eta)-\psi_t(0)-\nabla \psi_t(0)\eta}{t-s} 
	-  \frac{\psi_t(\eta)-\psi_s(\eta)}{t-s}
	\right)
	\overline{G}J\psi_t(0) 
	d\eta ds \\
	&=: J_{1,1}+J_{1,2},  
\end{aligned}
\]
where $G$, $\overline{G}$ and $H$ are given by \eqref{eq:GoGdef} and \eqref{eq:Hdef}. 
From \eqref{eq:ubK2}, \eqref{eq:Hbd} and similar computations to \eqref{eq:sfcalc}, 
it follows that 
\[
\begin{aligned}
	|J_{1,1}| &\leq 
	C(n,\mathcal{K})\int_{t-\delta}^t 
	(t-s)^{-\frac{n}{2}-1}
	e^{-\frac{|x-\overline{x}|^2}{4(t-s)} } 
	\int_{B_\delta}
	\left( |\eta| 
	+|x-\overline{x}|  
	+ (t-s)
	\right)  \\
	&\quad \times 
	\left( 
	\frac{\left(  |\eta| + |x-\overline{x}| \right) |\eta|^2 }{t-s}
	+  (t-s) +|x-\overline{x}| + |\eta| 
	\right)  
	e^{-\frac{|\eta|^2}{8\mathcal{K}^2(t-s)} }
	 d\eta ds  \\ 
	&\leq 
	C(n,m,\mathcal{K})\int_{-\infty}^t 
	\left( 
	(t-s)^{-\frac{n-m-2}{2}-1}
	+|x-\overline{x}|^2 (t-s)^{-\frac{n-m}{2}-1}
	\right)
	e^{-\frac{|x-\overline{x}|^2}{4(t-s)} }  ds  \\
	&\leq 
	C(n,m,\mathcal{K})|x-\overline{x}|^{-(n-m-2)}. 
\end{aligned}
\]
Moreover, by \eqref{eq:Reta}, \eqref{eq:ubK2} and \eqref{eq:oGes}, 
we can calculate that 
\[
\begin{aligned}
	|J_{1,2}| &\leq 
	C(n,\mathcal{K})\int_{t-\delta}^t 
	(t-s)^{-\frac{n}{2}-1}
	e^{-\frac{|x-\overline{x}|^2}{4(t-s)} } 
	\int_{B_\delta}
	\left( |\eta|^2 + (t-s) \right)
	e^{-\frac{|\eta|^2}{4\mathcal{K}^2(t-s)} }
	 d\eta ds  \\ 
	&\leq 
	C(n,m,\mathcal{K})\int_{-\infty}^t 
	(t-s)^{-\frac{n-m-2}{2}-1}
	e^{-\frac{|x-\overline{x}|^2}{4(t-s)} } ds \\
	&\leq 
	C(n,m,\mathcal{K})|x-\overline{x}|^{-(n-m-2)}. 
\end{aligned}
\]
Then the lemma follows. 
\end{proof}

\begin{lemma}\label{lem:IJ23oItJ23}
\eqref{eq:I23oI23} and \eqref{eq:J23tJ23} hold. 
\end{lemma}

\begin{proof}
Note that, by $|x-\xi| e^{-|x-\xi|^2/(8(t-s))} \leq 2(t-s)^{1/2} e^{-1/4}$, 
we have 
\begin{equation}\label{eq:Gxxi}
\begin{aligned}
	\left( 1+ \frac{|x-\xi|}{t-s} \right)  
	e^{-\frac{|x-\xi|^2}{4(t-s)}} &\leq 
	\left( e^{-\frac{|x-\xi|^2}{8(t-s)}} 
	+\frac{|x-\xi|}{t-s} e^{-\frac{|x-\xi|^2}{8(t-s)}} 
	\right)
	 e^{-\frac{|x-\xi|^2}{8(t-s)}} \\
	 &\leq
	\left( 1 + (t-s)^{-\frac{1}{2}} \right)
	 e^{-\frac{|x-\xi|^2}{8(t-s)}}.  
\end{aligned}
\end{equation}
This together with Propositions \ref{tdistcomp} and \ref{estVol} gives 
\[
\begin{aligned}
	I_2 +|J_2| 
	&\leq 
	C(n)\int_{t-\delta}^t (t-s)^{-\frac{n}{2}} 
	\left( 1  +  (t-s)^{-\frac{1}{2}} \right)  
	e^{-\frac{a^2}{8(t-s)}}
	\int_{M_s\setminus \psi_s(B_\delta)} 
	d\mathcal{H}^m(\xi)  ds  \\
	&\leq
	C(n,m,B)\int_{-\infty}^t (t-s)^{-\frac{n}{2}} 
	\left( 1  +  (t-s)^{-\frac{1}{2}} \right)  
	e^{-\frac{a^2}{8(t-s)}} ds  \\
	&\leq
	C(n,m,B,a), 
\end{aligned}
\]
where $a:=B^{-2}(\kappa(M_{0}))^{-1}\delta>0$. 
By \eqref{eq:Gxxi} and Proposition \ref{estVol}, we also have
\[
\begin{aligned}
	I_3+|J_3| &\leq
	C(n)\int_{\underline{T}}^{t-\delta} (t-s)^{-\frac{n}{2}} 
	\left( 1  +  (t-s)^{-\frac{1}{2}} \right)  
	\int_{M_s}  d\mathcal{H}^m(\xi) ds  \\
	&\leq
	C(n,m,B) \int_{-\infty}^{t-\delta} (t-s)^{-\frac{n}{2}} 
	\left( 1  +  (t-s)^{-\frac{1}{2}} \right)  ds  \\
	&\leq C(n,m,B,\delta). 
\end{aligned}
\]

From $J\overline{\psi}(\eta)=J\psi_t(0)$, \eqref{eq:Gxxi} and \eqref{eq:ubK2}, 
it follows that 
\[
\begin{aligned}
	\overline{I}_2 + |\tilde J_2| 
	&\leq C(n) \int_{-\infty}^{t-\delta} \int_{B_\delta}
	\left( 
	1+\frac{|x-\overline{\psi}(\eta)|}{t-s} 
	\right)
	G(x-\overline{\psi}(\eta),t-s) J\psi_t(0) d\eta ds  \\
	&\leq C(n,\mathcal{K}) \int_{-\infty}^{t-\delta} 
	(t-s)^{-\frac{n}{2}}
	\left( 1 + (t-s)^{-\frac{1}{2}} \right)
	\int_{B_\delta}
	 d\eta ds  \\
	&\leq 
	C(n,m,\delta,\mathcal{K}). 
\end{aligned}
\]
By \eqref{eq:opsi} and  \eqref{eq:lbK2}, we have 
$|x-\overline{\psi}(\eta)|^2 =
|x-\overline{x}|^2 + 
|D \psi_t (0)\eta |^2 \geq 
|\eta|^2/\mathcal{K}^2$. 
Thus, 
\[
\begin{aligned}
	&\overline{I}_3+|\tilde J_3| \\
	&\leq C(n,\mathcal{K}) \int_{-\infty}^t 
	(t-s)^{-\frac{n}{2}} \left( 1 + (t-s)^{-\frac{1}{2}} \right)
	\int_{\mathbb{R}^m \setminus B_\delta} e^{-\frac{|x-\overline{\psi}(\eta)|^2}{8(t-s)}} 
	d\eta ds  \\
	&\leq 
	C(n,\mathcal{K}) \int_{-\infty}^t (t-s)^{-\frac{n}{2}} 
	\left( 1 + (t-s)^{-\frac{1}{2}} \right)
	e^{-\frac{\delta^2}{16\mathcal{K}^2(t-s)}}
	\int_{\mathbb{R}^m\setminus B_{\delta}} 
	e^{-\frac{|\eta|^2}{16\mathcal{K}^2(t-s)}} d\eta ds  \\
	&\leq 
	C(n,m,\mathcal{K}) \int_{-\infty}^t 
	\left( 
	(t-s)^{-\frac{n-m-2}{2}-1} 
	+(t-s)^{-\frac{n-m-1}{2}-1} 
	\right)
	e^{-\frac{\delta^2}{16\mathcal{K}^2(t-s)}}  ds \\
	&\leq C(n,m,\delta,\mathcal{K}). 
\end{aligned}
\]
The above estimates show the desired inequalities. 
\end{proof}

As proved before, 
the proof of Proposition \ref{pro:U} is now complete. 
We can estimate $|\partial_{x_i}\partial_{x_j} U|$ 
in the same way as in the proof of Proposition \ref{pro:U}, 
and so is $|\partial_t U|$ by \eqref{eq:heat}. 
Hence, in addition to \eqref{eq:Uul} and \eqref{eq:Unul}, 
the following holds. 

\begin{proposition}\label{pro:EUnsg}
Under the same assumptions as in Proposition \ref{pro:U}, 
there exists a constant $\delta'>0$ depending on 
$m$, $B$, $M_{0}$ and $\underline{t}-\underline{T}$ 
such that the following estimates hold. 
For any $\delta\in (0,\delta')$, there exists a constant $C>0$ depending on 
$n$, $m$, $B$, $M_{0}$ and $\delta$ 
such that 
\[
\begin{aligned}
	&C^{-1}d(x,M_t)^{-(n-m-2)} \leq   U(x,t) \leq Cd(x,M_t)^{-(n-m-2)}, 
	\\
	& |\nabla U(x,t)| \leq C d(x,M_t)^{-(n-m-1)}, \\ 
	&|\partial_t U(x,t)| + \sum_{i,j=1}^n |\partial_{x_i}\partial_{x_j} U(x,t)| \leq 
	C d(x,M_t)^{-(n-m)},   
\end{aligned}
\]
for any $x\in M^\delta_t \setminus M_t$ and $t\in[\underline{t},\infty)$. 
\end{proposition}

\subsection{Uniform estimates of the solution}\label{subsect:Uunif}
We first derive estimates uniform in $Q_I\setminus M_I$. 

\begin{proposition}\label{pro:decay}
Let $n,m\geq1$ satisfy $n-m\geq 3$.  
Suppose that $F$ satisfies \eqref{eq:bF} and \eqref{eq:HF}. 
There exists a constant $C>0$ depending on 
$n$, $m$, $B$ and $M_0$ 
such that 
\begin{align}
	\label{eq:Uuni}
	&U(x,t)\leq C d(x,M_t)^{-(n-2)}, \quad 
	(x,t)\in Q_I\setminus M_I, \\
	\label{eq:nUuni}
	&|\nabla U(x,t)| \leq 
	C  \left( d(x,M_t)^{-n} + d(x,M_t)^{-(n-2)} \right), \quad 
	(x,t)\in Q_I\setminus M_I, 
\end{align}
\end{proposition}

\begin{proof}
Let $(x,t)\in  Q_I\setminus M_I$. 
If $\xi\in M_s$, then $F_{t}\circ F_{s}^{-1}(\xi)\in M_t$. 
Hence by \eqref{eq:HF},  
\begin{equation}\label{eq:xxil}
\begin{aligned}
	|x-\xi|^2 &=|x-F_{t}\circ F_{s}^{-1}(\xi) 
	-(F_{s}\circ F_{s}^{-1}(\xi)-F_{t}\circ F_{s}^{-1}(\xi))|^2  \\
	&\geq
	\frac{1}{2}|x-F_{t}\circ F_{s}^{-1}(\xi)|^2 
	- |F_{s}\circ F_{s}^{-1}(\xi)-F_{t}\circ F_{s}^{-1}(\xi)|^2  \\
	&\geq
	\frac{1}{2}d(x,M_t)^2 - B^2 (t-s). 
\end{aligned}
\end{equation}
By $F_{s}^{-1}(\xi)\in M_0$, \eqref{eq:bF} and Lemma \ref{diambound}, 
we have 
\begin{equation}\label{eq:xxil1}
\begin{aligned}
	|x-\xi| 
	&=|x-F_{t}\circ F_{s}^{-1}(\xi) 
	-(F_{s}\circ F_{s}^{-1}(\xi)-F_{t}\circ F_{s}^{-1}(\xi))|  \\
	&\leq
	|x-F_{t}\circ F_{s}^{-1}(\xi)| + 
	|F_{s}\circ F_{s}^{-1}(\xi)-F_{t}\circ F_{s}^{-1}(\xi)|  \\
	&\leq
	d(x,M_t) + \mathrm{diam} M_t  + B(t-s)\\
	&\leq
	C(B,M_0) ( 1 + d(x,M_t)+ (t-s) ). 
\end{aligned}
\end{equation}
From Proposition \ref{estVol} and \eqref{eq:xxil}, it follows that 
\[
\begin{aligned}
	U(x,t) &\leq 
	C(n,m) \int_{\underline{T}}^t (t-s)^{-\frac{n}{2}} 
	e^{-\frac{d(x,M_t)^2}{8(t-s)}} e^{B^2} \int_{M_s} d\mathcal{H}^m(\xi) ds \\
	&\leq 
	C(n,m,B) 
	\int_{-\infty}^t (t-s)^{-\frac{n}{2}} 
	e^{-\frac{d(x,M_t)^2}{8(t-s)}}  ds  \\
	&\leq 
	C(n,m,B) d(x,M_t)^{-(n-2)}.
\end{aligned}
\]
Then \eqref{eq:Uuni} holds. 
Proposition \ref{estVol}, \eqref{eq:xxil} and \eqref{eq:xxil1} imply that 
\[
\begin{aligned}
	|\nabla U(x,t)| &\leq 
	C \int_{\underline{T}}^t 
	\frac{1+d(x,M_t)+t-s}{t-s}
	(t-s)^{-\frac{n}{2}} 
	e^{-\frac{d(x,M_t)^2}{8(t-s)}} 
	\int_{M_s} d\mathcal{H}^m(\xi) ds \\
	&\leq 
	C\int_{-\infty}^t \left( (1+d(x,M_t))
	(t-s)^{-\frac{n}{2}-1} + (t-s)^{-\frac{n-2}{2}-1} 
	\right)
	e^{-\frac{d(x,M_t)^2}{8(t-s)}}  ds  \\
	&\leq 
	C\left( d(x,M_t)^{-n}  
	+d(x,M_t)^{-(n-1)}
	+d(x,M_t)^{-(n-2)} \right)
\end{aligned}
\]
with some constant $C=C(n,m,B,M_0)>0$, and \eqref{eq:nUuni} holds.  
\end{proof}

We next give a uniform estimate of $U$ in an intermediate region. 

\begin{proposition}\label{pro:interm}
Let $n,m\geq1$ satisfy $n-m\geq 3$ and let 
$\underline{t}>\underline{T}$. 
Suppose that $F$ satisfies \eqref{eq:bF} and \eqref{eq:HF}. 
Fix $c_0$ and $C_0$ with $c_0<C_0$. 
Then there exists a constant $C>0$ depending on 
$n$, $m$, $B$, $M_0$, 
$\underline{t}-\underline{T}$, 
$c_0$ and $C_0$ such that
\begin{equation}\label{eq:interm}
	C^{-1} \leq U(x,t)\leq C, \quad 
	x\in M^{C_0}_t\setminus M^{c_0}_t, t\in[\underline{t},\infty). 
\end{equation}
\end{proposition}

\begin{proof}
Let $x\in M^{C_0}_t\setminus M^{c_0}_t$ and $t\in[\underline{t},\infty)$. 
By \eqref{eq:Uuni}, $U(x,t)\leq C(n,m,B,M_0,c_0)$. 
From $F_{s}^{-1}(\xi)\in M_0$, \eqref{eq:HF} and Lemma \ref{diambound}, 
it follows that  
\[
\begin{aligned}
	|x-\xi|^2 
	&=|x-F_{t}\circ F_{s}^{-1}(\xi) 
	-(F_{s}\circ F_{s}^{-1}(\xi)-F_{t}\circ F_{s}^{-1}(\xi))|^2  \\
	&\leq
	2|x-F_{t}\circ F_{s}^{-1}(\xi)|^2 
	+ 2|F_{s}\circ F_{s}^{-1}(\xi)-F_{t}\circ F_{s}^{-1}(\xi)|^2  \\
	&\leq
	4d(x,M_t)^2+4(\mathrm{diam} M_t)^2 + 2 B^2 (t-s), \\
	&\leq
	C(B,M_0,C_0)+ 2 B^2 (t-s). 
\end{aligned}
\]
Then Proposition \ref{estVol} gives 
\[
\begin{aligned}
	U(x,t) &\geq 
	C(n,m,B)^{-1} \int_{\underline{T}}^t (t-s)^{-\frac{n}{2}} 
	e^{-\frac{C(B,M_0,C_0)}{4(t-s)}} ds \\
	&\geq 
	C(n,m,B,M_0,C_0)^{-1} 
	\int_{\frac{C(B,M_0,C_0)}{t-\underline{T}}}^\infty \tau^{\frac{n-2}{2}-1} 
	e^{-\frac{1}{4}\tau}  d\tau  \\
	&\geq 
	C(n,m,B,M_0,C_0)^{-1} 
	\int_{\frac{C(B,M_0,C_0)}{\underline{t}-\underline{T}}}^\infty \tau^{\frac{n-2}{2}-1} 
	e^{-\frac{1}{4}\tau}  d\tau \\
	&\geq 
	C(n,m,B,M_0,\underline{t}-\underline{T},C_0)^{-1}. 
\end{aligned}
\]
Hence \eqref{eq:interm} follows. 
\end{proof}

\section{Super- and sub-solution methods}\label{sec:supersub}
Let $I=(\underline{T},\infty)$ with $\underline{T}\in(-\infty,0)$ 
and let $F$ satisfy \eqref{eq:bF}. 
Our aim of this section is to give a method for solving 
the following initial value problem 
\begin{equation}\label{eq:unbddip}
\left\{
\begin{aligned}
	&\partial_t u-\Delta u = f(u) && \mbox{ in }Q_{(0,\infty)}\setminus M_{(0,\infty)}, \\
	&u(\cdot,0)= u_0 && \mbox{ on }\mathbb{R}^n\setminus M_0  
\end{aligned}
\right. 
\end{equation}
with $f\in C^1(\mathbb{R})$ and $u_0\in C(\mathbb{R}^n\setminus M_0)$ under the condition that 
super-solutions and sub-solutions exist. 
We prepare an existence theorem for bounded domains in Subsection \ref{subsec:bddss}. 
In Subsection \ref{subsec:unbddss}, 
we show the existence of solutions of \eqref{eq:unbddip}. 

\subsection{An existence theorem for bounded domains}\label{subsec:bddss}
Let $T>0$ and let $Q\subset \mathbb{R}^n\times(0,T)$ be a bounded domain in $\mathbb{R}^n\times\mathbb{R}$. 
We denote by $\overline{Q}_{\tau}\subset \mathbb{R}^{n}$ 
the time slice of $\overline{Q}$ at $t=\tau$. 
Then, $\overline{Q}=\cup_{t\in[0,T]}(\overline{Q}_{t}\times\{t\})$. 
We denote by $\mathcal{B}_t$ the interior of $\overline{Q}_{t}$. 
Set
\[
\left\{
\begin{aligned}
	&
	B_t:= \mathcal{B}_t\times\{t\}   \quad (0\leq t\leq T), \\
	&S:= \{ (x,t)\in \mathbb{R}^n\times \mathbb{R}; x\in \partial\mathcal{B}_t, 0<t<T\}, \\
	&\Gamma:= \overline{B_0}\cup \overline{S}.
\end{aligned}
\right.
\]
We impose the following conditions on $Q$ due to \cite[SECTION 3, CHAPTER 2]{Frbook}. 
\begin{itemize}
\item[(i)]
$S\neq \emptyset$ and $B_t\neq \emptyset$ for any $t\in[0,T]$. 
\item[(ii)]
$\partial Q= \overline{B}_0 \cup \overline{S} \cup \overline{B}_T$. 
\item[(iii)]
$B_t$ is a smooth domain in $\mathbb{R}^n\times\{t\}$ for any $t\in (0,T)$. 
\item[(iv)]
$S$ is an $n$-dimensional smooth manifold without boundary. 
\item[(v)]
For any $P_0\in B_0$ and $P_T\in B_T$, there exists a simple continuous curve 
in $Q$ joining $P_0$ to $P_T$ along which the $t$-coordinate is increasing. 
\end{itemize}

Let us consider 
\begin{equation}\label{eq:bddibp}
\left\{
\begin{aligned}
	&\partial_t u-\Delta u = f(u) && \mbox{ in } Q, \\
	&u= g && \mbox{ on }\Gamma, 
\end{aligned}
\right. 
\end{equation}
where $n \geq 1$, $f\in C^1(\mathbb{R})$ and $ g\in C(\Gamma)$. 
We say that $\overline{u}$ ($\underline{u}$) 
is a \emph{super-solution} (\emph{sub-solution}) of the problem \eqref{eq:bddibp} 
if $\overline{u}$ ($\underline{u}$) 
belongs to $C^{2,1}(Q) \cap C(\overline{Q})$ and satisfies 
\[
\left\{
\begin{aligned}
	&\partial_t u-\Delta u \geq (\leq) f(u) && \mbox{ in } Q, \\
	&u\geq (\leq) g && \mbox{ on }\Gamma. 
\end{aligned}
\right. 
\]
We prove an existence theorem for bounded domains. 
Remark that the theorem can be proved by 
the modification of the proof of \cite[Theorem 6.1]{Hubook} 
(see also \cite{Sa72} and \cite[Lemma 1.2]{Wa93}). 
However, for the completeness of this paper, we give a proof. 

\begin{theorem}\label{th:bddss}
Let $n \geq 1$, $f\in C^1(\mathbb{R})$ and $ g\in C(\Gamma)$. 
Suppose that $Q$ satisfies {\rm(i)}--{\rm(v)}. 
If \eqref{eq:bddibp} has a super-solution $\overline{u}$ and a sub-solution $\underline{u}$ 
satisfying $\underline{u}\leq \overline{u}$ on $\overline{Q}$, 
then \eqref{eq:bddibp} has a solution $u\in C^{2,1}(Q) \cap C(\overline{Q})$ 
satisfying $\underline{u}\leq u\leq \overline{u}$ on $\overline{Q}$.  
\end{theorem}

\begin{proof}
Fix $\tilde g\in C(\overline{Q})$ such that $\tilde g |_\Gamma =g$. 
Let $c_1:= \max_{\overline{Q}} \overline{u}$, 
$c_2:= \min_{\overline{Q}} \underline{u}$ and 
$c:= \max_{[c_1,c_2]} (\max\{ -f'(u), 0 \})$. 
Define $F(u):= cu + f(u)$. 
Then, $F(u)$ is nondecreasing in $u\in [c_1,c_2]$. 

Let $\underline{u}_0:=\underline{u}$. For each $N\geq1$, we denote 
by $\underline{u}_N\in C^{2,1}(Q) \cap C(\overline{Q}) $ the solution of 
\[
\left\{
\begin{aligned}
	&\partial_t \underline{u}_{N}-\Delta \underline{u}_{N} 
	+ c \underline{u}_{N}= F(\underline{u}_{N-1}) && \mbox{ in } Q, \\
	&\underline{u}_{N}= \tilde g && \mbox{ on }\Gamma. 
\end{aligned}
\right. 
\]
Remark that the existence, the uniqueness and the regularity of $\underline{u}_N$ 
follow from the conditions (i)--(v) and \cite[Corollary 1, SECTION 5, CHAPTER 3]{Frbook}. 
We also denote by $\overline{u}_1$ the solution of 
\[
\left\{
\begin{aligned}
	&\partial_t \overline{u}_1-\Delta \overline{u}_1 
	+ c \overline{u}_1= F(\overline{u}) && \mbox{ in } Q, \\
	&\overline{u}_1= \tilde g && \mbox{ on }\Gamma. 
\end{aligned}
\right. 
\]

We claim that, for $N\geq 0$, the following inequalities hold. 
\begin{equation}\label{eq:cluuN}
	\underline{u}\leq \underline{u}_N \leq \overline{u}, \quad 
	\underline{u}_{N} \leq \underline{u}_{N+1} \quad 
	\mbox{ on }\overline{Q}
\end{equation}
and 
\begin{equation}\label{eq:clou1}
	\underline{u}_N \leq \overline{u}_1 \leq \overline{u} \quad 
	\mbox{ on }\overline{Q}. 
\end{equation}
Let us first prove \eqref{eq:cluuN} by induction. 
Since $\underline{u}$ is a sub-solution of \eqref{eq:bddibp} and 
$F$ is nondecreasing in $u\in[c_1,c_2]$, 
we have 
\[
\left\{
\begin{aligned}
	&\partial_t (\underline{u}_1-\underline{u}_0)
	-\Delta (\underline{u}_1-\underline{u}_0) + c (\underline{u}_1-\underline{u}_0)
	\geq F(\underline{u}_0) - F(\underline{u}) = 0 && \mbox{ in } Q, \\
	&\underline{u}_1-\underline{u}_0 \geq \tilde g -\tilde g =0 && \mbox{ on }\Gamma. 
\end{aligned}
\right. 
\]
Then by the comparison principle (see for instance \cite[Corollary 2.5]{Libook}), 
$\underline{u}_0 \leq \underline{u}_1$ on $\overline{Q}$. 
Hence we see that \eqref{eq:cluuN} holds for $N=0$. 
If \eqref{eq:cluuN} also holds for $N$, then 
\[
\left\{
\begin{aligned}
	&\partial_t (\underline{u}_{N+1}-\underline{u})
	-\Delta (\underline{u}_{N+1}-\underline{u}) + c (\underline{u}_{N+1}-\underline{u})
	\geq F(\underline{u}_{N}) - F(\underline{u}) \geq 0,  \\
	&\partial_t (\overline{u}-\underline{u}_{N+1})
	-\Delta (\overline{u}-\underline{u}_{N+1}) + c (\overline{u}-\underline{u}_{N+1})
	\geq F(\overline{u}) - F(\underline{u}_{N}) \geq 0,  \\
	&\partial_t (\underline{u}_{N+1}-\underline{u}_{N})
	-\Delta (\underline{u}_{N+1}-\underline{u}_{N}) + c (\underline{u}_{N+1}-\underline{u}_{N})
	\geq F(\underline{u}_{N}) - F(\underline{u}_{N-1}) \geq 0,  \\
\end{aligned}
\right. 
\]
in $Q$, and 
\[
\left\{
\begin{aligned}
\	&\underline{u}_{N+1}-\underline{u} \geq \tilde g -\tilde g =0, \\
	&\overline{u}-\underline{u}_{N+1} \geq \tilde g -\tilde g =0, \\
	&\underline{u}_{N+1}-\underline{u}_N = \tilde g -\tilde g =0, 
\end{aligned}
\right. 
\]
on $\Gamma$. 
From the comparison principle, it follows that \eqref{eq:cluuN} holds for $N+1$, 
and \eqref{eq:cluuN} follows. 
Similarly, \eqref{eq:clou1} follows. 

By \eqref{eq:cluuN}, we see that 
\[
	u_\infty(x,t):= \lim_{N\to\infty} \underline{u}_N(x,t) 
\]
is well-defined for each $(x,t)\in \overline{Q}$. 
The interior regularity theory for parabolic equations and 
the Ascoli and Arzel\`a theorem deduce that $\underline{u}_N\to u_\infty$ 
in $C^{2,1}(\overline{Q'})$, where $Q'$ is any bounded cylinder satisfying $\overline{Q'}\subset Q$. 
Hence we obtain $u_\infty\in C^{2,1}(Q)$. Since 
\[
	\partial_t u_\infty -\Delta u_\infty + c u_\infty 
	= F(u_\infty)=
	c u_\infty + f(u_\infty) \quad \mbox{ in } Q, 
\]
$u_\infty$ satisfies the first equality in \eqref{eq:bddibp}. 

By \eqref{eq:cluuN} and \eqref{eq:clou1}, we have 
$\underline{u}_1\leq u_\infty\leq \overline{u}_1$ on $\overline{Q}$. 
Then, 
\[
	|u_\infty - \tilde g | 
	\leq \max\{ |\overline{u}_1 -\tilde g|, |\underline{u}_1 -\tilde g | \} 
	\quad \mbox{ on }\overline{Q}. 
\]
From $\underline{u}_1, \overline{u}_1 \in C(\overline{Q})$ and 
$\underline{u}_1= \overline{u}_1= \tilde g$ on $\Gamma$, 
it follows that $u_\infty\in C(\overline{Q})$ and that 
$u_\infty$ satisfies the second equality in \eqref{eq:bddibp}. 
The proof is complete. 
\end{proof}

\subsection{An existence theorem for unbounded domains}\label{subsec:unbddss}
We say that $\overline{u}$ ($\underline{u}$) 
is a \emph{super-solution} (\emph{sub-solution}) 
of \eqref{eq:unbddip} if $\overline{u}$ ($\underline{u}$) 
belongs to $C^{2,1}(Q_{(0,\infty)}\setminus M_{(0,\infty)}) 
\cap C(Q_{[0,\infty)}\setminus M_{[0,\infty)})$ 
and satisfies the inequalities
\[
\left\{
\begin{aligned}
	&\partial_t u-\Delta u \geq (\leq) f(u) && \mbox{ in } 
	Q_{(0,\infty)}\setminus M_{(0,\infty)}, \\
	&u(\cdot,0)\geq (\leq) u_0 && \mbox{ on }\mathbb{R}^n\setminus M_0. 
\end{aligned}
\right. 
\]
We prove the following theorem. 

\begin{theorem}\label{th:unbddss}
Let $n-m \geq 2$, $f\in C^1(\mathbb{R})$ and $u_0\in C(\mathbb{R}^n\setminus M_0)$. 
If \eqref{eq:unbddip} has a super-solution $\overline{u}$ and a sub-solution $\underline{u}$ 
satisfying $\underline{u}\leq \overline{u}$ on $Q_{[0,\infty)}\setminus M_{[0,\infty)}$, 
then \eqref{eq:unbddip} has a solution 
$u\in C^{2,1}(Q_{(0,\infty)}\setminus M_{(0,\infty)}) 
\cap C(Q_{[0,\infty)}\setminus M_{[0,\infty)})$ satisfying 
$\underline{u}\leq u\leq \overline{u}$ on $Q_{[0,\infty)}\setminus M_{[0,\infty)}$.  
\end{theorem}

In order to prove this theorem, we prepare sequences of sets. 
For $N\geq1$, set 
\[
\left\{
\begin{aligned}
	&A_N:=\max\left\{N,\max_{t\in[0,N]} d(0,M_t) +1\right\}, \\
	&Q_N :=\{ (x,t)\in Q_{(0,\infty)}; d(x,M_t)>1/N, |x|< A_N, 0<t<N \}, \\
	&\mathcal{B}_{t,N}:=\{ x\in \mathbb{R}^{n}; d(x,M_t)>1/N, |x|< A_N\} \quad (0\leq t\leq N), \\
	&B_{t,N} := \mathcal{B}_{t,N} \times\{t\}  \quad  (0\leq t\leq N), \\
	&S_N := \{ (x,t)\in Q_{(0,\infty)} ; x\in \partial \mathcal{B}_{t,N}, 0<t<N\}. 
\end{aligned}
\right. 
\]
See Figure \ref{fig:3}. 
\begin{figure}[tb]
\includegraphics[bb=30 311 564 532, clip, scale=0.67]{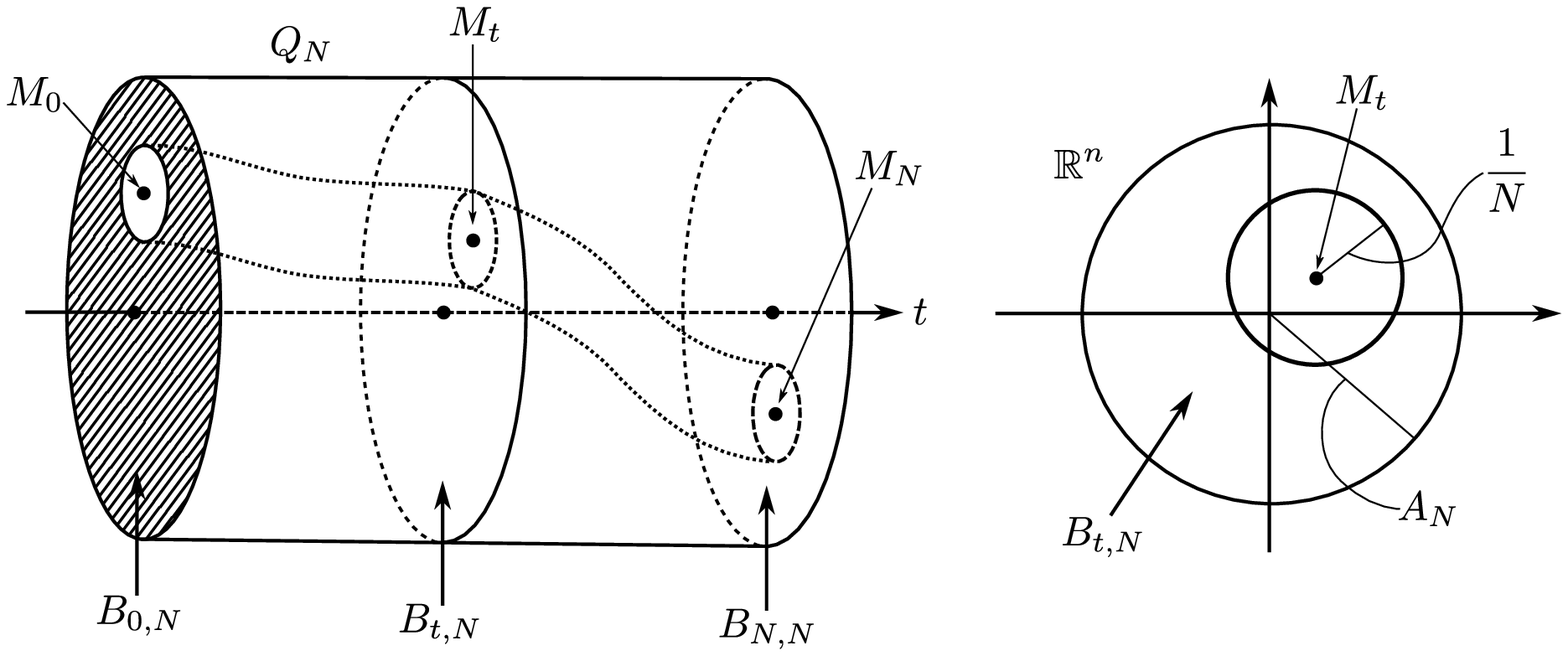}
\caption{$Q_{N}$ and $B_{t,N}$}
\label{fig:3}
\end{figure}
Then one can see that 
the corresponding conditions to (i)--(v) in Subsection \ref{subsec:bddss} 
hold provided that $N$ is large enough, 
since the codimension of $M_{t}$, $n-m$, is bigger than $1$. 
Actually, the following holds. 

\begin{lemma}\label{lem:12345}
There exists $N_0\geq1$ such that 
the following conditions {\rm(i')}--{\rm(v')} hold for $N\geq N_0$. 
\begin{itemize}
\item[(i')]
$S_N\neq \emptyset$ and $B_{t,N}\neq \emptyset$ for any $t\in[0,T]$. 
\item[(ii')]
$\partial Q_N = \overline{B}_{0,N} 
\cup \overline{S}_N \cup \overline{B}_{N,N}$. 
\item[(iii')]
$B_{t,N}$ is a smooth domain in $\mathbb{R}^n\times\{t\}$ for any $t\in (0,N)$. 
\item[(iv')]
$S_N$ is an $n$-dimensional smooth manifold without boundary. 
\item[(v')]
For any $P_0\in B_{0,N}$ and $P_T\in B_{N,N}$, there exists a simple continuous curve 
in $Q_N$ joining $P_0$ to $P_T$ along which the $t$-coordinate is increasing. 
\end{itemize}
\end{lemma}

At the end of this paper, we prove Theorem \ref{th:unbddss}.

\begin{proof}[Proof of Theorem \ref{th:unbddss}]
By Lemma \ref{lem:12345}, there exists $N_0\geq1$ such that 
Theorem \ref{th:bddss} is applicable to the problems 
\begin{equation}\label{eq:approxip}
\left\{
\begin{aligned}
	&\partial_t u-\Delta u = f (u) && \mbox{ in }Q_N, \\
	&u= \underline{u} && \mbox{ on }S_N,  \\
	&u(\cdot,0)= u_{0,N} && \mbox{ on }\overline{\mathcal{B}}_{0,N},  \\
\end{aligned}
\right. 
\end{equation}
and 
\begin{equation}\label{eq:approxipup}
\left\{
\begin{aligned}
	&\partial_t u-\Delta u = f (u) && \mbox{ in }Q_N, \\
	&u= \overline{u} && \mbox{ on }S_N,  \\
	&u(\cdot,0)= \tilde u_{0,N} && \mbox{ on }\overline{\mathcal{B}}_{0,N},  \\
\end{aligned}
\right. 
\end{equation}
for $N\geq N_0$ and given any functions 
$u_{0,N}\in C(\overline{\mathcal{B}}_{0,N} )$ 
with $u_{0,N}=\underline{u}$ on $\partial \mathcal{B}_{0,N}$ and 
$\tilde u_{0,N}\in C(\overline{\mathcal{B}}_{0,N} )$ 
with $\tilde u_{0,N}=\overline{u}$ on $\partial \mathcal{B}_{0,N}$. 
We choose sequences of initial data 
$\{\underline{u}_{0,N}\}_{N\geq N_0}$ 
and $\{\overline{u}_{0,N}\}_{N\geq N_0}$ 
as follows. 
Let $\underline{u}_{0,N_0}\in C(\overline{\mathcal{B}}_{0,N_0})$ 
and $\overline{u}_{0,N_0}\in C(\overline{\mathcal{B}}_{0,N_0} )$ 
satisfy 
that 
\[
	\underline{u}(\cdot,0)\leq \underline{u}_{0,N_0}\leq u_0
	\leq \overline{u}_{0,N_0}\leq \overline{u}(\cdot,0)
	\quad \mbox{ on }
	\overline{\mathcal{B}}_{0,N_0} 
\]
and that 
$\underline{u}_{0,N_0} = \underline{u}(\cdot,0)$ 
on $\partial \mathcal{B}_{0,N_0}$ and 
$\overline{u}_{0,N_0} = \overline{u}(\cdot,0)$ 
on $\partial \mathcal{B}_{0,N_0}$. 
For $N\geq N_0+1$, 
by $\overline{\mathcal{B}}_{0,N-1}\subset \mathcal{B}_{0,N}$, 
we inductively define a sequence of cut-off functions 
$\eta_N\in C^\infty_0(\mathbb{R}^n)$ such that 
\[
\left\{
\begin{aligned}
	&\eta_N = 1 &&\mbox { on } \overline{\mathcal{B}}_{0,N-1},   \\
	&\eta_N = 0 &&\mbox { on } \mathbb{R}^n\setminus \overline{\mathcal{B}}_{0,N},   \\
	&0\leq \eta_N \leq 1
	&&\mbox { on } \overline{\mathcal{B}}_{0,N} \setminus \mathcal{B}_{0,N-1}. 
\end{aligned}
\right.
\]
Then, $\eta_N\leq \eta_{N+1}$ on $\mathcal{B}_{0,N}$.  Set 
\[
\begin{aligned}
	& \underline{u}_{0,N} := 
	u_0 \eta_N + \underline{u} (\cdot,0) (1-\eta_N)
	&&\mbox{ for }N\geq N_0+1,  \\
	& \overline{u}_{0,N} := 
	u_0 \eta_N + \overline{u} (\cdot,0) (1-\eta_N)
	&&\mbox{ for }N\geq N_0+1. 
\end{aligned}
\]
\begin{figure}[tb]
\includegraphics[bb=31 283 566 559, clip, scale=0.67]{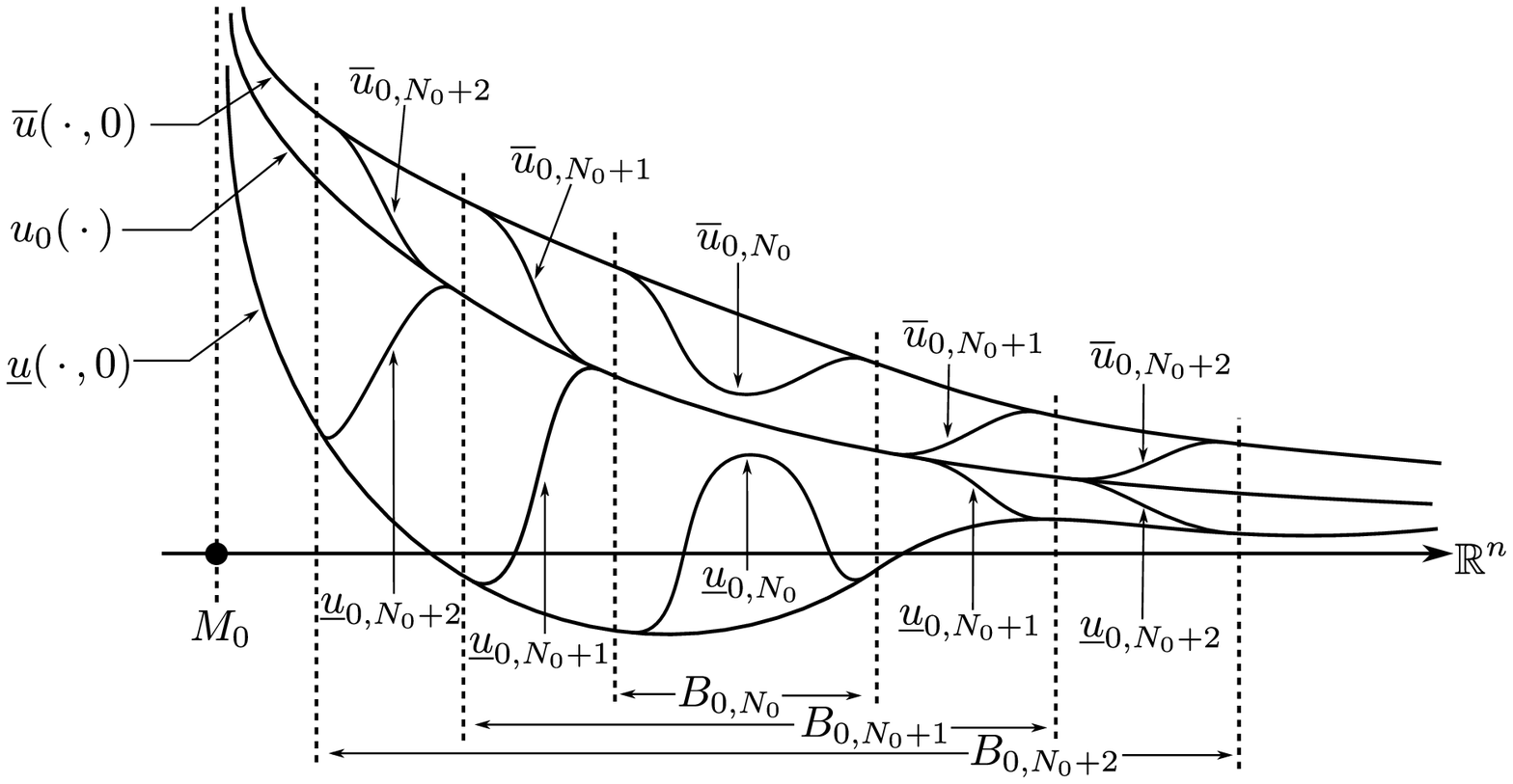}
\caption{$\underline{u}_{0,N}$ and $\overline{u}_{0,N}$}
\label{fig:4}
\end{figure}
See Figure \ref{fig:4}. 
Note that, for $N\geq N_0+1$, we have  
\begin{equation}\label{eq:iniuN}
\begin{aligned}
	&\underline{u}_{0,N} = \underline{u}(\cdot,0), \quad 
	\overline{u}_{0,N} = \overline{u}(\cdot,0) 
	\quad \mbox{ on } \partial \mathcal{B}_{0,N},  \\
	&\underline{u}(\cdot,0) \leq \underline{u}_{0,N} \leq 
	\underline{u}_{0,N+1} \leq u_0 
	\leq \overline{u}_{0,N+1} \leq \overline{u}_{0,N} \leq \overline{u}(\cdot,0)
	\quad \mbox{ in } \overline{\mathcal{B}}_{0,N}. 
\end{aligned}
\end{equation}
 
Let $\underline{u}_N, \overline{u}_N \in C^{2,1}(Q_N)\cap C(\overline{Q}_N)$ 
be solutions of \eqref{eq:approxip} and \eqref{eq:approxipup} with 
$u_{0,N}=\underline{u}_{0,N}$ and 
$\tilde u_{0,N}=\overline{u}_{0,N}$, respectively. 
Remark that the existence of $\underline{u}_N$ and $\overline{u}_N$ and 
\begin{equation}\label{eq:luuN}
	\underline{u}\leq \underline{u}_N \leq \overline{u}, \quad 
	\underline{u}\leq \overline{u}_N \leq \overline{u}
	\quad\mbox{ on } \overline{Q}_N  
\end{equation} 
for $N\geq N_0$ follow from Theorem \ref{th:bddss}.  

Let us prove the inequalities 
\begin{equation}\label{eq:orduN}
	\underline{u}_{N} \leq \underline{u}_{N+1}, \quad 
	\overline{u}_{N+1} \leq \overline{u}_{N} \quad 
	\mbox{ on }\overline{Q}_N 
\end{equation}
for $N\geq N_0$. 
By \eqref{eq:luuN} and \eqref{eq:iniuN}, we have 
\[
\left\{
\begin{aligned}
	&\partial_t (\underline{u}_{N+1}-\underline{u}_N) 
	-\Delta (\underline{u}_{N+1}-\underline{u}_N) 
	+h(x,t) (\underline{u}_{N+1}-\underline{u}_N)=0
	&&\mbox{ in }Q_N, \\
	&\underline{u}_{N+1}-\underline{u}_N 
	\geq \underline{u} -\underline{u} =0 &&\mbox{ on }S_N, \\
	&\underline{u}_{N+1}(\cdot,0)-\underline{u}_N(\cdot,0) = 
	\underline{u}_{0,N+1}-\underline{u}_{0,N}\geq 0
	&&\mbox{ on }\overline{\mathcal{B}}_{0,N}, 
\end{aligned}
\right.
\]
where $h$ is given by 
\[
	h(x,t) := -\int_0^1 f'(\theta u_{N+1}(x,t) + (1-\theta) u_N(x,t) ) d\theta. 
\]
By \eqref{eq:luuN}, we have 
$\max_{\overline{Q}_N} |h(x,t)| \leq
\max \{ |f'(u)| ; \min_{\overline{Q}_N} \underline{u} \leq u \leq 
\max_{\overline{Q}_N} \overline{u}\}$. 
Therefore, the function $h$ is bounded on $\overline{Q}_N$. 
Hence by the comparison principle (see for instance \cite[Corollary 2.5]{Libook}), 
the first inequality in \eqref{eq:orduN} holds. 
Similarly, the second inequality holds. 

Note that  $\{\overline{Q}_N\}_{N\geq N_0}$ 
is an approximate sequence of $Q_{[0,\infty)}\setminus M_{[0,\infty)}$. 
Therefore \eqref{eq:luuN} and \eqref{eq:orduN} imply that 
\[
	\underline{u}_\infty(x,t):= \lim_{N\to\infty} \underline{u}_N(x,t) 
\]
is well-defined for each $(x,t)\in Q_{[0,\infty)}\setminus M_{[0,\infty)}$. 
By the same argument as in the proof of Theorem \ref{th:bddss} and 
the diagonalization argument, 
$\underline{u}_N\to \underline{u}_\infty$ 
in $C^{2,1}(\overline{Q'})$ for any bounded cylinder $Q'$ satisfying 
$\overline{Q'}\subset Q_{(0,\infty)}\setminus M_{(0,\infty)}$. 
Thus, 
$\underline{u}_\infty\in C^{2,1}(Q_{(0,\infty)}\setminus M_{(0,\infty)})$ and 
$\underline{u}_\infty$ satisfies the first equality in \eqref{eq:unbddip}. 

We check $\underline{u}_\infty \in C(Q_{[0,\infty)}\setminus M_{[0,\infty)})$. 
Fix $x_0 \in \mathbb{R}^n\setminus M_0$. Let 
$\Omega\ni x_0$ and $T>0$ satisfy that 
$\Omega$ is a bounded domain and 
$\overline{\Omega}\times[0,T] \subset Q_{(0,\infty)}\setminus M_{(0,\infty)}$. 
Then, by the choice of $\underline{u}_{0,N}$ and $\overline{u}_{0,N}$, 
there exists $N_1\geq N_0+1$ such that 
\begin{equation}\label{eq:x0nbh}
	(x_0,0)\in \Omega\times[0,T) \subset \overline{Q}_{0,N},  \quad 
	\underline{u}_{0,N}=\overline{u}_{0,N}=u_0 \quad 
	\mbox{ in }\Omega 
\end{equation}
for any $N\geq N_1$. From \eqref{eq:orduN}, it follows that 
$\underline{u}_{N_1} \leq \underline{u}_\infty
\leq \overline{u}_{N_1}$  on $\overline{Q}_N$, and so
\[
	|\underline{u}_\infty - u_0 | \leq 
	\max\{ |\overline{u}_{N_1} - u_0|, |\underline{u}_{N_1} - u_0 | \} 
	\quad 
	\mbox{ on }\Omega \times[0,T). 
\]
This together with \eqref{eq:x0nbh} and 
$\underline{u}_{N_1}, \overline{u}_{N_1}\in C(\overline{Q}_{N_1})$ gives 
$\underline{u}_\infty \in C(Q_{[0,\infty)}\setminus M_{[0,\infty)})$ and 
$\underline{u}_\infty(x_0,0)=u_0(x_0)$. 
Thus, $u_\infty$ satisfies the second equality in \eqref{eq:bddibp}. 
The proof is complete. 
\end{proof}

\end{document}